\documentclass[11pt,a4paper]{article}

\usepackage{amsmath}
\usepackage{amssymb}
\usepackage{amsfonts}
\usepackage{amsthm}
\usepackage{mathtools}
\usepackage{mathrsfs}

\usepackage{graphicx}
\usepackage{multirow}
\usepackage{booktabs}

\usepackage[title]{appendix}

\usepackage{algorithm}
\usepackage{algorithmicx}
\usepackage{algpseudocode}

\usepackage{xcolor}
\usepackage{textcomp}
\usepackage{manyfoot}
\usepackage{listings}

\usepackage[a4paper,margin=1in]{geometry}

\usepackage[
colorlinks=true,
linkcolor=blue,
citecolor=blue,
urlcolor=blue
]{hyperref}

\theoremstyle{plain}

\newtheorem{theorem}{Theorem}[section]

\newtheorem{proposition}[theorem]{Proposition}
\newtheorem{lemma}[theorem]{Lemma}

\theoremstyle{definition}

\newtheorem{definition}[theorem]{Definition}

\theoremstyle{remark}

\newtheorem{remark}[theorem]{Remark}

\numberwithin{equation}{section}

\newcommand{\T}{\mathbb{T}}
\newcommand{\Z}{\mathbb{Z}}
\newcommand{\N}{\mathbb{N}}
\newcommand{\E}{\mathbb{E}}
\usepackage{enumitem}
\newcommand{\dd}{\,\mathrm{d}}

\newcommand{\Dt}{\Delta t}
\newcommand{\PN}{\mathcal{P}_N}

\begin{document}
	
	
	\title{
		Convergence of a fully discrete approximation for the
		stochastic 2D Euler equations with Kraichnan transport noise
	}
	
	
	\author{
		Abhishek Chaudhary\textsuperscript{1}\thanks{
			Corresponding author. Email: \href{mailto:chaudhary@na.uni-tuebingen.de}
			{chaudhary@na.uni-tuebingen.de}
		}
		\and
		Ujjwal Koley\textsuperscript{2}
		\thanks{
			Email: \href{mailto:ujjwal@tifrb.res.in}
			{ujjwal@tifrb.res.in}
		}
		\and
		Andreas Prohl\textsuperscript{1}
		\thanks{
			Email: \href{mailto:prohl@na.uni-tuebingen.de}
			{prohl@na.uni-tuebingen.de}
		}
	}
	
	
	\date{}
	
	\maketitle
	
	\begin{center}
		\small
		\textsuperscript{1}Mathematisches Institut, Universit\"at T\"ubingen,
		Auf der Morgenstelle 10, D-72076 T\"ubingen, Germany
		
		\vspace{0.3em}
		
		\textsuperscript{2}TIFR Centre for Applicable Mathematics,
		Tata Institute of Fundamental Research,
		Bengaluru, Karnataka 560065, India
	\end{center}
	
	\abstract{We study a fully discrete approximation theory for the two-dimensional
		incompressible Euler equations on $\mathbb{T}^{2}$, in vorticity form, driven by Kraichnan-type transport noise and with $\mathbb{H}^{-1}(\T^2)$-valued initial vorticity. The
		main new ingredient is a finite-dimensional coercivity estimate for the torus
		Kraichnan model, which captures at the discrete level the regularizing mechanism
		induced by the rough transport noise. Together with a new compactness strategy based
		on continuous equation-based interpolants and localized discrete estimates, this
		coercivity estimate gives tightness of the approximations and shows that subsequential limit is a weak martingale solution of the Euler equations.}

	
	\medskip
	
	
	
	\maketitle
	
	\section{Introduction}\label{sec1}
	The 2D incompressible Euler equations occupy a central position in
	fluid dynamics and in the mathematical theory of inviscid flows. In vorticity form,
	the 2D incompressible Euler equations read
	\[
	\partial_t \omega + (\mathbf{u}\cdot\nabla)\omega = 0,
	\qquad
	\mathbf{u} = \nabla^\perp(-\Delta)^{-1}\omega,
	\]
	where \(\omega=\operatorname{curl}\mathbf{u}\) is the scalar vorticity. Classical well-posedness in the deterministic setting is well understood for
	bounded vorticity through the work of Yudovich \cite{Yudovich1963}. In this
	regime, the associated velocity field is log-Lipschitz, and the Euler equations
	admits a unique global weak solution in the Yudovich class. For rougher data,
	however, the theory is far more subtle, and different notions of weak solution
	become relevant. DiPerna and Majda studied concentration effects arising
	from regularizations of 2D incompressible flow and developed compactness mechanisms
	for passing to weak limits \cite{DiPernaMajda1987,DiPernaMajda1988}. Delort proved
	existence of vortex-sheet solutions for initial vorticities given by bounded Radon
	measures with a distinguished sign \cite{Delort1991}. Schochet clarified the weak
	vorticity formulation and the role of concentration-cancellation in the 2D Euler
	equations \cite{Schochet1995}. Wiedemann constructed weak solutions with very low
	regularity, illustrating that weak formulations extend beyond the classical
	vorticity-based settings and that the precise formulation of the weak solution
	concept becomes important \cite{Wiedemann2011}. These works show that, for rough
	vorticity, compactness and the identification of weak limits are central issues in
	the analysis of the 2D Euler equations.
	
	Similar difficulties also appear in the deterministic numerical analysis of the Euler equations. For
	sufficiently regular solutions, Fourier spectral approximations of the incompressible
	Euler equations are stable and converge spectrally \cite{BardosTadmor2015}. Lopes
	Filho, Nussenzveig Lopes and Tadmor studied approximate solutions of the 2D
	incompressible Euler equations in the absence of concentration effects
	\cite{LopesFilhoNussenzveigLopesTadmor2001}. Lanthaler and Mishra proved convergence
	of a spectral-viscosity method for the 2D incompressible Euler equations with rough
	initial data, including vorticities in the Delort class
	\cite{LanthalerMishra2019}. Related computations for rough 2D Euler flows,
	especially vortex-sheet type flows, have also motivated measure-valued and
	statistical solution frameworks for incompressible Euler equations~\cite{LanthalerMishra2015,LanthalerMishraParesPulido2021}. In the
	present paper, we revisit this approximation problem in a stochastic setting:
	the approximations are generated by a fully discrete implementable scheme, and
	the subsequential limit is a weak martingale solution.
	
	In recent years there has been growing interest in stochastic perturbations of fluid
	equations, and in particular in transport-type noises that preserve the geometric
	structure of the Euler dynamics. Among these, Kraichnan-type noises are especially
	important. They arise naturally in turbulence modeling and in scaling-limit questions,
	and they have been used to capture effective dissipation and regularization phenomena
	generated by unresolved small scales; see,
	\cite{Kraichnan1968,Kraichnan1994,Galeati2020,FlandoliLuo2020,
		FlandoliGaleatiLuo2024,FlandoliPappalettera2021,FlandoliHuang2023} for more details. In the present work we consider the
	2D Euler equations on \(\mathbb{T}^2\) perturbed by Kraichnan transport noise,
	namely
	\begin{equation}\label{eq:intro_strat}
		\mathrm{d}\omega + (\mathbf{u}\cdot \nabla)\omega\,\mathrm{d}t
		+ \sum_{k\in\mathcal I}(\boldsymbol{\sigma}_k\cdot \nabla)\omega \circ \mathrm{d} {\bf W}_k(t) = 0,
		\qquad
		\mathbf{u}=\nabla^\perp(-\Delta)^{-1}\omega,
	\end{equation}
	with divergence-free Fourier modes \(\boldsymbol{\sigma}_k\) and Brownian motions ${\bf W}_k$ which are defined in Section~\ref{Section 2}. 
	
	Equation~\eqref{eq:intro_strat} can be viewed as an effective large-scale
	vorticity equation in which the resolved Euler velocity \(\mathbf{u}\) is coupled
	with unresolved incompressible transport fluctuations. The vector fields
	\(\boldsymbol{\sigma}_k\) describe the spatial structure of these fluctuations, and
	the Stratonovich formulation preserves the geometric character of transport by a
	random velocity field. Thus the noise is not an additive external force acting on
	the vorticity; rather, it enters through the transport mechanism and models the
	influence of unresolved scales on the resolved inviscid dynamics. This interpretation
	is consistent with stochastic model reduction and with scaling-limit descriptions of
	turbulent transport; see, for instance,
	\cite{Kraichnan1968,Kraichnan1994,FlandoliPappalettera2021,FlandoliHuang2023,
		FlandoliGaleatiLuo2024}.
	
	\subsection{Our aims and contributions in this work}
	The rough Kraichnan regime is by now known to produce regularization effects in
	fluid and transport equations. On \(\mathbb R^2\), Coghi and Maurelli \cite{CoghiMaurelli2026} studied
	the stochastic 2D Euler equations~ \eqref{eq:intro_strat} with unbounded vorticity and rough Kraichnan
	transport noise. They proved  existence of a weak martingale solution to \eqref{eq:intro_strat} for
	\(\dot{\mathbb H}^{-1}(\mathbb R^2)\)-initial vorticity and showed that the
	constructed solutions enjoy an additional
	\(\mathbb L^2([0,T];\mathbb H^{-\alpha}(\mathbb R^2))\)-regularity with $\alpha\in (0,1/2)$. For
	\(\mathbb L^1(\mathbb R^2)\cap \mathbb L^p(\mathbb R^2)\)-initial vorticities,
	they also proved pathwise uniqueness in the range \(p>3/2\), under suitable
	restrictions on the Kraichnan regularity index $\alpha$. This uniqueness result was later
	extended by Jiao and Luo \cite{JiaoLuo2024} to all \(p>1\), thereby removing the restriction
	\(p>3/2\). For stochastic 2D Euler
	equations with transport noise, existence and uniqueness in the bounded-vorticity
	regime were obtained by Brze\'zniak, Flandoli and Maurelli
	\cite{BrzezniakFlandoliMaurelli2016}, while a stochastic analogue of Delort-type
	existence for positive \(\mathbb{H}^{-1}\)-vorticity was proved by Brze\'zniak and Maurelli
	\cite{BrzezniakMaurelli2019}. Closely related anomalous regularization mechanisms
	have also been studied for Kraichnan passive scalars and scaling limits of stochastic
	transport equations
	\cite{Galeati2020,FlandoliGaleatiLuo2024,GaleatiGrottoMaurelli2024}. However, these
	results are primarily analytical and continuous in nature. A fully discrete
	approximation theory for the Euler equations~\eqref{eq:intro_strat}, capable of
	producing weak martingale solutions from an implementable numerical scheme, has not
	been available.
	
	The purpose of the present paper is to develop such a theory at the level of a fully
	discrete approximation. However, on \(\mathbb{R}^{2}\), the absence of an explicit and directly
	implementable expression for the vector fields \(\boldsymbol{\sigma}_{k}\) makes numerical approximation highly challenging and, at present, seemingly out of reach. For this reason and motivated by~\cite[Sec. IV]{FlandoliHuang2023}, we work instead on the 2D torus, where explicit implementable formulas for the vector fields \(\boldsymbol{\sigma}_{k}\) are available. This choice allows us to keep the rough Kraichnan structure while obtaining a genuinely finite-dimensional and implementable approximation scheme.
	
	More precisely, we introduce and analyze the spectral scheme~\eqref{eq:scheme_keep}, obtained by combining Fourier truncation with parameter $N$, a finite noise cutoff with parameter $K$, which combined a mixed implicit--explicit time discretization with step size $\Delta t$. In this scheme, the transport term is evaluated at the mixed time levels \(\bigl(\mathbf{u}_{N}^{\,n+1},\omega_{N}^{\,n}\bigr)\), while the It\^{o} correction is represented by the implicit Laplacian term \(\tfrac{c_{K}}{2}\,\Delta\omega_{N}^{\,n+1}\). This structure is not accidental: when the scheme is tested in the \(\mathbb{H}^{-1}(\mathbb{T}^2)\)-energy, the mixed-time discretization yields the exact cancellation described in Remark~\ref{rem:mixed_time_conv}, whereas the implicit Laplacian provides the coercive contribution needed for the stability estimate. In this way, the scheme is consistent both with the transport structure of the Euler dynamics and with the regularizing mechanism induced by the Kraichnan noise; see Remark~\ref{rem:necessity_implicit_euler} for more details.

	The first main contribution of the paper is a finite-dimensional uniform coercivity estimate \eqref{eq:finite_dimensional_coercivity} for the torus model, under the admissible scale-separation condition \[ K\ge \sqrt{2}\,\gamma N, \qquad \gamma>1+\sqrt{2}. \] This condition ensures that the truncated Kraichnan noise contains sufficiently many modes above the resolved spatial frequencies; these high modes are precisely responsible for the discrete negative-Sobolev dissipation, as explained in Remark~\ref{rem:regularization_high_modes}.
	
	In the Euclidean setting, coercivity is recovered through the covariance structure of
	the Kraichnan field and its singular behavior near the diagonal; see, in particular,
	the energy analysis in \cite{CoghiMaurelli2026}. On the torus, however, the problem
	is of a different nature. Here the noise is represented by explicit trigonometric
	Fourier modes, and the relevant quantity is the discrete quadratic form
	\[
	\mathfrak D_{N,K}(\omega)
	=
	\sum_{k\in\mathcal I^K}
	\bigl\|(-\Delta)^{-1/2}\mathcal P_N[(\boldsymbol{\sigma}_k\cdot \nabla)\omega]\bigr\|_{\mathbb{L}^2(\mathbb{T}^2)}^2,
	\]
	which is introduced in Definition~\ref{def:DLK}. The derivation of a finite-dimensional
	coercivity estimate~\eqref{eq:finite_dimensional_coercivity} on \(\mathbb{T}^2\) is, to the best of our
	knowledge, new. The proof is specific to the torus geometry and combines three novel
	ingredients: the uniform Riemann-sum estimate of Lemma~\ref{Riemann-sum}, the
	positivity of the continuum defect established in
	Lemma~\ref{lem:positivity_j_alpha_4}, and the discrete lattice analysis of
	Lemma~\ref{lem:finite_discrete_symbol_gap}. Together, these yield the lower bound
	\eqref{eq:finite_symbol_gap}, and hence the coercivity estimate
	\eqref{eq:finite_dimensional_coercivity}. This estimate is the key input in the
	discrete \(\mathbb{H}^{-1}(\mathbb{T}^2)\)-energy estimates~\eqref{thm:energy_keep_full}
	and~\eqref{thm:strong_energy_keep}. It is also the point at which the torus geometry
	and the finite-dimensional noise approximation enter in an essential way.
	
	The second main contribution is a numerical approximation of weak martingale solutions to the Euler equations~\eqref{eq:intro_strat}. More precisely, we prove that, along subsequences, the fully
	discrete scheme \eqref{eq:scheme_keep} converges to a weak martingale solution of
	the limiting It\^o Euler equations~\eqref{eq:limit_Ito_equation_conv} in the sense of
	Definition~\ref{def:weak_martingale_solution_torus_conv}; see~Theorem~\ref{thm:convergence_fully_discrete_martingale}. This is a subsequential
	convergence theorem, which is the natural form of compactness-based convergence in
	the absence of a full uniqueness theory for weak martingale solutions at this level
	of regularity. This convergence result is new not only for the Euler equations \eqref{eq:intro_strat} but also for the proof strategy. The
	abstract framework of Ondrej\'at, Prohl and Walkington
	\cite{OndrejatProhlWalkington2023} is formulated on discontinuous Skorokhod-type
	path spaces for piecewise-constant interpolants and relies on global uniform bounds
	of the type assumed in their general convergence theorem (see, \cite[Thm.~3.2]{OndrejatProhlWalkington2023}). In the present setting,
	however, this framework is not directly applicable; see Remark~\ref{remark6.12} for more relevant details. What is available here is a
	localized control obtained through the discrete stopping times introduced in
	\eqref{eq:stopping_time_conv}; see in particular
	Lemma~\ref{lem:stop_high_prob_conv},
	Lemma~\ref{lem:localized_drift_bound_conv}, and
	Lemma~\ref{lem:localized_holder_conv}. For this reason, the compactness argument
	must be organized in a different way.

	The convergence proof therefore proceeds through a continuous equation-based
	interpolant,
	\[
	\omega_j^\sharp(t)
	=
	\omega_j^0+\int_0^t F_j(s)\,\mathrm{d}s+\mathcal M_j(t)\qquad\forall\,t\in[0,T],
	\]
	where \(F_j\) and \(\mathcal M_j\) are defined in
	\eqref{eq:drift_j_def_conv} and \eqref{eq:martingale_j_def_conv}. By construction,
	\(\omega_j^\sharp\) coincides with the discrete approximation at the grid points,
	but is continuous in time and already has the temporal regularity expected of the
	limit process. The compactness analysis is carried out on the canonical path space
	\(\mathfrak X\) introduced in \eqref{eq:full_path_space_conv}, whose first
	coordinate is
	\[
	\mathcal X_\omega
	=
	C([0,T];\mathbb{H}^{-5}(\mathbb{T}^2)).
	\]
	The localized H\"older bound of Lemma~\ref{lem:localized_holder_conv}, proved from
	the stopping-time localization \eqref{eq:stopping_time_conv}, yields tightness of
	\(\omega_j^\sharp\) in \(C([0,T];\mathbb{H}^{-5}(\mathbb{T}^2))\). This is then combined with
	Proposition~\ref{prop:tightness_main_conv}, which gives tightness of the full family
	\[
	\bigl(
	\omega_j^\sharp,\underline\omega_j,\overline\omega_j,({\bf W}_k)_{k\in\mathcal I}
	\bigr)
	\]
	on \(\mathfrak X\). The piecewise-constant interpolants \(\underline\omega_j\) and
	\(\overline\omega_j\) still play an important role, but only at the level of the
	discrete weak formulation. They are needed to express the nonlinear and stochastic
	terms, but they are shown to be asymptotically equivalent to the continuous
	interpolant in the topology relevant for the limit passage. More precisely,
	Lemmas~\ref{lem:sharp_under_close_conv} and~\ref{lem:bar_under_close_conv} show that
	\[
	\omega_j^\sharp-\underline\omega_j\to0,
	\qquad
	\overline\omega_j-\underline\omega_j\to0
	\quad\text{in probability in }\mathbb{L}^2([0,T];\mathbb{H}^{-4}(\mathbb{T}^2)).
	\]
	After the Skorokhod--Jakubowski representation theorem,
	Proposition~\ref{prop:skorokhod_jakubowski_conv}, all three approximating objects
	converge to the same limit process \(\widetilde\omega\). In particular,
	\(\widetilde\omega\) is a continuous \(\mathbb{H}^{-5}\)-valued process, and by
	Proposition~\ref{prop:skorokhod_jakubowski_conv} it is adapted to the usual augmentation filtration
	generated by \(\big(\widetilde{\omega}, (\widetilde{\bf W}_k)_{k\in\mathcal I}\big)\). The limit identification is then carried out directly in the discrete weak
	formulation \eqref{eq:discrete_weak_form_conv}, which yields the limiting weak
	formulation \eqref{eq:weak_formimit_conv} associated with
	\eqref{eq:limit_Ito_equation_conv}.
	
	In addition, the corresponding discrete velocities converge strongly along the same
	subsequences to the velocity associated with the limiting martingale solution; see~\eqref{strongconvergene}. This strong convergence is important for the
	identification of the nonlinear transport term. Although the vorticities are compact
	only in weak negative Sobolev topologies at the level dictated by the energy estimates,
	the Biot--Savart law gives additional compactness for the velocities, which allows one
	to pass to the limit in the nonlinearity of Euler equations in the discrete weak formulation.
	
	We emphasize that the two novelties of the paper are closely related. The torus
	coercivity estimate~\eqref{eq:finite_dimensional_coercivity} provides the discrete counterpart of the anomalous
	\(\mathbb{H}^{-\alpha}\)-regularization induced by Kraichnan noise, while the convergence results show that this regularization survives at the level of fully discrete
	approximations and is sufficiently strong to produce weak martingale solutions in
	the limit. Taken together, these results provide a numerical existence theory for the
	Euler equations~\eqref{eq:intro_strat} on \(\mathbb{T}^2\), based on an implementable
	spectral discretization, on a novel coercivity argument in the torus setting,
	and on a different compactness method tailored to the structure of the fully discrete
	scheme~\eqref{eq:scheme_keep}. Thus the scheme is not only a fully discrete approximation procedure; it also gives
	a constructive way to obtain rough weak martingale solutions for
	2D Euler equations~\eqref{eq:intro_strat} on $\T^2$ with \(\mathbb{H}^{-1}(\T^2)\)-valued initial vorticity.
	
	The paper is organized as follows. We first introduce the Kraichnan noise and
	the fully discrete spectral scheme \eqref{eq:scheme_keep}. We then establish the
	coercivity mechanism, combining the uniform Riemann-sum approximation of
	Lemma~\ref{Riemann-sum}, the positivity of the continuum defect in
	Lemma~\ref{lem:positivity_j_alpha_4}, and the discrete lattice estimate of
	Lemma~\ref{lem:finite_discrete_symbol_gap} in order to derive the finite-dimensional
	coercivity bound \eqref{eq:finite_dimensional_coercivity}. After that, we derive the
	discrete energy estimates and stability properties of the scheme. Finally, we prove
	subsequential convergence of the fully discrete approximations to a weak martingale
	solution by establishing tightness on \(\mathfrak X\), applying the
	Skorokhod--Jakubowski representation theorem, and passing to the limit in the
	discrete weak formulation \eqref{eq:discrete_weak_form_conv}, thereby obtaining
	the limiting martingale identity \eqref{eq:weak_formimit_conv}.
	\section{Mathematical setup}\label{Section 2}
	
	Throughout the paper, we work on the standard two-dimensional torus
	\[
	\T^2 := (\mathbb R/2\pi\mathbb Z)^2 .
	\]
	We fix \(0<\alpha<1\) and define the spectral density
	\begin{equation}\label{eq:q_def}
		q(\boldsymbol{\xi}) := \langle \boldsymbol{\xi}\rangle^{-(2+2\alpha)}=(1+|\boldsymbol{\xi}|^2)^{-1-\alpha},
		\qquad \boldsymbol{\xi}\in\mathbb R^2.
	\end{equation}
	For each \(\mathbf{m}\in\mathbb Z^2\setminus\{{\bf 0}\}\), we define the unit vector
	\begin{align}\label{today07}
		\mathbf{e}_{\mathbf m}:=\frac{\mathbf m^\perp}{|\mathbf m|},
	\end{align}
	and the two real divergence-free noise modes
	\begin{equation}\label{eq:noise_ modes}
		\boldsymbol{\sigma}_{\mathbf m,c}(\mathbf x)
		:=
		\sqrt{2q(\mathbf m)}\,\mathbf e_{\mathbf m}\cos(\mathbf m\cdot \mathbf x),
		\qquad
		\boldsymbol{\sigma}_{\mathbf m,s}(\mathbf x)
		:=
		\sqrt{2q(\mathbf m)}\,\mathbf e_{\mathbf m}\sin(\mathbf m\cdot \mathbf x),
		\qquad \mathbf x\in\T^2.
	\end{equation}
	We collect these two families into the index set
	\[
	\mathcal I := (\mathbb Z^2\setminus\{{\bf 0}\})\times\{c,s\},
	\]
	and write
	\[
	\boldsymbol{\sigma}_{(\mathbf m,\ell)} :=
	\begin{cases}
		\boldsymbol{\sigma}_{\mathbf m,c}, & \ell=c,\\[1mm]
		\boldsymbol{\sigma}_{\mathbf m,s}, & \ell=s.
	\end{cases}
	\]
	Accordingly, the driving Brownian motions are denoted by
	\[
	\{ {\bf W}_k\}_{k\in\mathcal I}
	=
	\{W_{\mathbf m,c},W_{\mathbf m,s}\}_{\mathbf m\in\mathbb Z^2\setminus\{{\bf 0}\}},
	\]
	where all components are independent standard Brownian motions. With this notation, the stochastic transport term in \eqref{eq:intro_strat} may
	be written either in the compact form
	\[
	\sum_{k\in\mathcal I} (\boldsymbol{\sigma}_k\cdot \nabla)\omega \circ \dd{\bf W}_k(t),
	\]
	or, equivalently, in the expanded form
	\[
	\sum_{\mathbf m\in\mathbb Z^2\setminus\{{\bf 0}\}}
	\Big[
	(\boldsymbol{\sigma}_{\mathbf m,c}\cdot \nabla)\omega \circ \dd W_{\mathbf m,c}(t)
	+
	(\boldsymbol{\sigma}_{\mathbf m,s}\cdot \nabla)\omega \circ \dd W_{\mathbf m,s}(t)
	\Big].
	\]
	The same convention will be used throughout the paper.
	\subsection{Functional setting}
	\label{subsec:functional_setting}
	For \(\mathbf m\in\mathbb Z^2\), we write
	\[
	e_{\mathbf m}(\mathbf x):=e^{i\mathbf m\cdot \mathbf x},
	\qquad
	\widehat f_{\mathbf m}
	:=
	\frac{1}{(2\pi)^2}\int_{\T^2}
	f(\mathbf x)e^{-i\mathbf m\cdot\mathbf x}\,d\mathbf x .
	\]
	For \(s\in\mathbb R\), we define the Sobolev space
	\[
	\mathbb H^s(\T^2)
	:=
	\left\{
	f\in\mathcal D'(\T^2):\,
	\|f\|_{\mathbb H^s}^2
	:=
	\sum_{\mathbf m\in\mathbb Z^2}
	|\mathbf m|^{2s}|\widehat f_{\mathbf m}|^2
	<\infty
	\right\}.
	\]
	In particular,
	\[
	\mathbb L^2(\T^2):=\mathbb H^0(\T^2).
	\]
	For \(s>0\), the space \(\mathbb H^{-s}(\T^2)\) is identified with the dual of
	\(\mathbb H^s(\T^2)\) with respect to the \(\mathbb L^2\)-pairing. If
	\(f\in\mathbb H^{-s}(\T^2)\) and \(g\in\mathbb H^s(\T^2)\), we write
	\[
	\langle f,g\rangle
	:=
	\sum_{\mathbf m\in\mathbb Z^2}
	\widehat f_{\mathbf m}\,\overline{\widehat g_{\mathbf m}},
	\]
	whenever this expression is meaningful, and otherwise by duality.
	
	\noindent
	
	\noindent
	If \(\mathbb{X}\) is a Banach space, we write \(\mathbb{L}^p([0,T];\mathbb{X})\) and \(C([0,T];\mathbb{X})\) for the
	usual Bochner spaces. When \(\mathbb{X}\) is reflexive space, the notation
	\[
	\mathbb{L}^2([0,T];\mathbb{X})_{\mathrm{weak}}
	\]
	means \(\mathbb{L}^2([0,T];\mathbb{X})\) endowed with its weak topology. 
	
	\noindent
	For a countable index set \(\mathcal I\), the Brownian path space is understood
	as
	\[
	C([0,T];\mathbb R^2)^{\mathcal I}
	\]
	endowed with the product topology.
	\section{Fully discrete spectral scheme on \(\T^2\)}
	
	Fix a final time \(T>0\) and a time step \(\Delta t\in(0,1]\). We define the
	uniform time grid
	\[
	t^n:=n\Delta t,
	\qquad n=0,1,\dots,N_T,
	\qquad
	N_T:=\lfloor T/\Delta t\rfloor.
	\]
	We also fix a Fourier cutoff \(N\in\mathbb N\) and a noise cutoff \(K\in\mathbb N\).
	
	\subsection{Spectral space, projector, and discrete filtration}
	
	For \(\mathbf m\in\mathbb Z^2\), let
	\[
	\mathrm e_{\mathbf m}(\mathbf x):=e^{i\mathbf m\cdot  \mathbf x},
	\qquad \mathbf x\in\T^2
	\]
	denote the standard complex Fourier basis. We define the finite-dimensional space
	of mean-zero trigonometric polynomials by
	\[
	\mathbb V_N
	:=
	\left\{
	v\in \mathbb{L}^2(\T^2):
	v(\mathbf x)=
	\sum_{\substack{\mathbf m\in\mathbb Z^2\\0<|\mathbf m|_\infty\le N}}
	\widehat v(\mathbf m)\,\mathrm e_{\mathbf m}(\mathbf x)
	\right\}.
	\]
	Equivalently, \(\mathbb V_N\subset \mathbb{L}^2(\T^2)\) consists of those mean-zero functions
	whose Fourier coefficients vanish outside the set
	\(\{\mathbf m\in\mathbb Z^2:\ |\mathbf m|_\infty\le N\}\).
	
	\noindent
	We denote by \(\mathcal P_N:\mathbb{L}^2(\T^2)\to\mathbb V_N\) the corresponding Fourier
	projector, defined by
	\[
	\widehat{(\mathcal P_N v)}(\mathbf m)
	=
	\mathbf 1_{\{0<|\mathbf m|_\infty\le N\}}\,\widehat v(\mathbf m),
	\qquad \mathbf m\in\mathbb Z^2.
	\]
	Thus \(\mathcal P_N\) removes the zero mode and truncates all modes outside the
	spectral box \(\{|\mathbf m|_\infty\le N\}\).
	
	\noindent
	For the noise, we impose the finite cutoff
	\[
	\mathcal M_K:=\{\mathbf m\in\mathbb Z^2\setminus\{{\bf 0}\}:\ |\mathbf m|_\infty\le K\},
	\qquad
	\mathcal I^K:=\mathcal M_K\times\{c,s\}.
	\]
	The corresponding discrete filtration is defined by
	\[
	\mathcal F_n
	:=
	\sigma\{{\bf W}^k(s):\ 0\le s\le t^n,\ k\in\mathcal I^K\},
	\qquad n=0,1,\dots,N_T.
	\]
	We also set
	\begin{equation}\label{eq:cK_def}
		c_K:=\sum_{\mathbf m\in\mathcal M_K} q(\mathbf m).
	\end{equation}
	Since \(q(\mathbf m)\sim |\mathbf m|^{-(2+2\alpha)}\) and \(\alpha>0\) in
	dimension \(2\), one has \(c_K<\infty\) for every \(K\), and in fact
	\[
	c_\infty:=\sup_{K\ge1}c_K<\infty.
	\]
	
	\noindent
	If \(\omega_0\in \mathbb{H}^{-1}(\T^2)\), the mean-zero condition is understood in the
	distributional sense:
	\[
	\langle \omega_0,1\rangle=0.
	\]
	Equivalently,
	\[
	\widehat{\omega_0}({0})=0.
	\]
	We then define the initial spectral approximation by
	\[
	\omega_N^0:=\mathcal P_N\omega_0\in\mathbb V_N.
	\]
	
	\subsection{Fully discrete scheme}
	
	Fix $\Delta t\in (0,1)$ and $N, K\in\mathbb{N}$. For each \(n=0,1,\dots,N_T-1\), define the Brownian increments
	\[
	\Delta_{n+1} {\bf W}_k:={\bf W}_k(t_{n+1})-{\bf W}_k(t_n),
	\qquad k\in\mathcal I^K.
	\]
	We seek an \((\mathcal F_n)\)-adapted sequence
	\[
	\omega_N^n\in \mathbb{L}^2(\Omega,\mathcal F_n;\mathbb V_N),
	\qquad n=0,1,\dots,N_T,
	\]
	such that, for every \(n=0,1,\dots,N_T-1\),
	\begin{equation}\label{eq:scheme_keep}
		\boxed{
			\begin{aligned}
				\omega_N^{n+1}-\omega_N^n
				&\;+\;\Delta t\,\mathcal P_N\bigl[(\mathbf u_N^{n+1}\cdot \nabla)\omega_N^n\bigr]
				\\
				&=
				\Delta t\,\frac{c_K}{2}\Delta\omega_N^{n+1}
				-
				\sum_{k\in\mathcal I^K}
				\mathcal P_N\bigl[(\boldsymbol{\sigma}_k\cdot \nabla)\omega_N^n\bigr]\Delta_{n+1} {\bf W}_k,
				\\
				\mathbf u_N^{n+1}&:=\nabla^\perp(-\Delta)^{-1}\omega_N^{n+1}.
		\end{aligned}}
	\end{equation}
	Since \(\omega_N^{n+1}\in\mathbb V_N\), the zero Fourier mode vanishes
	automatically, so the mean-zero condition is built into the scheme.
	
	\noindent
	Expanding the index \(k=(\mathbf m,\ell)\in\mathcal I^K\), the stochastic term in
	\eqref{eq:scheme_keep} may equivalently be written as
	\[
	\sum_{\mathbf m\in\mathcal M_K}
	\Bigl[
	\mathcal P_N\bigl[(\boldsymbol{\sigma}_{\mathbf m,c}\cdot \nabla)\omega_N^n\bigr]\Delta_{n+1} {W}_{\mathbf m,c}
	+
	\mathcal P_N\bigl[(\boldsymbol{\sigma}_{\mathbf m,s}\cdot \nabla)\omega_N^n\bigr]\Delta_{n+1} { W}_{\mathbf m,s}
	\Bigr].
	\]
	Hence the compact notation \(\sum_{k\in\mathcal I^K}\) is exactly equivalent to
	the explicit cosine--sine decomposition.
	
	\begin{remark}[Choice of the mixed-time transport term]
		\label{rem:mixed_time_conv}
		The transport term uses iterates \((\mathbf u_N^{n+1},\omega_N^n)\), hence is evaluated at the subsequent time levels. This choice is dictated by the
		\(\mathbb{H}^{-1}(\T^2)\)-energy estimate. Indeed, if one tests \eqref{eq:scheme_keep}
		against
		\[
		\psi^{n+1}:=(-\Delta)^{-1}\omega_N^{n+1},
		\]
		then
		\[
		\big\langle
		\mathcal P_N[(\mathbf u_N^{n+1}\cdot \nabla)\omega_N^n],\,\psi^{n+1}
		\big\rangle_{\mathbb{L}^2(\T^2)}
		=0
		\]
		by Lemma~\ref{lem:conv_cancel_keep}. The implicit Laplacian term therefore
		provides the coercive contribution in the discrete energy inequality; Section~\ref{sec4} below.
	\end{remark}
	
	The fully discrete approximation introduced above may be summarized in the following implementable Algorithm~\ref{alg:full_nodelta}.
	
	\begin{algorithm}[ht]
		\caption{Fully discrete spectral approximation of the Euler equations~\eqref{eq:intro_strat} on \(\T^2\)}
		\label{alg:full_nodelta}
		\begin{algorithmic}[1]
			\Require Final time \(T\), time step \(\Delta t\), Fourier cutoff \(N\), noise cutoff \(K\), exponent \(\alpha\in(0,1)\), initial datum \(\omega_0\in \mathbb{H}^{-1}(\T^2)\) with zero mean.
			\State Construct the spectral space \(\mathbb V_N\) and the Fourier projector \(\mathcal P_N\).
			\State Define \(\mathcal M_K=\{\mathbf m\in\mathbb Z^2\setminus\{{\bf 0}\}:\ |\mathbf m|_\infty\le K\}\) and \(\mathcal I^K=\mathcal M_K\times\{c,s\}\).
			\State Construct the truncated noise modes \(\boldsymbol{\sigma}_{\mathbf m,c}\), \(\boldsymbol{\sigma}_{\mathbf m,s}\) from \eqref{eq:noise_ modes}.
			\State Compute \(c_K=\sum_{\mathbf m\in\mathcal M_K}q(\mathbf m)\).
			\State Set \(\omega_N^0=\mathcal P_N\omega_0\in \mathbb V_N\), which is \(\mathcal F_0\)-measurable.
			\For{\(n=0,1,\dots,N_T-1\)}
			\State Assume \(\omega_N^n\in \mathbb{L}^2(\Omega,\mathcal F_n;\mathbb V_N)\) is known.
			\State Generate the Gaussian increments
			\[
			\Delta_{n+1}{\bf W}_k
			\sim
			\mathcal N\bigl(0, \Delta t\,I_2\bigr),
			\qquad k\in\mathcal I^K,
			\]
			independently of \(\mathcal F_n\).
			\State Form the \(\mathcal F_{n+1}\)-measurable stochastic term
			\[
			\Xi_n
			:=
			\sum_{k\in\mathcal I^K}
			\mathcal P_N\bigl[(\boldsymbol{\sigma}_k\cdot \nabla)\omega_N^n\bigr]\Delta_{n+1} {\bf W}_k.
			\]
			\State Determine \(\omega_N^{n+1}\in \mathbb{L}^2(\Omega,\mathcal F_{n+1};\mathbb V_N)\) from
			\[
			\omega_N^{n+1}
			+\Delta t\,\mathcal P_N\bigl[(\mathbf u_N^{n+1}\cdot \nabla)\omega_N^n\bigr]
			-\Delta t\,\frac{c_K}{2}\Delta\omega_N^{n+1}
			=
			\omega_N^n-\Xi_n,
			\,\,
			\mathbf u_N^{n+1}=\nabla^\perp(-\Delta)^{-1}\omega_N^{n+1}.
			\]
			\EndFor
		\end{algorithmic}
	\end{algorithm}
	\section{Coercivity estimate}\label{sec4}
	In this section, we establish the key coercivity estimate~(see~\eqref{eq:finite_dimensional_coercivity}) induced by the structure of the Kraichnan noise. On the torus, the noise coefficients admit an explicit Fourier representation, which allows us to work directly with their concrete form. This is in contrast to the \(\mathbb{R}^2\) setting in \cite{CoghiMaurelli2026}, where such an explicit representation is not available and the analysis is therefore carried out through the associated covariance matrix and its properties. Our approach is thus tailored to the torus geometry, while being guided by the structural discussion in \cite[Sec. IV]{FlandoliHuang2023}. In this sense, the argument presented below may be viewed as the torus counterpart of the covariance-based analysis in the Euclidean setting, with the advantage that the explicit form of the noise can be exploited directly.
	
	\subsection{Uniform Riemann-sum estimate and positive continuum defect}
	
	In the next subsection, we use the following uniform Riemann-sum estimate.
	
	\begin{lemma}\label{Riemann-sum}
		Let $\gamma>0$. Let \(\{D_{\mathbf e}\}_{\mathbf e\in\mathbb S^1}\) be a family of smooth bounded sets contained in
		\(B_\gamma=\{\mathbf z\in \mathbb{R}^2: |\mathbf z|\le \gamma\}\). Assume that there exists a constant \(C_D>0\) such that, for every
		\(0<\eta<1\),
		\begin{equation}\label{eq:uniform_boundaryayer}
			\left|
			\left\{
			\mathbf z\in\mathbb R^2:\operatorname{dist}(\mathbf z,\partial D_{\mathbf e})\le \eta
			\right\}
			\right|
			\le
			C_D\eta
			\qquad\text{uniformly in }\mathbf e\in\mathbb S^1 .
		\end{equation}
		Let \(G_{\mathbf e}\in C^1(D_{\mathbf e})\) satisfy
		\begin{align}\label{mean_inequality}
			\sup_{\mathbf e\in\mathbb S^1}\|G_{\mathbf e}\|_{C^1(D_{\mathbf e})}\le M .
		\end{align}
		Then
		\begin{equation}\label{eq:uniform_Riemann_sum_general}
			\sup_{\mathbf e\in\mathbb S^1}
			\bigg|
			r^{-2}
			\sum_{\mathbf m/r\in D_{\mathbf e}}G_{\mathbf e}(\mathbf m/r)
			-
			\int_{D_{\mathbf e}}G_{\mathbf e}(\mathbf z)\,{\rm d}\mathbf z
			\bigg|
			\longrightarrow 0
			\qquad\text{as }r\to\infty .
		\end{equation}
	\end{lemma}
	
	\begin{proof}
		For \(\mathbf m\in\mathbb Z^2\), define
		\[
		Q_{\mathbf m}^r:=\frac{\mathbf m}{r}+\left[0,\frac1r\right)^2 .
		\]
		Let
		\[
		\mathcal J_{\mathbf e}^r:=\{\mathbf m\in\mathbb Z^2:\ \mathbf m/r\in D_{\mathbf e}\},
		\qquad
		\mathcal I_{\mathbf e}^r:=\{\mathbf m\in\mathcal J_{\mathbf e}^r:\ Q_{\mathbf m}^r\subset D_{\mathbf e}\},
		\qquad
		\mathcal B_{\mathbf e}^r:=\mathcal J_{\mathbf e}^r\setminus\mathcal I_{\mathbf e}^r,
		\]
		and consider the set
		\[
		V_{\mathbf e}^r:=\bigcup_{\mathbf m\in\mathcal I_{\mathbf e}^r}Q_{\mathbf m}^r.
		\]
		Then
		\[
		\begin{aligned}
			r^{-2}\sum_{\mathbf m/r\in D_{\mathbf e}}G_{\mathbf e}(\mathbf m/r)-\int_{D_{\mathbf e}}G_{\mathbf e}(\mathbf z)\,{\rm d}\mathbf z
			=
			\Big(
			r^{-2}\sum_{\mathbf m\in\mathcal I_{\mathbf e}^r}G_{\mathbf e}(\mathbf m/r)-\int_{V_{\mathbf e}^r}G_{\mathbf e}(\mathbf z)\,{\rm d}\mathbf z
			\Big)\\
			+
			r^{-2}\sum_{\mathbf m\in\mathcal B_{\mathbf e}^r}G_{\mathbf e}(\mathbf m/r)
			-
			\int_{D_{\mathbf e}\setminus V_{\mathbf e}^r}G_{\mathbf e}(\mathbf z)\,{\rm d}\mathbf z.
		\end{aligned}
		\]
		Since \(\mathcal J_{\mathbf e}^r=\mathcal B_{\mathbf e}^r\cup \mathcal I_{\mathbf e}^r\), therefore
		\begin{align}\label{eq:riemann_split_corrected}
			\bigg|
			r^{-2}\sum_{\mathbf m/r\in D_{\mathbf e}}G_{\mathbf e}(\mathbf m/r)-\int_{D_{\mathbf e}}G_{\mathbf e}(\mathbf z)\,{\rm d}\mathbf z
			\bigg|
			\le
			\left|
			r^{-2}\sum_{\mathbf m\in\mathcal I_{\mathbf e}^r}G_{\mathbf e}(\mathbf m/r)-\int_{V_{\mathbf e}^r}G_{\mathbf e}(\mathbf z)\,{\rm d}\mathbf z
			\right|\notag\\
			+
			r^{-2}\sum_{\mathbf m\in\mathcal B_{\mathbf e}^r}|G_{\mathbf e}(\mathbf m/r)|
			+
			\int_{D_{\mathbf e}\setminus V_{\mathbf e}^r}|G_{\mathbf e}(\mathbf z)|\,{\rm d}\mathbf z.
		\end{align}
		
		\smallskip
		\noindent\textbf{Step 1: interior quadrature error.}
		If \(\mathbf m\in\mathcal I_{\mathbf e}^r\), then \(Q_{\mathbf m}^r\subset D_{\mathbf e}\), so \(G_{\mathbf e}\) is defined on
		\(Q_{\mathbf m}^r\). By the mean-value inequality and \eqref{mean_inequality}, we obtain
		\[
		|G_{\mathbf e}(\mathbf m/r)-G_{\mathbf e}(\mathbf z)|\le M|\mathbf z-\mathbf m/r|
		\le \frac{CM}{r}
		\qquad\text{for }\mathbf z\in Q_{\mathbf m}^r.
		\]
		Therefore
		\[
		\left|
		r^{-2}G_{\mathbf e}(\mathbf m/r)-\int_{Q_{\mathbf m}^r}G_{\mathbf e}(\mathbf z)\,{\rm d}\mathbf z
		\right|=\left|
		\int_{Q_{\mathbf m}^r}\bigl(G_{\mathbf e}(\mathbf m/r)-G_{\mathbf e}(\mathbf z)\bigr)\,{\rm d}\mathbf z
		\right|
		\le
		\frac{CM}{r^3}.
		\]
		Since \(D_{\mathbf e}\subset B_\gamma\), the number of such interior squares is bounded by \(Cr^2\),
		uniformly in \(\mathbf e\). Hence
		\begin{equation}\label{eq:interior_error_corrected}
			\bigg|
			r^{-2}\sum_{\mathbf m\in\mathcal I_{\mathbf e}^r}G_{\mathbf e}(\mathbf m/r)-\int_{V_{\mathbf e}^r}G_{\mathbf e}(\mathbf z)\,{\rm d}\mathbf z
			\bigg|
			\le
			\frac{C}{r}.
		\end{equation}
		
		\smallskip
		\noindent\textbf{Step 2: boundary layer control.}
		We claim that
		\[
		D_{\mathbf e}\setminus V_{\mathbf e}^r
		\subset
		\left\{
		\mathbf z\in\mathbb R^2:\operatorname{dist}(\mathbf z,\partial D_{\mathbf e})\le \frac{\sqrt2}{r}
		\right\}.
		\]
		Indeed, let \(\mathbf z\in D_{\mathbf e}\setminus V_{\mathbf e}^r\). Since the half-open squares \(Q_{\mathbf m}^r\) form a
		partition of \(\mathbb R^2\), there exists a unique \(\mathbf m\in\mathbb Z^2\) such that
		\(\mathbf z\in Q_{\mathbf m}^r\). If \(Q_{\mathbf m}^r\subset D_{\mathbf e}\) and \(\mathbf m/r\in D_{\mathbf e}\), then \(\mathbf m\in\mathcal I_{\mathbf e}^r\)
		and \(\mathbf z\in V_{\mathbf e}^r\), which is a contradiction. Thus either \(Q_{\mathbf m}^r\not\subset D_{\mathbf e}\), or
		\(\mathbf m/r\notin D_{\mathbf e}\). In both cases, one finds a point \(\mathbf y\in Q_{\mathbf m}^r\cap\partial D_{\mathbf e}\),
		and hence
		\[
		\operatorname{dist}(\mathbf z,\partial D_{\mathbf e})\le |\mathbf z-\mathbf y|\le \frac{\sqrt2}{r}.
		\]
		Similarly, if \(\mathbf m\in\mathcal B_{\mathbf e}^r\), then \(\mathbf m/r\in D_{\mathbf e}\) but \(Q_{\mathbf m}^r\not\subset D_{\mathbf e}\),
		so again there exists \(\mathbf y\in Q_{\mathbf m}^r\cap\partial D_{\mathbf e}\), and therefore
		\[
		Q_{\mathbf m}^r\subset
		\left\{
		\mathbf z\in\mathbb R^2:\operatorname{dist}(\mathbf z,\partial D_{\mathbf e})\le \frac{\sqrt2}{r}
		\right\}.
		\]
		By assumption~\eqref{eq:uniform_boundaryayer}, we therefore have for the Lebesgue measure of this boundary layer
		\[
		\sup_{\mathbf e\in \mathbb{S}^1}\left|
		\left\{
		\mathbf z\in\mathbb R^2:\operatorname{dist}(\mathbf z,\partial D_{\mathbf e})\le \frac{\sqrt2}{r}
		\right\}
		\right|
		\le \frac{C}{r},
		\]
		Consequently, for the Lebesgue measure of $D_{\mathbf e}\setminus V_{\mathbf e}^r$,
		\begin{equation}\label{eq:boundaryayer_measure_corrected}
			\sup_{\mathbf e\in \mathbb{S}^1}|D_{\mathbf e}\setminus V_{\mathbf e}^r|\le \frac{C}{r},
			\qquad
			\sup_{\mathbf e\in \mathbb{S}^1}\bigg|\bigcup_{\mathbf m\in\mathcal B_{\mathbf e}^r}Q_{\mathbf m}^r\bigg|\le \frac{C}{r}.
		\end{equation}
		Since the squares are pairwise disjoint and each has area \(r^{-2}\), it follows that
		\[
		\#\mathcal B_{\mathbf e}^r\cdot r^{-2}
		=
		\bigg|\bigcup_{\mathbf m\in\mathcal B_{\mathbf e}^r}Q_{\mathbf m}^r\bigg|
		\le \frac{C}{r},
		\]
		hence
		\[
		\#\mathcal B_{\mathbf e}^r\le Cr
		\qquad\text{uniformly in }\mathbf e.
		\]
		By using \(|G_{\mathbf e}|\le M\), we obtain
		\begin{equation}\label{eq:boundary_sum_corrected}
			r^{-2}\sum_{\mathbf m\in\mathcal B_{\mathbf e}^r}|G_{\mathbf e}(\mathbf m/r)|
			\le
			Mr^{-2}\#\mathcal B_{\mathbf e}^r
			\le
			\frac{C}{r},
		\end{equation}
		and
		\begin{equation}\label{eq:boundary_integral_corrected}
			\int_{D_{\mathbf e}\setminus V_{\mathbf e}^r}|G_{\mathbf e}(\mathbf z)|\,{\rm d}\mathbf z
			\le
			M|D_{\mathbf e}\setminus V_{\mathbf e}^r|
			\le
			\frac{C}{r}.
		\end{equation}
		Finally, by combining \eqref{eq:riemann_split_corrected},
		\eqref{eq:interior_error_corrected},
		\eqref{eq:boundary_sum_corrected}, and
		\eqref{eq:boundary_integral_corrected}, we conclude that
		\[
		\sup_{\mathbf e\in\mathbb S^1}
		\bigg|
		r^{-2}\sum_{\mathbf m/r\in D_{\mathbf e}}G_{\mathbf e}(\mathbf m/r)-\int_{D_{\mathbf e}}G_{\mathbf e}(\mathbf z)\,{\rm d}\mathbf z
		\bigg|
		\le
		\frac{C}{r},
		\]
		which proves the claim~\eqref{eq:uniform_Riemann_sum_general}.
	\end{proof}
	
	\begin{lemma}[Positivity of the continuum defect]
		\label{lem:positivity_j_alpha_4}
		Fix \(0<\alpha<1\) and $\gamma>1$. For \(\mathbf e\in\mathbb S^1\), define
		\begin{align}\label{t8}
			\tau_{\mathbf e}(\mathbf z)
			:=
			\frac{|\mathbf z^\perp\cdot  \mathbf e|^2}{|\mathbf z|^2}
			\left(
			\frac{1}{|\mathbf z+\mathbf e|^2}
			+
			\frac{1}{|\mathbf z-\mathbf e|^2}
			\right),
			\qquad \mathbf z\neq {\bf 0},\pm \mathbf e,
		\end{align}
		and
		\begin{align}\label{t9}
			F_{\mathbf e}(\mathbf z):=|\mathbf z|^{-2-2\alpha}\bigl(1-\tau_{\mathbf e}(\mathbf z)\bigr).
		\end{align}
		Then the principal value integral
		\begin{align}\label{t99}
			J_\gamma(\mathbf e):=\operatorname{p.v.}\int_{|\mathbf z|\le\gamma}F_{\mathbf e}(\mathbf z)\,{\rm d}\mathbf z
			=\lim_{\varepsilon\downarrow0}
			\int_{\varepsilon<|\mathbf z|\le\gamma}F_{\mathbf e}(\mathbf z)\,{\rm d}\mathbf z
		\end{align}
		is well defined, independent of \(\mathbf e\in\mathbb S^1\), and satisfies
		\[
		J_\gamma(\mathbf e)=j_{\alpha, \gamma},
		\]
		where
		\begin{align}\label{today20}
			j_{\alpha, \gamma}
			:=
			2\pi\int_1^\gamma \rho^{-1-2\alpha}(1-\rho^{-2})\,{\rm d}\rho .
		\end{align}
		In particular,
		\begin{align}\label{today09}
			j_{\alpha, \gamma}>0.
		\end{align}
	\end{lemma}
	
	\begin{proof}
		The principal value in \eqref{t99} is taken only at the singular point \({\mathbf z}={\bf 0}\). The function \(\tau_{\mathbf e}\) is bounded near
		\(\pm \mathbf e\). Indeed, near \(\mathbf z=\mathbf e\),
		\[
		\mathbf z^\perp\cdot  \mathbf e=\mathbf z^\perp\cdot (\mathbf e-\mathbf z).
		\]
		Therefore
		\[
		\frac{|\mathbf z^\perp\cdot  \mathbf e|^2}{|\mathbf z|^2|\mathbf z-\mathbf e|^2}
		=
		\frac{|\mathbf z^\perp\cdot (\mathbf e-\mathbf z)|^2}{|\mathbf z|^2|\mathbf z-\mathbf e|^2}
		\le 1.
		\]
		The same argument applies near \(\mathbf z=-\mathbf e\).
		
		\noindent
		\textbf{Step 1.} We first show that \(J_\gamma(\mathbf e)\) is independent of \(\mathbf e\). Let \(R\in SO(2)\) be a
		rotation such that \(R\mathbf e_1=\mathbf e\), where ${\bf e}_1= (1,0)$. Since rotations commute with the perpendicular
		map in two dimensions, we have
		\[
		(R\mathbf y)^\perp=R(\mathbf y^\perp).
		\]
		Therefore, for \(\mathbf z=R\mathbf y\),
		\[
		\mathbf z^\perp\cdot  \mathbf e
		=
		(R\mathbf y)^\perp\cdot  R\mathbf e_1
		=
		R(\mathbf y^\perp)\cdot  R\mathbf e_1
		=
		\mathbf y^\perp\cdot  \mathbf e_1.
		\]
		Moreover,
		\[
		|\mathbf z|=|\mathbf y|,
		\qquad
		|\mathbf z\pm \mathbf e|
		=
		|R\mathbf y\pm R\mathbf e_1|
		=
		|R(\mathbf y\pm \mathbf e_1)|
		=
		|\mathbf y\pm \mathbf e_1|.
		\]
		Hence
		\[
		\tau_{\mathbf e}(\mathbf z)=\tau_{\mathbf e_1}(\mathbf y)=\tau_{\mathbf e_1}(R^{-1}\mathbf z).
		\]
		Since the ball \(B_\gamma=\{|\mathbf z|\le\gamma\}\), the Lebesgue measure \({\rm d}\mathbf z\), and the weight
		\(|\mathbf z|^{-2-2\alpha}\) are rotationally invariant, the change of variables
		\(\mathbf z=R\mathbf y\) gives
		\[
		\begin{aligned}
			J_\gamma(\mathbf e)
			&=
			\operatorname{p.v.}\int_{|\mathbf z|\le\gamma}
			|\mathbf z|^{-2-2\alpha}
			\bigl(1-\tau_{\mathbf e}(\mathbf z)\bigr)\,{\rm d}\mathbf z
			\\
			&=
			\operatorname{p.v.}\int_{|\mathbf y|\le\gamma}
			|\mathbf y|^{-2-2\alpha}
			\bigl(1-\tau_{\mathbf e_1}(\mathbf y)\bigr)\,{\rm d}\mathbf y
			\\
			&=
			J_\gamma(\mathbf e_1).
		\end{aligned}
		\]
		Thus it is enough to compute the defect for \(\mathbf e=\mathbf e_1\). We write
		\[
		\mathbf z=\rho(\cos\varphi,\sin\varphi),
		\qquad
		\rho>0,\quad \varphi\in[0,2\pi).
		\]
		Then
		\[
		\frac{|\mathbf z^\perp\cdot  \mathbf e_1|^2}{|\mathbf z|^2}
		=
		\sin^2\varphi,
		\]
		and
		\[
		|\mathbf z\pm \mathbf e_1|^2
		=
		\rho^2+1\pm2\rho\cos\varphi.
		\]
		Therefore by using~\(\mathbf e_\varphi=(\cos\varphi,\sin\varphi)\) and \eqref{t8}, we write
		\begin{align}\label{hee}
			\tau_{\mathbf e_1}(\rho \mathbf e_\varphi)
			=
			\sin^2\varphi
			\left(
			\frac{1}{1+\rho^2+2\rho\cos\varphi}
			+
			\frac{1}{1+\rho^2-2\rho\cos\varphi}
			\right).
		\end{align}
		
		\noindent
		\textbf{Step 2.} In this step, we now compute the angular average of \(\tau_{\mathbf e_1}\). We claim that
		\begin{align}\label{t1}
			\frac1{2\pi}
			\int_0^{2\pi}
			\tau_{\mathbf e_1}(\rho \mathbf e_\varphi)\,{\rm d}\varphi
			=
			\begin{cases}
				1, & 0<\rho<1,\\[1mm]
				\rho^{-2}, & \rho>1.
			\end{cases}
		\end{align}
		
		\noindent
		\textbf{Step 2({\bf a}).}
		Let \(0<\rho<1\). The Poisson-kernel expansion (see, for instance, \cite[Chapter~2, Section~5.4, Lemma~5.5]{SteinShakarchi2003}) gives
		\[
		\frac{1}{1-2\rho\cos\varphi+\rho^2}
		=
		\frac{1}{1-\rho^2}
		\left(
		1+2\sum_{k\ge1}\rho^k\cos(k\varphi)
		\right),
		\]
		and, replacing \(\rho\) by \(-\rho\),
		\[
		\frac{1}{1+2\rho\cos\varphi+\rho^2}
		=
		\frac{1}{1-\rho^2}
		\left(
		1+2\sum_{k\ge1}(-\rho)^k\cos(k\varphi)
		\right).
		\]
		By adding the two identities, all odd Fourier modes cancel and we obtain
		\begin{align}\label{hee2}
			\frac{1}{1-2\rho\cos\varphi+\rho^2}
			+
			\frac{1}{1+2\rho\cos\varphi+\rho^2}
			=
			\frac{2}{1-\rho^2}
			\left(
			1+2\sum_{j\ge1}\rho^{2j}\cos(2j\varphi)
			\right).
		\end{align}
		By using
		\[
		\sin^2\varphi=\frac{1-\cos(2\varphi)}2,
		\]
		we compute
		\begin{align}\label{hee3}
			&\frac1{2\pi}\int_0^{2\pi}
			\sin^2\varphi
			\left[
			\frac{1}{1-2\rho\cos\varphi+\rho^2}
			+
			\frac{1}{1+2\rho\cos\varphi+\rho^2}
			\right]
			\,{\rm d}\varphi\notag
			\\
			&\quad =
			\frac1{1-\rho^2}
			\frac1{2\pi}
			\int_0^{2\pi}
			(1-\cos(2\varphi))
			\left(
			1+2\sum_{j\ge1}\rho^{2j}\cos(2j\varphi)
			\right)
			\,{\rm d}\varphi.
		\end{align}
		By the property of cos function, we obtain
		\begin{align}\label{hee4}
			\frac1{2\pi}\int_0^{2\pi}
			\left(
			1+2\sum_{j\ge1}\rho^{2j}\cos(2j\varphi)
			\right)
			\,{\rm d}\varphi
			=
			1,
		\end{align}
		By the orthogonality of trigonometric functions, we have also
		\begin{align}\label{hee5}
			\frac1{2\pi}\int_0^{2\pi}
			\cos(2\varphi)
			\left(
			1+2\sum_{j\ge1}\rho^{2j}\cos(2j\varphi)
			\right)
			\,{\rm d}\varphi
			=
			\rho^2.
		\end{align}
		Indeed, only the \(j=1\) term survives in the second integral and
		\[
		\frac1{2\pi}\int_0^{2\pi}\cos^2(2\varphi)\,{\rm d}\varphi=\frac12.
		\]
		Therefore from~\eqref{hee}-\eqref{hee5}, the angular average for \(0<\rho<1\),
		\begin{align}\label{t3}
			\frac1{2\pi}
			\int_0^{2\pi}
			\tau_{\mathbf e_1}(\rho \mathbf e_\varphi)\,{\rm d}\varphi=\frac{1-\rho^2}{1-\rho^2}=1.
		\end{align}
		This proves the formula~\eqref{t1} for \(0<\rho<1\).
		
		\noindent
		\textbf{Step 2(b).} For \(\rho>1\), set \(s=\rho^{-1}\in(0,1)\). Since
		\[
		1+\rho^2\pm2\rho\cos\varphi
		=
		\rho^2(1+s^2\pm2s\cos\varphi),
		\]
		we have
		\[
		\tau_{\mathbf e_1}(\rho \mathbf e_\varphi)
		=
		\rho^{-2}
		\sin^2\varphi
		\left(
		\frac{1}{1+s^2+2s\cos\varphi}
		+
		\frac{1}{1+s^2-2s\cos\varphi}
		\right)
		=
		\rho^{-2}\tau_{\mathbf e_1}(s \mathbf e_\varphi).
		\]
		Then \eqref{t3} gives for \(\rho>1\)
		\[
		\frac1{2\pi}
		\int_0^{2\pi}
		\tau_{\mathbf e_1}(\rho \mathbf e_\varphi)\,{\rm d}\varphi
		=
		\rho^{-2}\frac1{2\pi}
		\int_0^{2\pi}
		\tau_{\mathbf e_1}(s \mathbf e_\varphi)\,{\rm d}\varphi
		=
		\rho^{-2}.
		\]
		This proves the angular-average identity~\eqref{t1} for \(\rho>1\).
		
		\noindent
		\textbf{Step 3.}
		In this step, we can now compute \(J_\gamma(\mathbf e)\). For \(\varepsilon>0\), define
		\[
		J_{4,\varepsilon}(\mathbf e_1)
		:=
		\int_{\varepsilon<|\mathbf z|\le\gamma}
		|\mathbf z|^{-2-2\alpha}
		\bigl(1-\tau_{\mathbf e_1}(\mathbf z)\bigr)\,{\rm d}\mathbf z.
		\]
		By using polar coordinates, we have
		\[
		\begin{aligned}
			J_{4,\varepsilon}(\mathbf e_1)
			&=
			\int_\varepsilon^\gamma
			\rho^{-2-2\alpha}\rho
			\left[
			\int_0^{2\pi}
			\bigl(1-\tau_{\mathbf e_1}(\rho \mathbf e_\varphi)\bigr)
			\,{\rm d}\varphi
			\right]\,{\rm d}\rho
			\\
			&=
			2\pi
			\int_\varepsilon^\gamma
			\rho^{-1-2\alpha}
			\left[
			1-
			\frac1{2\pi}
			\int_0^{2\pi}
			\tau_{\mathbf e_1}(\rho \mathbf e_\varphi)\,{\rm d}\varphi
			\right]
			\,{\rm d}\rho.
		\end{aligned}
		\]
		For \(0<\rho<1\), the angular average is \(1\) by using \eqref{t1}, so the contribution from
		\(\varepsilon<\rho<1\) is exactly zero.
		
		\noindent
		For \(1<\rho<\gamma\), the angular average is
		\(\rho^{-2}\) by using~\eqref{t1}. Hence, for every \(\varepsilon<1\),
		\[
		J_{4,\varepsilon}(\mathbf e_1)
		=
		2\pi
		\int_1^\gamma
		\rho^{-1-2\alpha}
		(1-\rho^{-2})\,{\rm d}\rho.
		\]
		The right-hand side is independent of \(\varepsilon\), so the principal value
		exists and
		\[
		J_\gamma(\mathbf e)
		=
		J_\gamma(\mathbf e_1)
		=
		2\pi
		\int_1^\gamma
		\rho^{-1-2\alpha}
		(1-\rho^{-2})\,{\rm d}\rho.
		\]
		We denote
		\[
		j_{\alpha, \gamma}
		:=
		2\pi
		\int_1^\gamma
		\rho^{-1-2\alpha}
		(1-\rho^{-2})\,{\rm d}\rho.
		\]
		Since
		\[
		\rho^{-1-2\alpha}>0
		\qquad\text{and}\qquad
		1-\rho^{-2}>0
		\quad\text{for }1<\rho<\gamma,
		\]
		we obtain
		\begin{align*}
			j_{\alpha, \gamma}>0.
		\end{align*}
		This proves the desired positivity of the continuum defect.
	\end{proof}
	\begin{remark}[Role of the condition \(\gamma>1\)]
		\label{rem:importance_gamma}
		The condition \(\gamma>1\) is essential for the positivity of the continuum
		defect. Indeed, as follows from the proof of
		Lemma~\ref{lem:positivity_j_alpha_4}, if \(0<\gamma\le 1\), then
		\[
		J_\gamma(\mathbf e)=0
		\qquad \text{for every } \mathbf e\in\mathbb S^1,
		\]
		and hence
		\[
		j_{\alpha,\gamma}=0.
		\]
		Thus no positive dissipative contribution can be obtained in this regime. The
		strict inequality \(\gamma>1\) is therefore necessary in order to detect the
		regularizing effect generated by the high Kraichnan modes; see Remark~\ref{rem:regularization_high_modes} below for more details.
	\end{remark}
	\subsection{Discrete noise quadratic form}
	
	\begin{definition}[Discrete noise quadratic form]\label{def:DLK}
		For mean-zero $\omega$ on $\T^2$ define
		\begin{align}\label{quadratic form}
			\mathfrak D_{N, K}(\omega)
			&:=\sum_{k\in\mathcal I^K}
			\Big\|(-\Delta)^{-1/2}\mathcal{P}_N[(\boldsymbol{\sigma}_k\!\cdot \nabla)\omega]\Big\|_{\mathbb{L}^2(\T^2)}^2\notag \\
			&=
			\sum_{k\in\mathcal I^K}
			\Big\langle\mathcal{P}_N[(\boldsymbol{\sigma}_k\!\cdot \nabla)\omega],\ (-\Delta)^{-1}\mathcal{P}_N[(\boldsymbol{\sigma}_k\!\cdot \nabla)\omega]\Big\rangle_{\mathbb{L}^2(\T^2)}.
		\end{align}
	\end{definition}

	The following lemma establishes a key inequality that plays a crucial role in the derivation of the coercivity estimate~\eqref{eq:finite_dimensional_coercivity} below. Although its proof is roughly motivated by \cite[Corollary 11]{FlandoliHuang2023} and \cite[Lemma 4.3]{CoghiMaurelli2026}, the proof given below is self-contained and, to the best of our knowledge, new in the present torus setting. In this sense, the result may be viewed as a discrete torus analogue of \cite[Lemma 4.3]{CoghiMaurelli2026}. However, the argument in \cite[Lemma 4.3]{CoghiMaurelli2026} relies on structural features of the Euclidean setting that are not directly available on the torus, and therefore a different proof is required in the present framework.
	\begin{lemma}
		\label{lem:finite_discrete_symbol_gap}
		Fix \(0<\alpha<1\) and $\gamma>1+\sqrt{2}$. Let
		\[
		\mathcal K_N:=\{\mathbf n\in\mathbb Z^2:\ |\mathbf n|_\infty\le N\},
		\qquad
		\mathcal M_K:=\{\mathbf m\in\mathbb Z^2\setminus\{{\bf 0}\}:\ |\mathbf m|_\infty\le K\}.
		\]
		Assume 
		\begin{align}\label{today30}
			K\ge \gamma \sqrt2\,N.
		\end{align}
		Then there exist constants
		\(c_0>0\) and \(C_0\ge0\), depending only on \(\alpha\), such that for every \({\bf 0}\neq \mathbf n\in\mathcal K_N\),
		\begin{equation}\label{eq:finite_symbol_gap}
			\sum_{\mathbf m\in\mathcal M_K}
			q({\mathbf m})
			\bigl(1 - T({\mathbf n},{\mathbf m})\bigr)
			\ge
			c_0\langle{\mathbf n}\rangle^{-2\alpha}
			-
			C_0\langle{\mathbf n}\rangle^{-2},
		\end{equation}
		where
		\begin{align}\label{t6}
			T({\mathbf n},{\mathbf m})
			:=
			\left|
			\frac{{\mathbf m}^\perp}{|{\mathbf m}|}\cdot {\mathbf n}
			\right|^2
			\left(
			\frac{\mathbf 1_{\{\mathbf n+\mathbf {\bf m}\neq{\bf 0}\}}}{|{\mathbf n+\mathbf m}|^2}
			+
			\frac{\mathbf 1_{\{\mathbf n-\mathbf {\bf m}\neq{\bf 0}\}}}{|{\mathbf n-\mathbf m}|^2}
			\right).
		\end{align}
	\end{lemma}
	
	\begin{proof} We set
		\[
		r:=|\mathbf n|=\sqrt{n_1^2+n_2^2}.
		\]
		Since \(\mathbf n\in\mathcal K_N\), we have
		\[
		r=|\mathbf n|\le \sqrt2\,|\mathbf n|_\infty\le \sqrt2\,N.
		\]
		Hence the condition \(K\ge \gamma\sqrt2\,N\) implies
		\[
		K\ge \gamma r.
		\]
		Therefore
		\[
		\{\mathbf m\in\mathbb Z^2:\ 0<|\mathbf m|\le  \gamma r\}
		\subset \mathcal M_K.
		\]
		We split the finite sum as
		\begin{align}\label{new600}
			\sum_{\mathbf m\in\mathcal M_K}q_{\mathbf m}\bigl(1-T(\mathbf n,\mathbf m)\bigr)
			=
			\sum_{0<|\mathbf m|\le\gamma r}q_{\mathbf m}\bigl(1-T(\mathbf n,\mathbf m)\bigr)
			+
			\sum_{\substack{\mathbf m\in\mathcal M_K\\ |\mathbf m|>\gamma r}}
			q_{\mathbf m}\bigl(1-T(\mathbf n,\mathbf m)\bigr).
		\end{align}
		The last terms is nonnegative. Indeed, if \(|\mathbf m|>\gamma r\), then
		\[
		|\mathbf m\pm \mathbf n|
		\ge |\mathbf m|-|\mathbf n|
		>
		\frac{\gamma-1}{\gamma}|\mathbf m|.
		\]
		Therefore
		\[
		\begin{aligned}
			T(\mathbf n,\mathbf m)
			&\le
			|\mathbf n|^2
			\left(
			\frac{1}{|\mathbf m+\mathbf n|^2}
			+
			\frac{1}{|\mathbf m-\mathbf n|^2}
			\right)
			\\
			&\le
			2r^2\left(\frac{\gamma}{(\gamma-1)|\mathbf m|}\right)^2
			=
			\frac{2\gamma^2}{(\gamma-1)^2}\frac{r^2}{|\mathbf m|^2}
			<
			\frac{2}{(\gamma-1)^2}
			<1.
		\end{aligned}
		\]
		Thus
		\[
		1-T(\mathbf n,\mathbf m)
		>
		1-\frac{2}{(\gamma-1)^2}
		>0,
		\]
		where we used \(\gamma>1+\sqrt{2}\).
		Consequently,
		\begin{align}\label{t13}
			\sum_{\mathbf m\in\mathcal M_K}q_{\mathbf m}\bigl(1-T(\mathbf n,\mathbf m)\bigr)
			\ge
			S_r(\mathbf n),
		\end{align}
		where
		\begin{align}\label{first}
			S_r(\mathbf n):=
			\sum_{0<|\mathbf m|\le\gamma r}q_{\mathbf m}\bigl(1-T(\mathbf n,\mathbf m)\bigr).\end{align}
		To prove \eqref{eq:finite_symbol_gap}, it remains to prove that there exist constants \(c_0>0\) and \(C_0\ge0\), depending
		only on \(\alpha\), such that
		\begin{equation}\label{eq:ball_sum_gap}
			S_r(\mathbf n)\ge c_0 r^{-2\alpha}-C_0 r^{-2}.
		\end{equation}
		In Step~1 of the proof below, we show the asymptotic statement as $r\to \infty$, 
		\begin{equation}\label{eq:discrete_to_continuum}
			r^{2\alpha}S_r(\mathbf n)\longrightarrow j_{\alpha, \gamma}>0
		\end{equation}
		uniformly in the direction \(\mathbf e=\mathbf n/|\mathbf n|\), meaning that
		\[
		\lim_{R\to\infty}
		\sup_{\substack{\mathbf n\in\mathbb Z^2\setminus\{{\bf 0}\}\\ |\mathbf n|\ge R}}
		\left|
		|\mathbf n|^{2\alpha}S_{|\mathbf n|}(\mathbf n)-j_{\alpha, \gamma}
		\right|
		=0,
		\]
		where $j_{\alpha, \gamma}$ is given in \eqref{today20}.
		This uniformity is needed in order to choose a single \(r_0\) such that
		\[
		S_{|\mathbf n|}(\mathbf n)\ge \frac12 j_{\alpha, \gamma}|\mathbf n|^{-2\alpha}
		\]
		for every \(\mathbf n\) with \(|{\mathbf n}|\ge r_0\) in Step~2 to eventually show~\eqref{eq:finite_symbol_gap}.

		\noindent
		\textbf{Step 1.} In this step, we now prove \eqref{eq:discrete_to_continuum}. We fix
		\[
		0<\delta<\frac1{10}.
		\]
		For \(\mathbf e\in\mathbb S^1\), define
		\[
		\Omega_{\delta,\mathbf e}
		:=
		B_\gamma
		\setminus
		\Big(
		B_\delta({\bf 0})\cup B_\delta(\mathbf e)\cup B_\delta(-\mathbf e)
		\Big).
		\]
		We define the following regions to split the discrete sum in~\eqref{first}:
		\[
		A_0^r:=\{\mathbf m\in\mathbb Z^2:\ 0<|\mathbf m|\le \delta r\},
		\]
		\[
		A_+^r
		:=
		\{\mathbf m\in\mathbb Z^2:\ 0<|\mathbf m|\le\gamma r,\ |\mathbf m-\mathbf n|\le\delta r,\ \mathbf m\notin A_0^r\},
		\]
		\[
		A_-^r
		:=
		\{\mathbf m\in\mathbb Z^2:\ 0<|\mathbf m|\le\gamma r,\ |\mathbf m+\mathbf n|\le\delta r,\ 
		\mathbf m\notin A_0^r\cup A_+^r\},
		\]
		and
		\[
		A_{\mathrm{reg}}^r
		:=
		\{\mathbf m\in\mathbb Z^2:\ 0<|\mathbf m|\le\gamma r\}
		\setminus
		\bigl(A_0^r\cup A_+^r\cup A_-^r\bigr).
		\]
		Then the sets
		\[
		A_0^r,\qquad A_+^r,\qquad A_-^r,\qquad A_{\mathrm{reg}}^r
		\]
		are pairwise disjoint and their union is exactly
		\[
		\{\mathbf m\in\mathbb Z^2:\ 0<|\mathbf m|\le\gamma r\}.
		\]
		Therefore
		
		\begin{align}\label{oh7}
			S_r(\mathbf n)
			&=
			\sum_{\mathbf m\in A_{\mathrm{reg}}^r}q_{\mathbf m}(1-T(\mathbf n,\mathbf m))
			+
			\sum_{\mathbf m\in A_0^r}q_{\mathbf m}(1-T(\mathbf n,\mathbf m))\notag
			\\
			&\quad
			+
			\sum_{\mathbf m\in A_+^r}q_{\mathbf m}(1-T(\mathbf n,\mathbf m))
			+
			\sum_{\mathbf m\in A_-^r}q_{\mathbf m}(1-T(\mathbf n,\mathbf m)).
		\end{align}
		We now identify the regular part exactly. By construction,
		\(\mathbf m\in A_{\mathrm{reg}}^r\) if and only if
		\[
		0<|\mathbf m|\le\gamma r,
		\qquad
		|\mathbf m|>\delta r,
		\qquad
		|\mathbf m-\mathbf n|>\delta r,
		\qquad
		|\mathbf m+\mathbf n|>\delta r.
		\]
		Dividing by \(r\), and using \(\mathbf e=\mathbf n/r\), this is equivalent to
		\[
		0<\left|\frac{\mathbf m}{r}\right|\le\gamma,
		\qquad
		\left|\frac{\mathbf m}{r}\right|>\delta,
		\qquad
		\left|\frac{\mathbf m}{r}-\mathbf e\right|>\delta,
		\qquad
		\left|\frac{\mathbf m}{r}+\mathbf e\right|>\delta.
		\]
		By the definition of \(\Omega_{\delta,\mathbf e}\), this is equivalent to
		\[
		0<|\mathbf m|\le\gamma r,
		\qquad
		\frac{\mathbf m}{r}\in\Omega_{\delta,\mathbf e}.
		\]
		Therefore
		\begin{align*}
			A_{\mathrm{reg}}^r
			=
			\left\{
			\mathbf m\in\mathbb Z^2:\ 0<|\mathbf m|\le\gamma r,\ \frac{\mathbf m}{r}\in\Omega_{\delta,\mathbf e}
			\right\}.
		\end{align*}
		Consequently,
		\begin{align}\label{today0111}
			\sum_{\mathbf m\in A_{\mathrm{reg}}^r}q_{\mathbf m}(1-T(\mathbf n,\mathbf m))
			=
			\sum_{\mathbf m/r\in\Omega_{\delta,\mathbf e}}
			q_{\mathbf m}(1-T(\mathbf n,\mathbf m)).\end{align}
		
		\noindent
		To proceed with the proof, we shall use the following auxiliary results,
		whose proofs are given in Step~3 below.
		
		\begin{itemize}
			\item[\textup{1.}]
			The discrete contribution near the origin satisfies
			\begin{equation}\label{today03}
				r^{2\alpha}
				\left|
				\sum_{\mathbf m\in A_0^r}q_{\mathbf m}\bigl(1-T(\mathbf n,\mathbf m)\bigr)
				\right|
				\le
				C\delta^{2-2\alpha}.
			\end{equation}
			
			\item[\textup{2.}]
			The discrete contribution near the points \(\pm\mathbf e\) satisfies
			\begin{equation}\label{today04}
				r^{2\alpha}
				\left|
				\sum_{\mathbf m\in A_+^r\cup A_-^r}q_{\mathbf m}
				\bigl(1-T(\mathbf n,\mathbf m)\bigr)
				\right|
				\le
				C\delta^2+Cr^{-2}.
			\end{equation}
			
			\item[\textup{3.}]
			By recalling \eqref{t9}, the continuum contribution near the origin satisfies
			\begin{equation}\label{today07}
				\left|
				\operatorname{p.v.}\int_{|\mathbf z|\le\delta}
				F_{\mathbf e}(\mathbf z)\,{\rm d}\mathbf z
				\right|
				\le
				C\delta^{2-2\alpha}.
			\end{equation}
			
			\item[\textup{4.}]
			The continuum contributions near the points \(\pm\mathbf e\) satisfy
			\begin{equation}\label{today08}
				\left|
				\int_{|\mathbf z-\mathbf e|\le\delta}
				F_{\mathbf e}(\mathbf z)\,{\rm d}\mathbf z
				\right|
				+
				\left|
				\int_{|\mathbf z+\mathbf e|\le\delta}
				F_{\mathbf e}(\mathbf z)\,{\rm d}\mathbf z
				\right|
				\le
				C\delta^2.
			\end{equation}
			
			\item[\textup{5.}]
			On the regular region \(\Omega_{\delta,\mathbf e}\), the discrete expression is
			asymptotically equivalent to its Riemann-sum form:
			\begin{equation}\label{today02}
				\lim_{r\to \infty}\Bigg[\sup_{\mathbf e\in\mathbb S^1}\bigg|
				r^{2\alpha}
				\sum_{\mathbf m/r\in\Omega_{\delta,\mathbf e}}
				q_{\mathbf m}\bigl(1-T(\mathbf n,\mathbf m)\bigr)
				-
				r^{-2}
				\sum_{\mathbf m/r\in\Omega_{\delta,\mathbf e}}F_{\mathbf e}(\mathbf m/r)
				\bigg|\Bigg]=0.
			\end{equation}
			
			\item[\textup{6.}]
			Finally, the Riemann sums on the regular region \(\Omega_{\delta,\mathbf e}\) converge to the corresponding
			continuum integral:
			\begin{equation}\label{eq:regular_region_plain_riemann}
				\lim_{r\to\infty }\bigg[
				r^{-2}
				\sum_{\mathbf m/r\in\Omega_{\delta,\mathbf e}}
				F_{\mathbf e}(\mathbf m/r)
				\bigg]
				=
				\int_{\Omega_{\delta,\mathbf e}}
				F_{\mathbf e}(\mathbf z)\,{\rm d}\mathbf z .
			\end{equation}
		\end{itemize}
		
		\noindent
		From the regular-region Riemann-sum convergence~\eqref{eq:regular_region_plain_riemann} and \eqref{today02}, we have, uniformly in
		\(\mathbf e=\mathbf n/|\mathbf n|\),
		\begin{align}\label{oh6}
			r^{2\alpha}
			\sum_{\mathbf m/r\in\Omega_{\delta,\mathbf e}}
			q_{\mathbf m}(1-T(\mathbf n,\mathbf m))
			\longrightarrow
			\int_{\Omega_{\delta,\mathbf e}}F_{\mathbf e}(\mathbf z)\,{\rm d}\mathbf z.
		\end{align}
		Therefore, by using \eqref{first},\eqref{oh7}-\eqref{oh6}, we obtain
		\begin{equation}\label{eq:discrete_decomposition_error}
			\limsup_{r\to\infty}
			\left|
			r^{2\alpha}S_r(\mathbf n)
			-
			\int_{\Omega_{\delta,\mathbf e}}F_{\mathbf e}(\mathbf z)\,{\rm d}\mathbf z
			\right|
			\le
			C\delta^{2-2\alpha}+C\delta^2,
		\end{equation}
		uniformly in \(\mathbf e\in\mathbb S^1\).
		
		\noindent
		We now compare the regular-region integral with the full principal-value
		integral. By definition,
		\[
		\operatorname{p.v.}\int_{|\mathbf z|\le\gamma}F_{\mathbf e}(\mathbf z)\,{\rm d}\mathbf z
		=
		\lim_{\varepsilon\downarrow0}
		\int_{\varepsilon<|\mathbf z|\le\gamma}F_{\mathbf e}(\mathbf z)\,{\rm d}\mathbf z.
		\]
		By again using the decomposition
		\[
		B_\gamma
		=
		\Omega_{\delta,\mathbf e}
		\cup B_\delta({\bf 0})
		\cup B_\delta(\mathbf e)
		\cup B_\delta(-\mathbf e),
		\]
		We thank to ~\eqref{today07}--\eqref{today08} to obtain 
		\begin{equation}\label{eq:continuum_decomposition_error}
			\left|
			\operatorname{p.v.}\int_{|\mathbf z|\le\gamma}F_{\mathbf e}(\mathbf z)\,{\rm d}\mathbf z
			-
			\int_{\Omega_{\delta,\mathbf e}}F_{\mathbf e}(\mathbf z)\,{\rm d}\mathbf z
			\right|
			\le
			C\delta^{2-2\alpha}+C\delta^2.
		\end{equation}
		
		\noindent
		By combining \eqref{eq:discrete_decomposition_error} and
		\eqref{eq:continuum_decomposition_error}, we obtain
		\[
		\limsup_{r\to\infty}
		\left|
		r^{2\alpha}S_r(\mathbf n)
		-
		\operatorname{p.v.}\int_{|\mathbf z|\le\gamma}F_{\mathbf e}(\mathbf z)\,{\rm d}\mathbf z
		\right|
		\le
		C\delta^{2-2\alpha}+C\delta^2,
		\]
		uniformly in \(\mathbf e\in\mathbb S^1\).
		
		\noindent
		Since \(0<\alpha<1\), both exponents \(2-2\alpha\) and \(2\) are positive.
		Therefore, letting \(\delta\downarrow0\), we conclude that
		\begin{align}\label{t12}
			r^{2\alpha}S_r(\mathbf n)
			\longrightarrow
			\operatorname{p.v.}\int_{|\mathbf z|\le\gamma}F_{\mathbf e}(\mathbf z)\,{\rm d}\mathbf z
		\end{align}
		uniformly in the direction \(\mathbf e\in\mathbb S^1\).
		
		\noindent
		Finally, from Lemma~\ref{lem:positivity_j_alpha_4},
		\[
		\operatorname{p.v.}\int_{|\mathbf z|\le\gamma}F_{\mathbf n/|\mathbf n|}(\mathbf z)\,{\rm d}\mathbf z
		=
		j_{\alpha, \gamma},
		\]
		where
		\[
		j_{\alpha, \gamma}
		=
		2\pi\int_1^\gamma\rho^{-1-2\alpha}(1-\rho^{-2})\,{\rm d}\rho
		>0.
		\]
		Hence by using~\eqref{t12}, we obtain
		\[
		r^{2\alpha}S_r(\mathbf n)\to j_{\alpha, \gamma}
		\]
		uniformly in the direction \(\mathbf e\in\mathbb S^1\).
		
		\noindent
		\textbf{Step 2.} Since \(j_{\alpha, \gamma}>0\), there exists \(r_0=r_0(\alpha)\) such that, for all
		\(r\ge r_0\),
		\[
		S_r(\mathbf n)\ge \frac12 j_{\alpha, \gamma}r^{-2\alpha}.
		\]
		For \(1\le r<r_0\), there are only finitely many lattice frequencies \(\mathbf n\).
		Therefore, by decreasing the constant \(c_0>0\) if necessary and increasing
		\(C_0\ge0\), we obtain for all \(\mathbf n\neq{\bf 0}\)
		\[
		S_r(\mathbf n)\ge c_0 r^{-2\alpha}-C_0 r^{-2}.
		\]
		This proves \eqref{eq:ball_sum_gap}.
		
		\noindent
		By combining this with \eqref{t13}, we obtain
		\[
		\sum_{\mathbf m\in\mathcal M_K}q_{\mathbf m}(1-T\bigl(\mathbf n,\mathbf m)\bigr)
		\ge
		c_0 r^{-2\alpha}-C_0 r^{-2}.
		\]
		Since \(r=|\mathbf n|\), this is equivalent to
		\[
		\sum_{\mathbf m\in\mathcal M_K}q_{\mathbf m}\bigl(1-T(\mathbf n,\mathbf m)\bigr)
		\ge
		c_0\langle \mathbf n\rangle^{-2\alpha}
		-
		C_0\langle \mathbf n\rangle^{-2}.
		\]
		This is exactly the
		desired estimate \eqref{eq:finite_symbol_gap}.
		
		\noindent
		\textbf{Step 3.} In this step, we give the proofs of claims~\eqref{today03}-\eqref{eq:regular_region_plain_riemann}.

		\noindent
		\textbf{Step 3({\bf a}): \textbf{proof of estimate~{\eqref{today03}}}.}  In this substep, we consider the region \[|\mathbf m|\le\delta r.\] 
		We put
		\begin{align}\label{today14}
			\mu:=\frac{|\mathbf m|}{|\mathbf n|},
			\qquad
			a:=\cos\varphi:=\frac{\mathbf m\cdot  \mathbf n}{|\mathbf m||\mathbf n|}.\end{align}
		Then, \(0\le \mu\le\delta<1/10\) and \(|{\bf a}|\le1\). By using the similar argument as for the derivation of \eqref{hee}, we may write \eqref{t6} as 
		\[
		T(\mathbf n,\mathbf m)
		=
		\sin^2\varphi
		\left(
		\frac{1}{1+2\mu\cos\varphi+\mu^2}
		+
		\frac{1}{1-2\mu\cos\varphi+\mu^2}
		\right).
		\]
		Now we prove the following claim:
		\begin{align}\label{oh1}
			1-T(\mathbf n,\mathbf m)=P_{\mathbf n}(\mathbf m)+R_{\mathbf n}(\mathbf m),
		\end{align}
		where
		\begin{align}\label{today15}
			P_{\mathbf n}(\mathbf m)
			:=
			2\frac{(\mathbf m\cdot  \mathbf n)^2}{|\mathbf m|^2|\mathbf n|^2}-1
		\end{align}
		and
		\begin{align}\label{oh2}
			|R_{\mathbf n}(\mathbf m)|\le C\frac{|\mathbf m|^2}{|\mathbf n|^2}.
		\end{align}
		Since \(\mu<1\), the modes
		\(\mathbf n+\mathbf m\) and \(\mathbf n-\mathbf m\) are nonzero. Hence no indicator vanishes in this region. We have
		\[
		\frac{|\mathbf m^\perp\cdot  \mathbf n|^2}{|\mathbf m|^2}
		=
		|\mathbf n|^2(1-a^2),
		\]
		and
		\[
		|\mathbf n\pm \mathbf m|^2
		=
		|\mathbf n|^2(1+\mu^2\pm2\mu a).
		\]
		Therefore
		\[
		T(\mathbf n,\mathbf m)
		=
		(1-a^2)
		\left(
		\frac{1}{1+\mu^2+2\mu a}
		+
		\frac{1}{1+\mu^2-2\mu a}
		\right).
		\]
		The two fractions can be combined exactly:
		
		\begin{align}\label{today122}
			&\frac{1}{1+\mu^2+2\mu a}
			+
			\frac{1}{1+\mu^2-2\mu a}
			=
			\frac{2(1+\mu^2)}
			{(1+\mu^2)^2-4\mu^2a^2}.
		\end{align}
		Since \(|{\bf a}|\le1\) and \(\mu\le1/10\),
		\begin{align}\label{today133}
			(1+\mu^2)^2-4\mu^2a^2
			\ge
			1-2\mu^2+\mu^4
			\ge
			\frac12.
		\end{align}
		Moreover, we compute for \eqref{today122} as
		\[
		\begin{aligned}
			&\frac{2(1+\mu^2)}
			{(1+\mu^2)^2-4\mu^2a^2}
			-2
			=
			\frac{
				2(1+\mu^2)-2\big((1+\mu^2)^2-4\mu^2a^2\big)
			}
			{(1+\mu^2)^2-4\mu^2a^2}
			=
			\frac{
				2\mu^2(4a^2-1-\mu^2)
			}
			{(1+\mu^2)^2-4\mu^2a^2}.
		\end{aligned}
		\]
		By~\eqref{today133}, the denominator is bounded below by \(1/2\), while
		\[
		|4a^2-1-\mu^2|\le 5.
		\]
		Hence
		\[
		\left|
		\frac{1}{1+\mu^2+2\mu a}
		+
		\frac{1}{1+\mu^2-2\mu a}
		-2
		\right|
		\le
		C\mu^2.
		\]
		Therefore
		\[
		T(\mathbf n,\mathbf m)
		=
		(1-a^2)\bigl(2+O(\mu^2)\bigr)
		=
		2(1-a^2)+O(\mu^2),
		\]
		because \(0\le1-a^2\le1\). Consequently,
		\[
		1-T(\mathbf n,\mathbf m)
		=
		1-2(1-a^2)+O(\mu^2)
		=
		2a^2-1+O(\mu^2).
		\]
		By~\eqref{today14}, we therefore have
		\[
		1-T(\mathbf n,\mathbf m)
		=
		2\frac{(\mathbf m\cdot  \mathbf n)^2}{|\mathbf m|^2|\mathbf n|^2}
		-
		1
		+
		R_{\mathbf n}(\mathbf m),
		\]
		where
		\[
		|R_{\mathbf n}(\mathbf m)|
		\le
		C\mu^2
		=
		C\frac{|\mathbf m|^2}{|\mathbf n|^2}.
		\]
		The constant \(C\) is universal for \(0<\delta<1/10\). This proves the claim~\eqref{oh1}.

		\noindent
		We next prove the exact cancellation of the principal part \(P_{\mathbf n}\) from \eqref{today15}. Let for any $\rho>0$
		\[
		A_\rho:=\{\mathbf m\in\mathbb Z^2:\ 0<|\mathbf m|\le \rho\}.
		\]
		We claim that
		\begin{align}\label{oh3}
			\sum_{\mathbf m\in A_\rho} q_{\mathbf m}\,P_{\mathbf n}(\mathbf m)=0.
		\end{align}
		It is enough to show that
		\[
		\sum_{\mathbf m\in A_\rho} q_{\mathbf m}\,\frac{(\mathbf m\cdot  \mathbf n)^2}{|\mathbf m|^2|\mathbf n|^2}
		=
		\frac12\sum_{\mathbf m\in A_\rho} q_{\mathbf m}.
		\]
		We write \(\mathbf n=(n_1,n_2)\). By expanding \((\mathbf m\cdot  \mathbf n)^2\), we obtain
		\[
		\frac{(\mathbf m\cdot  \mathbf n)^2}{|\mathbf m|^2|\mathbf n|^2}
		=
		\frac{n_1^2}{|\mathbf n|^2}\frac{m_1^2}{|\mathbf m|^2}
		+
		\frac{2n_1n_2}{|\mathbf n|^2}\frac{m_1m_2}{|\mathbf m|^2}
		+
		\frac{n_2^2}{|\mathbf n|^2}\frac{m_2^2}{|\mathbf m|^2}.
		\]
		Therefore 
		
		\begin{align}\label{today155}
			\sum_{\mathbf m\in A_\rho} q_{\mathbf m}\,\frac{(\mathbf m\cdot  \mathbf n)^2}{|\mathbf m|^2|\mathbf n|^2}
			&=
			\frac{n_1^2}{|\mathbf n|^2}
			\sum_{\mathbf m\in A_\rho} q_{\mathbf m}\,\frac{m_1^2}{|\mathbf m|^2}
			+
			\frac{2n_1n_2}{|\mathbf n|^2}
			\sum_{\mathbf m\in A_\rho} q_{\mathbf m}\,\frac{m_1m_2}{|\mathbf m|^2}
			\\
			&\quad
			+
			\frac{n_2^2}{|\mathbf n|^2}
			\sum_{\mathbf m\in A_\rho} q_{\mathbf m}\,\frac{m_2^2}{|\mathbf m|^2}.
		\end{align}

		\noindent Now the mixed term vanishes by symmetry. Indeed, \(A_\rho\) is invariant under the
		reflection \((m_1,m_2)\mapsto(-m_1,m_2)\), while \(q_{\mathbf m}\) is radial and hence
		unchanged by this reflection. Consequently,
		\[
		\sum_{\mathbf m\in A_\rho} q_{\mathbf m}\,\frac{m_1m_2}{|\mathbf m|^2}=0.
		\]
		Next, the two diagonal terms are equal by the symmetry
		\((m_1,m_2)\mapsto(m_2,m_1)\), again because \(A_\rho\) is invariant and \(q_{\mathbf m}\)
		depends only on \(|\mathbf m|\). Hence
		\[
		\sum_{\mathbf m\in A_\rho} q_{\mathbf m}\,\frac{m_1^2}{|\mathbf m|^2}
		=
		\sum_{\mathbf m\in A_\rho} q_{\mathbf m}\,\frac{m_2^2}{|\mathbf m|^2}.
		\]
		By summing over \(\mathbf m\in A_\rho\), we get
		\[
		\sum_{\mathbf m\in A_\rho} q_{\mathbf m}\,\frac{m_1^2}{|\mathbf m|^2}
		+
		\sum_{\mathbf m\in A_\rho} q_{\mathbf m}\,\frac{m_2^2}{|\mathbf m|^2}
		=
		\sum_{\mathbf m\in A_\rho} q_{\mathbf m}.
		\]
		Since the two terms on the left are equal each of them must be half of the total:
		\[
		\sum_{\mathbf m\in A_\rho} q_{\mathbf m}\,\frac{m_1^2}{|\mathbf m|^2}
		=
		\sum_{\mathbf m\in A_\rho} q_{\mathbf m}\,\frac{m_2^2}{|\mathbf m|^2}
		=
		\frac12\sum_{\mathbf m\in A_\rho} q_{\mathbf m}.
		\]
		By substituting these identities into \eqref{today155} yields
		\[
		\sum_{\mathbf m\in A_\rho} q_{\mathbf m}\,\frac{(\mathbf m\cdot  \mathbf n)^2}{|\mathbf m|^2|\mathbf n|^2}
		=
		\frac{n_1^2+n_2^2}{|\mathbf n|^2}\cdot  \frac12\sum_{\mathbf m\in A_\rho} q_{\mathbf m}
		=
		\frac12\sum_{\mathbf m\in A_\rho} q_{\mathbf m}.
		\]
		Therefore
		\[
		\sum_{\mathbf m\in A_\rho} q_{\mathbf m}\,P_{\mathbf n}(\mathbf m)
		=
		2\sum_{\mathbf m\in A_\rho} q_{\mathbf m}\,\frac{(\mathbf m\cdot  \mathbf n)^2}{|\mathbf m|^2|\mathbf n|^2}
		-
		\sum_{\mathbf m\in A_\rho} q_{\mathbf m}
		=
		0,
		\]
		which proves the exact cancellation~\eqref{oh3}.
		
		\noindent
		By taking \(\rho=\delta r\) in \eqref{oh3} and using~\eqref{oh1} we obtain
		\[
		\sum_{0<|\mathbf m|\le\delta r}q_{\mathbf m}(1-T(\mathbf n,\mathbf m))
		=
		\sum_{0<|\mathbf m|\le\delta r}q_{\mathbf m}R_{\mathbf n}(\mathbf m).
		\]
		Hence, by the bound~\eqref{oh2} on \(R_{\mathbf n}\),
		\begin{align}\label{oh4}
			\left|
			\sum_{0<|\mathbf m|\le\delta r}q_{\mathbf m}(1-T(\mathbf n,\mathbf m))
			\right|
			\le
			\frac{C}{r^2}
			\sum_{0<|\mathbf m|\le\delta r}q_{\mathbf m}|\mathbf m|^2.
		\end{align}
		Now we estimate the right hand side of \eqref{oh4}. Since
		\[
		q_{\mathbf m}=\langle \mathbf m\rangle^{-2-2\alpha}=(1+|\mathbf m|^2)^{-1-\alpha},
		\]
		we have, for every \(\mathbf {\bf m}\neq{\bf 0}\),
		\[
		q_{\mathbf m}|\mathbf m|^2
		=
		|\mathbf m|^2(1+|\mathbf m|^2)^{-1-\alpha}
		\le
		|\mathbf m|^2 |\mathbf m|^{-2-2\alpha}
		=
		|\mathbf m|^{-2\alpha}.
		\]
		Consequently,
		\begin{align}\label{today166}
			\left|
			\sum_{0<|\mathbf m|\le\delta r}q_{\mathbf m}(1-T(\mathbf n,\mathbf m))
			\right|
			\le
			\frac{C}{r^2}
			\sum_{0<|\mathbf m|\le\delta r}|\mathbf m|^{-2\alpha}.
		\end{align}
		It remains to estimate the lattice sum. Since \(0<\alpha<1\), we have
		\[
		\sum_{0<|\mathbf m|\le \delta r}|\mathbf m|^{-2\alpha}
		\le
		C (\delta r)^{2-2\alpha}.
		\]
		Indeed, decomposing into unit annuli,
		\[
		\{\mathbf m\in\mathbb Z^2:\ j\le |\mathbf m|<j+1\},
		\qquad j=1,2,\dots,\lfloor \delta r\rfloor,
		\]
		and using that each annulus contains at most \(Cj\) lattice points, we get
		\[
		\sum_{0<|\mathbf m|\le \delta r}|\mathbf m|^{-2\alpha}
		\le
		C\sum_{j=1}^{\lfloor \delta r\rfloor} j\cdot  j^{-2\alpha}
		=
		C\sum_{j=1}^{\lfloor \delta r\rfloor} j^{1-2\alpha}
		\le
		C (\delta r)^{2-2\alpha}.
		\]
		If \(0<\delta r<1\), the sum is empty, so the same bound is trivial after increasing
		\(C\). 
		Therefore  we further estimate \eqref{today166} as follows,
		\[
		\left|
		\sum_{0<|\mathbf m|\le\delta r}q_{\mathbf m}\bigl(1-T(\mathbf n,\mathbf m)\bigr)
		\right|
		\le
		C r^{-2}(\delta r)^{2-2\alpha}
		=
		C\delta^{2-2\alpha}r^{-2\alpha}.
		\]
		Equivalently, by using $A_0^r\subset\{\mathbf m:\ 0<|\mathbf m|\le\delta r\},$ we get~\eqref{today03}.
		
		\noindent
		\textbf{3(b):
			\textbf{proof of the estimate~\eqref{today04}.}}
		We recall that \(r=|\mathbf n|\). In this substep, we  consider first the region
		\[
		|\mathbf m-\mathbf n|\le\delta r.
		\]
		Since \(\delta<1/10\), we have
		\[
		|\mathbf m|\ge |\mathbf n|-|\mathbf m-\mathbf n|\ge (1-\delta)r\ge \frac9{10}r,
		\]
		and
		\[
		|\mathbf m|\le |\mathbf n|+|\mathbf m-\mathbf n|\le (1+\delta)r\le \frac{11}{10}r.
		\]
		Thus \(|\mathbf m|\simeq r\) in this region.
		
		We show that \(T(\mathbf n,\mathbf m)\) is uniformly bounded there. The singular term in \eqref{t6} is the one
		with denominator \(|\mathbf m-\mathbf n|^2\). Since
		\[
		\mathbf m^\perp\cdot  \mathbf n
		=
		\mathbf m^\perp\cdot (\mathbf n-\mathbf m),
		\]
		we get
		\[
		\frac{|\mathbf m^\perp\cdot  \mathbf n|^2}{|\mathbf m|^2|\mathbf m-\mathbf n|^2}
		=
		\frac{|\mathbf m^\perp\cdot (\mathbf n-\mathbf m)|^2}{|\mathbf m|^2|\mathbf m-\mathbf n|^2}
		\le 1.
		\]
		We use
		\[
		|\mathbf m+\mathbf n|
		\ge |2\mathbf n|-|\mathbf m-\mathbf n|
		\ge 2r-\delta r
		\ge \frac{19}{10}r.
		\]
		Therefore
		\[
		\frac{|\mathbf m^\perp\cdot  \mathbf n|^2}{|\mathbf m|^2|\mathbf m+\mathbf n|^2}
		\le
		\frac{|\mathbf m|^2 r^2}{|\mathbf m|^2 |\mathbf m+\mathbf n|^2}
		\le
		C.
		\]
		Hence
		\[
		T(\mathbf n,\mathbf m)\le C
		\qquad
		\text{whenever } |\mathbf m-\mathbf n|\le\delta r.
		\]
		Consequently,
		\[
		|1-T(\mathbf n,\mathbf m)|\le C.
		\]
		Moreover, since \(|\mathbf m|\simeq r\),
		\[
		q_{\mathbf m}=\langle \mathbf m\rangle^{-2-2\alpha}\le C r^{-2-2\alpha}.
		\]
		Therefore
		\[
		q_{\mathbf m}|1-T(\mathbf n,\mathbf m)|
		\le
		C r^{-2-2\alpha}
		\qquad
		\text{for } |\mathbf m-\mathbf n|\le\delta r.
		\]
		The number of lattice points in the disk \(\{\mathbf m\in\mathbb Z^2:\ |\mathbf m-\mathbf n|\le\delta r\}\)
		is bounded by
		\[
		C(\delta r)^2+C.
		\]
		Indeed, the disk can be covered by \(C\bigl((\delta r)^2+1\bigr)\) unit lattice squares.
		Thus
		\[
		\begin{aligned}
			\left|
			\sum_{|\mathbf m-\mathbf n|\le\delta r}q_{\mathbf m}\bigl(1-T(\mathbf n,\mathbf m)\bigr)
			\right|
			&\le
			C r^{-2-2\alpha}\bigl(\delta^2 r^2+1\bigr)
			\\
			&=
			C\delta^2 r^{-2\alpha}
			+
			C r^{-2-2\alpha}.
		\end{aligned}
		\]
		Multiplying by \(r^{2\alpha}\), we obtain
		\begin{align}\label{today0909}
			r^{2\alpha}
			\left|
			\sum_{|\mathbf m-\mathbf n|\le\delta r}q_{\mathbf m}\bigl(1-T(\mathbf n,\mathbf m)\bigr)
			\right|
			\le
			C\delta^2+Cr^{-2}.
		\end{align}
		The region \(|\mathbf m+\mathbf n|\le\delta r\) is treated in exactly the same way. There
		\(|\mathbf m|\simeq r\), the singular term is controlled by
		\[
		\mathbf m^\perp\cdot  \mathbf n=\mathbf m^\perp\cdot (\mathbf n+\mathbf m),
		\]
		because \(\mathbf m^\perp\cdot  \mathbf m=0\), and the other denominator is bounded below by
		\[
		|\mathbf m-\mathbf n|\ge 2r-\delta r\ge \frac{19}{10}r.
		\]
		Thus
		\begin{align}\label{9191}
			r^{2\alpha}
			\left|
			\sum_{|\mathbf m+\mathbf n|\le\delta r}q_{\mathbf m}\bigl(1-T(\mathbf n,\mathbf m)\bigr)
			\right|
			\le
			C\delta^2+Cr^{-2}.
		\end{align}
		By combining these two estimates \eqref{today0909}, \eqref{9191} with \[
		A_+^r\subset\{\mathbf m:\ |\mathbf m-\mathbf n|\le\delta r\},
		\]
		and
		\[
		A_-^r\subset\{\mathbf m:\ |\mathbf m+\mathbf n|\le\delta r\},
		\] 
		we get~\eqref{today04}.
		
		\medskip
		\noindent
		\textbf{Step 3(c): proof of the estimate~\eqref{today07}.}  In this substep, for \(|\mathbf z|\le\delta\) and \(\mathbf e\in\mathbb{S}^1\), we write
		\[
		\rho:=|\mathbf z|,
		\qquad
		\cos\varphi:=\frac{\mathbf z\cdot  \mathbf e}{|\mathbf z|}.
		\]
		Then
		\[
		\frac{|\mathbf z^\perp\cdot  \mathbf e|^2}{|\mathbf z|^2}
		=
		\sin^2\varphi.
		\]
		Moreover,
		\[
		|\mathbf z\pm \mathbf e|^2
		=
		1+\rho^2\pm 2\rho\cos\varphi.
		\]
		Hence from \eqref{t8}, we get
		\[
		\tau_{\mathbf e}(\mathbf z)
		=
		\sin^2\varphi
		\left(
		\frac{1}{1+\rho^2+2\rho\cos\varphi}
		+
		\frac{1}{1+\rho^2-2\rho\cos\varphi}
		\right).
		\]
		By using the similar arguments as in step 3({\bf a}), we obtain
		\[
		\left|
		\frac{1}{1+\rho^2+2\rho\cos\varphi}
		+
		\frac{1}{1+\rho^2-2\rho\cos\varphi}
		-2
		\right|
		\le
		C\rho^2,
		\]
		with a constant \(C\) independent of \(\varphi\), \(\mathbf e\), and \(0<\delta<1/10\). In
		other words,
		\[
		\frac{1}{1+\rho^2+2\rho\cos\varphi}
		+
		\frac{1}{1+\rho^2-2\rho\cos\varphi}
		=
		2+O(\rho^2).
		\]
		Hence
		\[
		\tau_{\mathbf e}(\mathbf z)
		=
		2\sin^2\varphi+O(\rho^2),
		\]
		and therefore
		\[
		1-\tau_{\mathbf e}(\mathbf z)
		=
		1-2\sin^2\varphi+O(\rho^2)
		=
		2\cos^2\varphi-1+O(\rho^2).
		\]
		Define
		\[
		P_{\mathbf e}(\mathbf z):=
		2\frac{(\mathbf z\cdot  \mathbf e)^2}{|\mathbf z|^2}-1
		=
		2\cos^2\varphi-1.
		\]
		Then
		\begin{align}\label{t10}
			1-\tau_{\mathbf e}(\mathbf z)=P_{\mathbf e}(\mathbf z)+R_{\mathbf e}(\mathbf z),
			\qquad
			|R_{\mathbf e}(\mathbf z)|\le C|\mathbf z|^2,
		\end{align}
		for all \(0<|\mathbf z|\le\delta\), with \(C\) independent of \(\mathbf e\).
		
		We now show that the principal part has zero principal-value integral. Choose a
		rotation \(R\in SO(2)\) with \(R\mathbf e_1=\mathbf e\). By using the change of variables
		\(\mathbf z=R\mathbf y\), the ball \(\{|\mathbf z|\le\delta\}\), the measure \({\rm d}\mathbf z\), and the weight
		\(|\mathbf z|^{-2-2\alpha}\) are unchanged. Thus it is enough to take \(\mathbf e=\mathbf e_1\). In
		polar coordinates ({\em i.e.,} \(\mathbf z=\rho\mathbf e_\varphi\)),
		\[
		P_{\mathbf e_1}(\rho \mathbf e_\varphi)
		=
		2\cos^2\varphi-1
		=
		\cos(2\varphi).
		\]
		Therefore
		\[
		\int_0^{2\pi}P_{\mathbf e_1}(\rho \mathbf e_\varphi)\,{\rm d}\varphi
		=
		\int_0^{2\pi}\cos(2\varphi)\,{\rm d}\varphi
		=
		0.
		\]
		Consequently, for every \(\varepsilon\in(0,\delta)\),
		\[
		\int_{\varepsilon<|\mathbf z|\le\delta}
		|\mathbf z|^{-2-2\alpha}P_{\mathbf e}(\mathbf z)\,{\rm d}\mathbf z
		=
		0.
		\]
		Hence
		\[
		\operatorname{p.v.}\int_{|\mathbf z|\le\delta}
		|\mathbf z|^{-2-2\alpha}P_{\mathbf e}(\mathbf z)\,{\rm d}\mathbf z=0.
		\]
		By recalling ~\eqref{t9} and using the remainder estimate~\eqref{t10}, we obtain
		\[
		\begin{aligned}
			\left|
			\operatorname{p.v.}\int_{|\mathbf z|\le\delta}F_{\mathbf e}(\mathbf z)\,{\rm d}\mathbf z
			\right|
			&=
			\left|
			\operatorname{p.v.}\int_{|\mathbf z|\le\delta}
			|\mathbf z|^{-2-2\alpha}(P_{\mathbf e}(\mathbf z)+R_{\mathbf e}(\mathbf z))\,{\rm d}\mathbf z
			\right|
			\\
			&\le
			C\int_{|\mathbf z|\le\delta}|\mathbf z|^{-2-2\alpha}|\mathbf z|^2\,{\rm d}\mathbf z
			\\
			&=
			C\int_0^\delta \rho^{-2-2\alpha}\rho^2 \rho\,{\rm d}\rho
			\\
			&=
			C\int_0^\delta \rho^{1-2\alpha}\,{\rm d}\rho
			\\
			&=
			C\delta^{2-2\alpha},
		\end{aligned}
		\]
		because \(0<\alpha<1\). Thus we conclude the esitmate~\eqref{today07}.
		
		\medskip
		\noindent
		\textbf{Step 3(d): proof of the estimate~\eqref{today08}.} In this substep, we next the estimate~\eqref{today08} the contribution from the balls \(B_\delta(\mathbf e)\) and
		\(B_\delta(-\mathbf e)\). We prove the estimate near \(\mathbf e\); the estimate near \(-\mathbf e\) is
		identical.
		
		\noindent
		Assume \(|\mathbf z-\mathbf e|\le\delta<1/10\). Then
		\[
		|\mathbf z|\ge |\mathbf e|-|\mathbf z-\mathbf e|\ge 1-\delta\ge \frac9{10},
		\]
		and
		\[
		|\mathbf z|\le 1+\delta\le \frac{11}{10}.
		\]
		Thus \(|\mathbf z|\simeq1\) uniformly in \(\mathbf e\).
		
		\noindent
		The potentially singular part of \(\tau_{\mathbf e}\) near \(\mathbf z=\mathbf e\) is the term with
		\(|\mathbf z-\mathbf e|^{-2}\). Since
		\[
		\mathbf z^\perp\cdot  \mathbf e
		=
		\mathbf z^\perp\cdot (\mathbf e-\mathbf z),
		\]
		we have
		\[
		\frac{|\mathbf z^\perp\cdot  \mathbf e|^2}{|\mathbf z|^2|\mathbf z-\mathbf e|^2}
		=
		\frac{|\mathbf z^\perp\cdot (\mathbf e-\mathbf z)|^2}{|\mathbf z|^2|\mathbf z-\mathbf e|^2}
		\le
		\frac{|\mathbf z|^2|\mathbf e-\mathbf z|^2}{|\mathbf z|^2|\mathbf z-\mathbf e|^2}
		=
		1.
		\]
		The other term is clearly bounded, because
		\[
		|\mathbf z+\mathbf e|\ge |\mathbf e+\mathbf e|-|\mathbf z-\mathbf e|\ge 2-\delta\ge \frac{19}{10}.
		\]
		Therefore from~\eqref{t8}, we get
		\[
		|\tau_{\mathbf e}(\mathbf z)|\le C
		\qquad
		\text{for } |\mathbf z-\mathbf e|\le\delta,
		\]
		with \(C\) independent of \(\mathbf e\). Since \(|\mathbf z|\simeq1\) in this region, from~\eqref{t9}, we also
		have
		\[
		|F_{\mathbf e}(\mathbf z)|
		=
		|\mathbf z|^{-2-2\alpha}|1-\tau_{\mathbf e}(\mathbf z)|
		\le C.
		\]
		Thus
		\[
		\left|
		\int_{|\mathbf z-\mathbf e|\le\delta}F_{\mathbf e}(\mathbf z)\,{\rm d}\mathbf z
		\right|
		\le
		C |B_\delta(\mathbf e)|
		\le
		C\delta^2.
		\]
		The same proof near \(-\mathbf e\) gives
		\[
		\left|
		\int_{|\mathbf z+\mathbf e|\le\delta}F_{\mathbf e}(\mathbf z)\,{\rm d}\mathbf z
		\right|
		\le
		C\delta^2.
		\]
		Therefore, we obtain~\eqref{today08}.
		
		\medskip
		\noindent
		\textbf{Step 3(e): proof of the identity~\eqref{today02}.} We consider \({\bf z}=\mathbf m/r\in\Omega_{\delta,\mathbf e}\). Clearly,
		\[
		\delta\le |\mathbf z|\le\gamma,
		\]
		so \(|\mathbf m|=r|\mathbf z|\simeq r\). Moreover, from~\eqref{t6}, we get for \(\mathbf e=\mathbf n/|\mathbf n|\),
		\begin{align}\label{today300}
			T(\mathbf n,\mathbf m)=\tau_{\mathbf e}(\mathbf z).
		\end{align}
		Next, since
		\[
		q_{\mathbf m}=\langle \mathbf m\rangle^{-2-2\alpha}=(1+|\mathbf m|^2)^{-1-\alpha},
		\]
		we have
		\[
		r^{2+2\alpha}q_{\mathbf m}
		=
		r^{2+2\alpha}(1+r^2|\mathbf z|^2)^{-1-\alpha}
		=
		\left(|\mathbf z|^2+r^{-2}\right)^{-1-\alpha}.
		\]
		Because \(|\mathbf z|\ge\delta\), it follows that
		\begin{align}\label{t11}
			\sup_{\substack{\mathbf e\in\mathbb S^1\\ \mathbf z\in\Omega_{\delta,\mathbf e}}}
			\left|
			\left(|\mathbf z|^2+r^{-2}\right)^{-1-\alpha}
			-
			|\mathbf z|^{-2-2\alpha}
			\right|
			\longrightarrow0
			\qquad\text{as }r\to\infty.
		\end{align}
		Also, since \(F_{\mathbf e}\) is uniformly bounded on \(\Omega_{\delta,\mathbf e}\), the quantity
		\[
		1-\tau_{\mathbf e}(\mathbf z)
		\]
		is uniformly bounded there. Therefore, by using~\eqref{today300} and~\eqref{t9}, we obtain
		\[
		\begin{aligned}
			&\sup_{\substack{\mathbf e\in\mathbb S^1}}\Bigg|
			r^{2\alpha}
			\sum_{\mathbf m/r\in\Omega_{\delta,\mathbf e}}
			q_{\mathbf m}\bigl(1-T(\mathbf n,\mathbf m)\bigr)
			-
			r^{-2}
			\sum_{\mathbf m/r\in\Omega_{\delta,\mathbf e}}F_{\mathbf e}(\mathbf m/r)
			\Bigg|
			\\
			&\qquad=
			\sup_{\substack{\mathbf e\in\mathbb S^1}}\Bigg|
			r^{-2}
			\sum_{\mathbf m/r\in\Omega_{\delta,\mathbf e}}
			\left[
			r^{2+2\alpha}q_{\mathbf m}
			-
			|\mathbf m/r|^{-2-2\alpha}
			\right]
			\bigl(1-\tau_{\mathbf e}(\mathbf m/r)\bigr)
			\Bigg|
			\\
			&\qquad\le
			C
			\sup_{\substack{\mathbf e\in\mathbb S^1\\ \mathbf z\in\Omega_{\delta,\mathbf e}}}
			\left|
			\left(|\mathbf z|^2+r^{-2}\right)^{-1-\alpha}
			-
			|\mathbf z|^{-2-2\alpha}
			\right|
			\cdot 
			r^{-2}\#\{\mathbf m:\ \mathbf m/r\in B_\gamma\}.
		\end{aligned}
		\]
		Since
		\[
		r^{-2}\#\{\mathbf m:\ \mathbf m/r\in B_\gamma\}\le C,
		\]
		the right-hand side tends to \(0\), uniformly in \(\mathbf e\in\mathbb S^1\), by \eqref{t11}. Hence we obtain~\eqref{today02}.

		\noindent
		\textbf{Step 3(f): proof of the identity~{\eqref{eq:regular_region_plain_riemann}}.}
		On \(\Omega_{\delta,\mathbf e}\), we have
		\[
		|\mathbf z|\ge \delta,
		\qquad
		|\mathbf z-\mathbf e|\ge\delta,
		\qquad
		|\mathbf z+\mathbf e|\ge\delta,
		\qquad
		|\mathbf z|\le \gamma.
		\]
		Therefore every denominator appearing in
		\[
		F_{\mathbf e}(\mathbf z)
		=
		|\mathbf z|^{-2-2\alpha}
		\left[
		1-
		\frac{|\mathbf z^\perp\cdot  \mathbf e|^2}{|\mathbf z|^2}
		\left(
		\frac{1}{|\mathbf z+\mathbf e|^2}
		+
		\frac{1}{|\mathbf z-\mathbf e|^2}
		\right)
		\right]
		\]
		is bounded away from zero by a constant depending only on \(\delta\). Moreover
		\(|\mathbf e|=1\). Differentiating the expression for \(F_{\mathbf e}\) with respect to \(\mathbf z\), each
		term is a finite sum of products of powers of
		\[
		|\mathbf z|^{-1},
		\qquad
		|\mathbf z-\mathbf e|^{-1},
		\qquad
		|\mathbf z+\mathbf e|^{-1},
		\]
		and bounded linear factors involving \(\mathbf z\) and \(\mathbf e\). Hence there exists a
		constant \(C_\delta>0\), independent of \(\mathbf e\in\mathbb S^1\), such that
		\[
		\|F_{\mathbf e}\|_{C^1(\Omega_{\delta,\mathbf e})}\le C_\delta.
		\]
		
		\medskip
		\noindent
		
		\noindent
		We now apply Lemma~\ref{Riemann-sum} to
		\[
		D_{\mathbf e}=\Omega_{\delta,\mathbf e},
		\qquad
		G_{\mathbf e}=F_{\mathbf e}.
		\]
		The boundary of $\Omega_{\delta,{\bf e}}$ is contained in the union of four
		circles:
		\[
		\partial B_\gamma,\qquad
		\partial B_\delta({\bf 0}),\qquad
		\partial B_\delta(\mathbf e),\qquad
		\partial B_\delta(-\mathbf e).
		\]
		Hence its total boundary length is bounded uniformly in \(\mathbf e\). Consequently,
		the boundary-layer estimate \eqref{eq:uniform_boundaryayer} holds uniformly in
		\(\mathbf e\). Therefore \eqref{eq:uniform_Riemann_sum_general} gives~\eqref{eq:regular_region_plain_riemann}.
		
	\end{proof}
	In the following theorem, we establish the main coercivity estimate. Its proof relies on the preceding lemma as a key structural input, combined with further Fourier-analytic arguments.
	
	\begin{theorem}[coercivity estimate]
		\label{prop:finite_dimensional_coercivity}
		Fix \(0<\alpha<1\). Let \(N,K\in\mathbb N\) satisfy $K\ge \gamma\sqrt2\,N$, with $\gamma>1+\sqrt{2}$. There exist constants
		\(c_0>0\) and \(C_0\ge0\), depending only on \(\alpha\), such that for every
		\(\omega\in\mathbb{V}_N\),
		\begin{equation}\label{eq:finite_dimensional_coercivity}
			\mathfrak D_{N,K}(\omega)
			-
			c_K\|\omega\|_{\mathbb{L}^2(\T^2)}^2
			\le
			-c_0\|\omega\|_{\mathbb{H}^{-\alpha}(\T^2)}^2
			+
			C_0\|\omega\|_{\mathbb{H}^{-1}(\T^2)}^2 .
		\end{equation}
		
	\end{theorem}
	
	\begin{proof}
		We divide the proof into several steps.
		
		\medskip
		\noindent\textbf{Step 1: reduction to the unprojected quadratic form.} Since \(\mathcal P_N\) is the orthogonal Fourier projector onto the modes
		\(\{|{\bf m}|_\infty\le N\}\), it is a Fourier multiplier. The operator
		\((-\Delta)^{-1/2}\) is also a Fourier multiplier. Hence the two operators commute:
		\[
		(-\Delta)^{-1/2}\mathcal P_N f
		=
		\mathcal P_N(-\Delta)^{-1/2}f.
		\]
		Moreover, \(\mathcal P_N\) is an orthogonal projection on \(\mathbb{L}^2(\T^2)\), so
		\[
		\|\mathcal P_N g\|_{ \mathbb{L}^2(\T^2)}\le \|g\|_{ \mathbb{L}^2(\T^2)}.
		\]
		Therefore, for every \(k\in\mathcal I^K\),
		\begin{align*}
			\Big\|
			(-\Delta)^{-1/2}
			\mathcal P_N\big[(\boldsymbol{\sigma}_k\cdot \nabla)\omega\big]
			\Big\|_{ \mathbb{L}^2(\T^2)}
			&=
			\Big\|
			\mathcal P_N
			(-\Delta)^{-1/2}\big[(\boldsymbol{\sigma}_k\cdot \nabla)\omega\big]
			\Big\|_{ \mathbb{L}^2(\T^2)}\\
			&\le
			\Big\|
			(-\Delta)^{-1/2}\big[(\boldsymbol{\sigma}_k\cdot \nabla)\omega\big]
			\Big\|_{ \mathbb{L}^2(\T^2)}.
		\end{align*}
		By summing over \(k\) and recalling Definition~\ref{def:DLK}, we obtain
		\[
		\mathfrak D_{N,K}(\omega)
		\le
		\widetilde{\mathfrak D}_{K}(\omega),
		\]
		where
		\begin{align}\label{today345}
			\widetilde{\mathfrak D}_{K}(\omega)
			:=
			\sum_{k\in\mathcal I^K}
			\Big\|
			(-\Delta)^{-1/2}
			(\boldsymbol{\sigma}_k\cdot \nabla)\omega
			\Big\|_{\mathbb{L}^2(\T^2)}^2 .
		\end{align}
		Thus it is enough to estimate the unprojected quadratic form
		\(\widetilde{\mathfrak D}_{K}\).
		
		\medskip
		\noindent\textbf{Step 2: Fourier expansion of \(\omega\).}
		Since \(\omega\in \mathbb{V}_N\) has zero mean, it has the
		finite Fourier expansion
		\begin{align}\label{today55}
			\omega(\mathbf x)=\sum_{{\bf 0}\neq \mathbf n\in\mathcal K_N}\widehat\omega(\mathbf n)e^{i\mathbf n\cdot \mathbf x}.
		\end{align}

		\medskip
		\noindent\textbf{Step 3: action of one noise mode on one Fourier mode.} We recall the definition of \eqref{today07} and \eqref{eq:noise_ modes}. Fix \({\bf 0}\neq \mathbf n\in\mathcal K_N\). Since
		\[
		\nabla\big(e^{i\mathbf n\cdot \mathbf x}\big)=i\mathbf n\,e^{i\mathbf n\cdot \mathbf x},
		\]
		we get
		\[
		(\boldsymbol{\sigma}_{\mathbf m,c}\cdot \nabla)\big(\widehat\omega(\mathbf n)e^{i\mathbf n\cdot \mathbf x}\big)
		=
		i\sqrt{2q(\mathbf m)}\,(\mathbf e_{\mathbf m}\cdot \mathbf n)\widehat\omega(\mathbf n)\,
		\cos(\mathbf m\cdot \mathbf x)\,e^{i\mathbf n\cdot \mathbf x}.
		\]
		By using
		\[
		\cos(\mathbf m\cdot \mathbf x)=\frac12\Big(e^{i\mathbf m\cdot \mathbf x}+e^{-i\mathbf m\cdot \mathbf x}\Big),
		\]
		this becomes
		\[
		(\boldsymbol{\sigma}_{\mathbf m,c}\cdot \nabla)\big(\widehat\omega(\mathbf n)e^{i\mathbf n\cdot \mathbf x}\big)
		=
		\frac{i\sqrt{2q(\mathbf m)}}{2}
		(\mathbf e_{\mathbf m}\cdot \mathbf n)\widehat\omega(\mathbf n)
		\Big(
		e^{i(\mathbf n+\mathbf m)\cdot \mathbf x}+e^{i(\mathbf n-\mathbf m)\cdot \mathbf x}
		\Big).
		\]
		Similarly, since
		\[
		\sin(\mathbf m\cdot \mathbf x)=\frac{1}{2i}
		\Big(e^{i\mathbf m\cdot \mathbf x}-e^{-i\mathbf m\cdot \mathbf x}\Big),
		\]
		we obtain
		\[
		(\boldsymbol{\sigma}_{\mathbf m,s}\cdot \nabla)\big(\widehat\omega(\mathbf n)e^{i\mathbf n\cdot \mathbf x}\big)
		=
		\frac{\sqrt{2q(\mathbf m)}}{2}
		(\mathbf e_{\mathbf m}\cdot \mathbf n)\widehat\omega(\mathbf n)
		\Big(
		e^{i(\mathbf n+\mathbf m)\cdot \mathbf x}-e^{i(\mathbf n-\mathbf m)\cdot \mathbf x}
		\Big).
		\]
		
		\medskip
		\noindent\textbf{Step 4: Fourier coefficients of \((\boldsymbol{\sigma}_{\mathbf m,c}\cdot \nabla)\omega\) and \((\boldsymbol{\sigma}_{\mathbf m,s}\cdot \nabla)\omega\).}
		We now pass from the single-mode formulas of Step 3 to the full Fourier expansion of
		\(\omega\). By linearity,
		\[
		(\boldsymbol{\sigma}_{\mathbf m,c}\cdot \nabla)\omega
		=
		\sum_{{\bf 0}\neq \mathbf n\in\mathcal K_N}
		(\boldsymbol{\sigma}_{\mathbf m,c}\cdot \nabla)
		\big(\widehat\omega(\mathbf n)e^{i\mathbf n\cdot \mathbf x}\big),
		\]
		and similarly
		\[
		(\boldsymbol{\sigma}_{\mathbf m,s}\cdot \nabla)\omega
		=
		\sum_{{\bf 0}\neq \mathbf n\in\mathcal K_N}
		(\boldsymbol{\sigma}_{\mathbf m,s}\cdot \nabla)
		\big(\widehat\omega(\mathbf n)e^{i\mathbf n\cdot \mathbf x}\big).
		\]
		We obtain
		\[
		\begin{aligned}
			(\boldsymbol{\sigma}_{\mathbf m,c}\cdot \nabla)\omega
			&=
			\frac{i\sqrt{2q(\mathbf m)}}{2}
			\sum_{{\bf 0}\neq \mathbf n\in\mathcal K_N}
			(\mathbf e_{\mathbf m}\cdot \mathbf n)\widehat\omega(\mathbf n)
			\Big(
			e^{i(\mathbf n+\mathbf m)\cdot \mathbf x}+e^{i(\mathbf n-\mathbf m)\cdot \mathbf x}
			\Big),
			\\[1mm]
			(\boldsymbol{\sigma}_{\mathbf m,s}\cdot \nabla)\omega
			&=
			\frac{\sqrt{2q(\mathbf m)}}{2}
			\sum_{{\bf 0}\neq \mathbf n\in\mathcal K_N}
			(\mathbf e_{\mathbf m}\cdot \mathbf n)\widehat\omega(\mathbf n)
			\Big(
			e^{i(\mathbf n+\mathbf m)\cdot \mathbf x}-e^{i(\mathbf n-\mathbf m)\cdot \mathbf x}
			\Big).
		\end{aligned}
		\]
		We now compute the coefficient of the Fourier mode \(e^{i\mathbf r\cdot \mathbf x}\), where
		\(\mathbf r\in\mathbb Z^2\). In the cosine case, the mode \(e^{i\mathbf r\cdot \mathbf x}\) can arise
		in exactly two ways:
		\begin{itemize}
			\item from the term \(e^{i(\mathbf n+\mathbf m)\cdot \mathbf x}\), which requires \(\mathbf r=\mathbf n+\mathbf m\), {\em i.e.,}
			\(\mathbf n=\mathbf r-\mathbf m\);
			\item from the term \(e^{i(\mathbf n-\mathbf m)\cdot \mathbf x}\), which requires \(\mathbf r=\mathbf n-\mathbf m\), {\em i.e.,}
			\(\mathbf n=\mathbf r+\mathbf m\).
		\end{itemize}
		Therefore the \(\mathbf r\)-th Fourier coefficient of
		\((\boldsymbol{\sigma}_{\mathbf m,c}\cdot \nabla)\omega\) is
		
		\[
		\widehat{(\boldsymbol{\sigma}_{\mathbf m,c}\cdot \nabla)\omega}(\mathbf r)
		=
		i\sqrt{\frac{q(\mathbf m)}{2}}
		\Big(
		(\mathbf e_{\mathbf m}\cdot (\mathbf r-\mathbf m))\,\widehat\omega(\mathbf r-\mathbf m)
		+
		(\mathbf e_{\mathbf m}\cdot (\mathbf r+\mathbf m))\,\widehat\omega(\mathbf r+\mathbf m)
		\Big).
		\]
		The sine case is identica except that the second term carries a minus sign.
		Indeed, again the mode \(e^{i\mathbf r\cdot \mathbf x}\) comes from \(\mathbf n=\mathbf r-\mathbf m\) and \(\mathbf n=\mathbf r+\mathbf m\),
		and therefore
		\[
		\widehat{(\boldsymbol{\sigma}_{\mathbf m,s}\cdot \nabla)\omega}(\mathbf r)
		=
		\sqrt{\frac{q(\mathbf m)}{2}}
		\Big(
		(\mathbf e_{\mathbf m}\cdot (\mathbf r-\mathbf m))\,\widehat\omega(\mathbf r-\mathbf m)
		-
		(\mathbf e_{\mathbf m}\cdot (\mathbf r+\mathbf m))\,\widehat\omega(\mathbf r+\mathbf m)
		\Big).
		\]
		Finally, we adopt the convention
		\begin{align}\label{today055}
			\widehat\omega(\mathbf a)=0
			\qquad\text{if }\mathbf a\notin\mathcal K_N,
		\end{align}
		so that the above formulas hold for every \(\mathbf r\in\mathbb Z^2\) without any further
		restriction.
		
		\medskip
		\noindent\textbf{Step 5: Parseval and cancellation of the cross terms.}
		Since \(\operatorname{div}\boldsymbol{\sigma}_{\mathbf m,c}=0\), the transport term
		\((\boldsymbol{\sigma}_{\mathbf m,c}\cdot \nabla)\omega\) has zero mean. Indeed,
		\[
		\int_{\T^2}(\boldsymbol{\sigma}_{\mathbf m,c}\cdot \nabla)\omega\,{\rm d}\mathbf x
		=
		-\int_{\T^2}\omega\,\operatorname{div}\boldsymbol{\sigma}_{\mathbf m,c}\,{\rm d}\mathbf x
		=0.
		\]
		Hence \((-\Delta)^{-1/2}(\boldsymbol{\sigma}_{\mathbf m,c}\cdot \nabla)\omega\) is well defined as the
		mean-zero Fourier multiplier
		\[
		(-\Delta)^{-1/2}(\boldsymbol{\sigma}_{\mathbf m,c}\cdot \nabla)\omega
		=
		\sum_{\mathbf r\in\mathbb Z^2\setminus\{{\bf 0}\}}
		|\mathbf r|^{-1}\,
		\widehat{(\boldsymbol{\sigma}_{\mathbf m,c}\cdot \nabla)\omega}(\mathbf r)\,
		e^{i\mathbf r\cdot \mathbf x}.
		\]
		For \(\mathbf r\in\mathbb Z^2\setminus\{{\bf 0}\}\), Parseval's identity gives
		\[
		\Big\|(-\Delta)^{-1/2}
		(\boldsymbol{\sigma}_{\mathbf m,c}\cdot \nabla)\omega
		\Big\|_{ \mathbb{L}^2(\T^2)}^2
		=
		\sum_{\mathbf r\neq{\bf 0}}
		\frac{
			\big|
			\widehat{(\boldsymbol{\sigma}_{\mathbf m,c}\cdot \nabla)\omega}(\mathbf r)
			\big|^2
		}{|\mathbf r|^2},
		\]
		and similarly
		\[
		\Big\|(-\Delta)^{-1/2}
		(\boldsymbol{\sigma}_{\mathbf m,s}\cdot \nabla)\omega
		\Big\|_{ \mathbb{L}^2(\T^2)}^2
		=
		\sum_{\mathbf r\neq{\bf 0}}
		\frac{
			\big|
			\widehat{(\boldsymbol{\sigma}_{\mathbf m,s}\cdot \nabla)\omega}(\mathbf r)
			\big|^2
		}{|\mathbf r|^2}.
		\]
		We set
		\[
		a_{\mathbf r}:=(\mathbf e_{\mathbf m}\cdot (\mathbf r-\mathbf m))\,\widehat\omega(\mathbf r-\mathbf m),
		\qquad
		b_{\mathbf r}:=(\mathbf e_{\mathbf m}\cdot (\mathbf r+\mathbf m))\,\widehat\omega(\mathbf r+\mathbf m).
		\]
		By Step~4, then
		\[
		\widehat{(\boldsymbol{\sigma}_{\mathbf m,c}\cdot \nabla)\omega}(\mathbf r)
		=
		i\sqrt{\frac{q(\mathbf m)}{2}}\,(a_{\mathbf r}+b_{\mathbf r}),
		\qquad
		\widehat{(\boldsymbol{\sigma}_{\mathbf m,s}\cdot \nabla)\omega}(\mathbf r)
		=
		\sqrt{\frac{q(\mathbf m)}{2}}\,(a_{\mathbf r}-b_{\mathbf r}).
		\]
		Hence, by binomial formula
		\[
		\begin{aligned}
			&
			\big|
			\widehat{(\boldsymbol{\sigma}_{\mathbf m,c}\cdot \nabla)\omega}(\mathbf r)
			\big|^2
			+
			\big|
			\widehat{(\boldsymbol{\sigma}_{\mathbf m,s}\cdot \nabla)\omega}(\mathbf r)
			\big|^2
			\\
			&\qquad=
			\frac{q(\mathbf m)}{2}
			\Big(
			|a_{\mathbf r}+b_{\mathbf r}|^2+|a_{\mathbf r}-b_{\mathbf r}|^2
			\Big)
			\\
			&\qquad=
			q(\mathbf m)\big(|a_{\mathbf r}|^2+|b_{\mathbf r}|^2\big).
		\end{aligned}
		\]
		This is the exact cancellation of the cross terms between the cosine and sine
		contributions.
		
		\noindent
		By substituting this identity into the Parseval formulas, we obtain
		\[
		\begin{aligned}
			&\Big\|(-\Delta)^{-1/2}
			(\boldsymbol{\sigma}_{\mathbf m,c}\cdot \nabla)\omega
			\Big\|_{ \mathbb{L}^2(\T^2)}^2
			+
			\Big\|(-\Delta)^{-1/2}
			(\boldsymbol{\sigma}_{\mathbf m,s}\cdot \nabla)\omega
			\Big\|_{ \mathbb{L}^2(\T^2)}^2
			\\
			&\qquad=
			q(\mathbf m)\sum_{\mathbf r\neq{\bf 0}}
			\frac{
				|\mathbf e_{\mathbf m}\cdot (\mathbf r-\mathbf m)|^2\,|\widehat\omega(\mathbf r-\mathbf m)|^2
				+
				|\mathbf e_{\mathbf m}\cdot (\mathbf r+\mathbf m)|^2\,|\widehat\omega(\mathbf r+\mathbf m)|^2
			}{|\mathbf r|^2}.
		\end{aligned}
		\]
		
		\medskip
		\noindent\textbf{Step 6: reindexing.}
		We now reindex the two sums: in the first term, set \(\mathbf n=\mathbf r-\mathbf m\) and in the second term, set \(\mathbf n=\mathbf r+\mathbf m\). Since the sums are over \(\mathbf r\neq{\bf 0}\), this produces the indicators
		\(\mathbf 1_{\{\mathbf n+\mathbf {\bf m}\neq{\bf 0}\}}\) and \(\mathbf 1_{\{\mathbf n-\mathbf {\bf m}\neq{\bf 0}\}}\). Thus on using~\eqref{today055},
		
		\begin{align}\label{today344}
			&\Big\|(-\Delta)^{-1/2}
			(\boldsymbol{\sigma}_{\mathbf m,c}\cdot \nabla)\omega
			\Big\|_{ \mathbb{L}^2(\T^2)}^2
			+
			\Big\|(-\Delta)^{-1/2}
			(\boldsymbol{\sigma}_{\mathbf m,s}\cdot \nabla)\omega
			\Big\|_{ \mathbb{L}^2(\T^2)}^2\notag
			\\
			&\qquad=
			q(\mathbf m)
			\sum_{{\bf 0}\neq \mathbf n\in\mathcal K_N}
			|\mathbf e_{\mathbf m}\cdot \mathbf n|^2
			\left(
			\frac{\mathbf 1_{\{\mathbf n+\mathbf {\bf m}\neq{\bf 0}\}}}{|\mathbf n+\mathbf m|^2}
			+
			\frac{\mathbf 1_{\{\mathbf n-\mathbf {\bf m}\neq{\bf 0}\}}}{|\mathbf n-\mathbf m|^2}
			\right)
			|\widehat\omega(\mathbf n)|^2 .
		\end{align}
		Here the zero mode is omitted because \((-\Delta)^{-1/2}\) is defined only on
		mean-zero functions.
		
		\medskip
		\noindent\textbf{Step 7: diagonal representation of the quadratic form.}
		By summing over \(\mathbf m\in\mathcal M_K\) in \eqref{today344} and using~\eqref{today345}, we obtain
		\[
		\widetilde{\mathfrak D}_{K}(\omega)
		=
		\sum_{{\bf 0}\neq \mathbf n\in\mathcal K_N}
		A_{K}(\mathbf n)\,|\widehat\omega(\mathbf n)|^2,
		\]
		where
		\[
		A_{K}(\mathbf n)
		:=
		\sum_{\mathbf m\in\mathcal M_K}
		q(\mathbf m)
		|\mathbf e_{\mathbf m}\cdot \mathbf n|^2
		\left(
		\frac{\mathbf 1_{\{\mathbf n+\mathbf {\bf m}\neq{\bf 0}\}}}{|\mathbf n+\mathbf m|^2}
		+
		\frac{\mathbf 1_{\{\mathbf n-\mathbf {\bf m}\neq{\bf 0}\}}}{|\mathbf n-\mathbf m|^2}
		\right).
		\]
		Equivalently,
		\[
		A_{K}(\mathbf n)
		=
		\sum_{\mathbf m\in\mathcal M_K}
		q(\mathbf m)\,T(\mathbf n,\mathbf m),
		\]
		where \(T\) is introduced in
		Lemma~\ref{lem:finite_discrete_symbol_gap}.
		
		\medskip
		\noindent\textbf{Step 8: use the Lemma~\ref{lem:finite_discrete_symbol_gap}.}
		By definition,
		\[
		c_K=\sum_{\mathbf m\in\mathcal M_K} q(\mathbf m).
		\]
		Hence
		\[
		c_K-A_{K}(\mathbf n)
		=
		\sum_{\mathbf m\in\mathcal M_K}
		q(\mathbf m)\bigl(1-T(\mathbf n,\mathbf m)\bigr).
		\]
		Since \(K\ge \gamma\sqrt2\,N\) and \({\bf 0}\neq \mathbf n\in\mathcal K_N\) with $\gamma>1+\sqrt{2}$,
		Lemma~\ref{lem:finite_discrete_symbol_gap} applies and yields
		\[
		c_K-A_{K}(\mathbf n)
		\ge
		c_0\langle \mathbf n\rangle^{-2\alpha}
		-
		C_0\langle \mathbf n\rangle^{-2}.
		\]
		Equivalently,
		\[
		A_{K}(\mathbf n)-c_K
		\le
		-c_0\langle \mathbf n\rangle^{-2\alpha}
		+
		C_0\langle \mathbf n\rangle^{-2}.
		\]
		
		\medskip
		\noindent\textbf{Step 9: conclude.}
		By using the estimate from Step 1 and the diagonal representation from Step 7, we get
		\[
		\begin{aligned}
			\mathfrak D_{N,K}(\omega)-c_K\|\omega\|_{ \mathbb{L}^2(\T^2)}^2
			&\le
			\widetilde{\mathfrak D}_{K}(\omega)-c_K\|\omega\|_{ \mathbb{L}^2(\T^2)}^2
			\\
			&=
			\sum_{{\bf 0}\neq \mathbf n\in\mathcal K_N}
			\bigl(A_{K}(\mathbf n)-c_K\bigr)|\widehat\omega(\mathbf n)|^2
			\\
			&\le
			-c_0
			\sum_{{\bf 0}\neq \mathbf n\in\mathcal K_N}
			\langle \mathbf n\rangle^{-2\alpha}|\widehat\omega(\mathbf n)|^2
			+
			C_0
			\sum_{{\bf 0}\neq \mathbf n\in\mathcal K_N}
			\langle \mathbf n\rangle^{-2}|\widehat\omega(\mathbf n)|^2.
		\end{aligned}
		\]
		By the definition of the Sobolev norms,
		\[
		\sum_{{\bf 0}\neq \mathbf n\in\mathcal K_N}
		\langle \mathbf n\rangle^{-2\alpha}|\widehat\omega(\mathbf n)|^2
		=
		\|\omega\|_{\mathbb{H}^{-\alpha}(\T^2)}^2,
		\]
		and
		\[
		\sum_{{\bf 0}\neq \mathbf n\in\mathcal K_N}
		\langle \mathbf n\rangle^{-2}|\widehat\omega(\mathbf n)|^2
		=
		\|\omega\|_{\mathbb{H}^{-1}(\T^2)}^2.
		\]
		Therefore
		\[
		\mathfrak D_{N,K}(\omega)-c_K\|\omega\|_{\mathbb{L}^2(\T^2)}^2
		\le
		-c_0\|\omega\|_{\mathbb{H}^{-\alpha}(\T^2)}^2
		+
		C_0\|\omega\|_{\mathbb{H}^{-1}(\T^2)}^2,
		\]
		which is exactly \eqref{eq:finite_dimensional_coercivity}.
	\end{proof}
	
	\begin{remark}[Regularization from high Kraichnan modes]
		\label{rem:regularization_high_modes}
		The admissibility condition relating the spatial Fourier cutoff \(N\) to the
		noise cutoff \(K\) is $ K\ge \gamma\sqrt{2}\,N,\,\gamma>1+\sqrt{2}.$ This condition should be viewed as a scale-separation requirement between the
		resolved vorticity modes and the Kraichnan modes retained in the truncated noise.
		It plays a structural role in the coercivity mechanism.
		
		The first threshold, \(\gamma>1\), is intrinsic. The continuum defect
		\(j_{\alpha,\gamma}\) is strictly positive only when the noise cutoff extends
		beyond the range of resolved spatial frequencies. For \(0<\gamma\le1\), one has
		\(j_{\alpha,\gamma}=0\), and the corresponding noise modes do not yield a
		positive dissipative contribution; see Remark~\ref{rem:importance_gamma}. Thus the coercive effect is generated by
		noise modes which are genuinely higher than the spatial resolution.
		
		The stronger condition \(\gamma>1+\sqrt{2}\) enters the discrete lattice analysis in Lemma~\ref{lem:finite_discrete_symbol_gap}.
		More precisely, it is used in the proof of Lemma~\ref{lem:finite_discrete_symbol_gap}
		to control the remainder term in \eqref{new600}. This additional separation
		ensures that the relevant interactions between resolved modes and retained noise
		modes remain in the coercive regime, so that the non-coercive contributions can
		be controlled.
		
		Consequently, the condition \(K\ge\gamma\sqrt{2}\,N\) guarantees that the
		truncated Kraichnan noise contains sufficiently many high-frequency modes
		relative to the spatial discretization. Under this condition, the discrete noise
		quadratic form~\eqref{quadratic form} produces the negative-Sobolev dissipation quantified by the
		coercivity estimate~\eqref{eq:finite_dimensional_coercivity}. In this sense, the regularizing effect captured by the fully
		discrete scheme~\eqref{eq:scheme_keep} comes from the action of high Kraichnan modes on the resolved
		vorticity field.
	\end{remark}

	\section{Uniform discrete energy estimate}
	To prove the convergence of the fully discrete scheme \eqref{eq:scheme_keep}, we first need uniform a priori bounds that are compatible with the compactness argument and the passage to the limit. In this section, we therefore establish uniform energy estimates for the discrete approximations in the norms relevant for the later analysis, whose derivation rests on the coercivity estimate \eqref{eq:finite_dimensional_coercivity}, in particular.
	
	A basic ingredient in these estimates is the mixed-time structure of the convection term in \eqref{eq:scheme_keep}; see also Remark~\ref{rem:mixed_time_conv}. When the scheme is tested in the \(\mathbb{H}^{-1}(\T^2)\)-sense against \((-\Delta)^{-1}\omega_N^{n+1}\), this structure yields an exact cancellation of the transport contribution. We record this fact in the following lemma, which will be used repeatedly in the derivation of the uniform energy bounds.

	\begin{lemma}[Exact cancellation of the mixed-time convection in the $\mathbb{H}^{-1}(\T^2)$ test]
		\label{lem:conv_cancel_keep}
		Let $\omega^n$ be any scalar field on $\T^2$, and let $\omega^{n+1}$ be a
		mean-zero scalar field on $\T^2$. Define
		\[
		\psi^{n+1}:=(-\Delta)^{-1}\omega^{n+1},
		\qquad
		{\bf u}^{n+1}:=\nabla^\perp \psi^{n+1}.
		\]
		Then
		\[
		\big\langle ({\bf u}^{n+1}\cdot \nabla)\omega^n,\psi^{n+1}\big\rangle_{\mathbb{L}^2(\T^2)}=0.
		\]
		If, in addition, $\psi^{n+1}\in\mathbb V_N$, then
		\[
		\big\langle \mathcal P_N[({\bf u}^{n+1}\cdot \nabla)\omega^n],\psi^{n+1}\big\rangle_{\mathbb{L}^2(\T^2)}=0.
		\]
	\end{lemma}
	
	\begin{proof}
		Since \({\bf u}^{n+1}=\nabla^\perp\psi^{n+1}\), we have pointwise
		\[
		\nabla\psi^{n+1}\cdot {\bf u}^{n+1}
		=
		\nabla\psi^{n+1}\cdot \nabla^\perp\psi^{n+1}
		=0.
		\]
		Hence, for any scalar test function \(\phi\),
		\[
		0=\int_{\T^2} (\nabla\psi^{n+1}\cdot {\bf u}^{n+1})\,\phi\,\dd {\bf x}.
		\]
		Using periodic integration by parts, we obtain
		\[
		0
		=
		-\int_{\T^2}\psi^{n+1}\,\nabla\cdot({\bf u}^{n+1}\phi)\,\dd {\bf x}.
		\]
		Because \({\bf u}^{n+1}\) is divergence-free, that is,
		\[
		\nabla\cdot {\bf u}^{n+1}=\nabla\cdot\nabla^\perp\psi^{n+1}=0,
		\]
		the divergence expands as
		\[
		\nabla\cdot({\bf u}^{n+1}\phi)={\bf u}^{n+1}\cdot\nabla\phi.
		\]
		Therefore
		\[
		0
		=
		-\int_{\T^2}\psi^{n+1}({\bf u}^{n+1}\cdot\nabla\phi)\,\dd {\bf x}.
		\]
		By choosing \(\phi=\omega^n\) gives
		\[
		\big\langle ({\bf u}^{n+1}\cdot\nabla)\omega^n,\psi^{n+1}\big\rangle_{\mathbb{L}^2(\T^2)}=0.
		\]
		If moreover \(\psi^{n+1}\in\mathbb V_N\), then \(\mathcal P_N\) is self-adjoint on
		\(\mathbb{L}^2(\T^2)\) and \(\mathcal P_N\psi^{n+1}=\psi^{n+1}\). Hence
		\begin{align*}
			\big\langle \mathcal P_N[({\bf u}^{n+1}\cdot\nabla)\omega^n],\psi^{n+1}\big\rangle_{\mathbb{L}^2(\T^2)}&
			=
			\big\langle ({\bf u}^{n+1}\cdot\nabla)\omega^n,\mathcal P_N\psi^{n+1}\big\rangle_{\mathbb{L}^2(\T^2)}\\
			&=
			\big\langle ({\bf u}^{n+1}\cdot\nabla)\omega^n,\psi^{n+1}\big\rangle_{\mathbb{L}^2(\T^2)}
			=0.
		\end{align*}
		This proves the result.
	\end{proof}
	The proof of the following result rests on the coercivity estimate~\eqref{eq:finite_dimensional_coercivity} in Theorem~\ref{prop:finite_dimensional_coercivity}.
	\begin{theorem}[Uniform \(\mathbb{H}^{-1}(\T^2)\) estimate for \eqref{eq:scheme_keep}]
		\label{thm:energy_keep_full}
		Fix \(T>0\), and \(0<\alpha<1\). Let \(N_T=\lfloor T/\Dt\rfloor\). Let \(N,K\in\mathbb N\) satisfy $K\ge \gamma\sqrt2\,N$, with $\gamma>1+\sqrt{2}$.
		Then there exists a constant \(C\ge1\), depending only on
		\((\alpha, T)\), such that 
		\begin{equation}\label{eq:uniform_energy_keep}
			\begin{aligned}
				&\max_{0\le n\le N_T}
				\bigg(
				\mathbb E\Big[
				\|\omega_N^n\|_{\mathbb{H}^{-1}(\T^2)}^2
				\Big]
				+
				\Dt\,c_K\,
				\mathbb E\Big[
				\|\omega_N^n\|_{\mathbb{L}^2(\T^2)}^2
				\Big]
				\bigg)+
				c_0
				\sum_{n=0}^{N_T-1}
				\Dt\,
				\mathbb E\Big[
				\|\omega_N^n\|_{\mathbb{H}^{-\alpha}(\T^2)}^2
				\Big]
				\\
				&\le
				C
				\bigg(
				\|\omega_N^0\|_{\mathbb{H}^{-1}(\T^2)}^2
				+
				\Dt\,c_K\,
				\|\omega_N^0\|_{\mathbb{L}^2(\T^2)}^2
				\bigg).
			\end{aligned}
		\end{equation}
	\end{theorem}
	
	\begin{proof}
		We set
		\[
		\psi_N^n:=(-\Delta)^{-1}\omega_N^n,
		\qquad n=0,\dots,N_T.
		\]
		Since each \(\omega_N^n\) has zero spatial mean, this is well-defined.
		
		\medskip
		\noindent\textbf{Step 1.}
		We test \eqref{eq:scheme_keep} with
		\(\psi_N^{n+1}\). The polarization identity in \(\mathbb{H}^{-1}(\T^2)\) gives
		\begin{align}\label{1111}
			\big\langle \omega_N^{n+1}-\omega_N^n,\psi_N^{n+1}\big\rangle
			=
			\frac12
			\Big(
			\|\omega_N^{n+1}\|_{\mathbb{H}^{-1}(\T^2)}^2
			-
			\|\omega_N^n\|_{\mathbb{H}^{-1}(\T^2)}^2
			+
			\|\omega_N^{n+1}-\omega_N^n\|_{\mathbb{H}^{-1}(\T^2)}^2
			\Big).
		\end{align}
		The convection term vanishes by Lemma~\ref{lem:conv_cancel_keep}, and
		\begin{align}\label{3333}
			\big\langle \Delta\omega_N^{n+1},\psi_N^{n+1}\big\rangle
			=
			-\|\omega_N^{n+1}\|_{\mathbb{L}^2(\T^2)}^2.
		\end{align}
		Therefore
		\begin{equation}\label{eq:energy_id_keep_corrected}
			\begin{aligned}
				&\frac12
				\Big(
				\|\omega_N^{n+1}\|_{\mathbb{H}^{-1}(\T^2)}^2
				-
				\|\omega_N^n\|_{\mathbb{H}^{-1}(\T^2)}^2
				\Big)
				+
				\frac12
				\|\omega_N^{n+1}-\omega_N^n\|_{\mathbb{H}^{-1}(\T^2)}^2
				+
				\Dt\,\frac{c_K}{2}
				\|\omega_N^{n+1}\|_{ \mathbb{L}^2(\T^2)}^2
				\\
				&\qquad
				=
				-
				\sum_{k\in\mathcal I^{K}}
				\Big\langle
				\mathcal{P}_N[(\boldsymbol{\sigma}_k\cdot \nabla)\omega_N^n],
				\psi_N^{n+1}
				\Big\rangle
				\Delta_{n+1} {\bf W}_k .
			\end{aligned}
		\end{equation}
		We define
		\begin{align}\label{today32}
			\Xi_n
			:=
			\sum_{k\in\mathcal I^{K}}
			\mathcal{P}_N[(\boldsymbol{\sigma}_k\cdot \nabla)\omega_N^n]\Delta_{n+1} {\bf W}_k .
		\end{align}
		Then the right-hand side of \eqref{eq:energy_id_keep_corrected} is
		\[
		-\langle \Xi_n,\psi_N^{n+1}\rangle .
		\]
		We split
		\begin{align}\label{2222}
			-\langle \Xi_n,\psi_N^{n+1}\rangle
			=
			-\langle \Xi_n,\psi_N^n\rangle
			-
			\langle \Xi_n,\psi_N^{n+1}-\psi_N^n\rangle .
		\end{align}
		Since \(\psi_N^n\) and the coefficients in \(\Xi_n\) are
		\(\mathcal F_n\)-measurable, while
		\(\mathbb E[\Delta_{n+1} {\bf W}_k\mid\mathcal F_n]={\bf 0}\), we have
		\[
		\mathbb E\Big[
		\langle \Xi_n,\psi_N^n\rangle
		\Big]=0.
		\]
		For the second term, using
		\[
		\psi_N^{n+1}-\psi_N^n
		=
		(-\Delta)^{-1}(\omega_N^{n+1}-\omega_N^n),
		\]
		we obtain
		\[
		\big|
		\langle \Xi_n,\psi_N^{n+1}-\psi_N^n\rangle
		\big|
		\le
		\|\Xi_n\|_{\mathbb{H}^{-1}(\T^2)}
		\|\omega_N^{n+1}-\omega_N^n\|_{\mathbb{H}^{-1}(\T^2)}.
		\]
		By Young's inequality, we get
		\begin{equation}\label{eq:young_half_corrected}
			\big|
			\langle \Xi_n,\psi_N^{n+1}-\psi_N^n\rangle
			\big|
			\le
			\frac12
			\|\omega_N^{n+1}-\omega_N^n\|_{\mathbb{H}^{-1}(\T^2)}^2
			+
			\frac12
			\|\Xi_n\|_{\mathbb{H}^{-1}(\T^2)}^2 .
		\end{equation}
		By taking expectations in \eqref{eq:energy_id_keep_corrected}, using the adaptedness
		of the first stochastic term, and applying \eqref{eq:young_half_corrected}, the
		increment term cancels. Hence
		\begin{equation}\label{eq:energy_before_iso_corrected}
			\begin{aligned}
				&\frac12
				\Big(
				\mathbb E\Big[
				\|\omega_N^{n+1}\|_{\mathbb{H}^{-1}(\T^2)}^2
				\Big]
				-
				\mathbb E\Big[
				\|\omega_N^n\|_{\mathbb{H}^{-1}(\T^2)}^2
				\Big]
				\Big)
				+
				\Dt\,\frac{c_K}{2}
				\mathbb E\Big[
				\|\omega_N^{n+1}\|_{ \mathbb{L}^2(\T^2)}^2
				\Big]
				\\
				&\qquad
				\le
				\frac12
				\mathbb E\Big[
				\|\Xi_n\|_{\mathbb{H}^{-1}(\T^2)}^2
				\Big].
			\end{aligned}
		\end{equation}
		
		\medskip
		\noindent\textbf{Step 2.}
		By It\^o isometry and Definition~\ref{def:DLK},
		\begin{align}\label{today33}
			\mathbb E\Big[
			\|\Xi_n\|_{\mathbb{H}^{-1}(\T^2)}^2
			\Big]
			=
			\Dt
			\sum_{k\in\mathcal I^{K}}
			\mathbb{E}\bigg[\big\|
			\mathcal{P}_N[(\boldsymbol{\sigma}_k^{L}\cdot \nabla)\omega_N^n]
			\big\|_{\mathbb{H}^{-1}(\T^2)}^2\bigg]
			\le
			\Dt\,\mathbb{E}\big[\mathfrak D_{N, K}(\omega_N^n)\big].
		\end{align}
		By inserting this into \eqref{eq:energy_before_iso_corrected} gives
		\begin{equation}\label{eq:energy_before_coercivity_corrected}
			\begin{aligned}
				&\frac12
				\Big(
				\mathbb E\Big[
				\|\omega_N^{n+1}\|_{\mathbb{H}^{-1}(\T^2)}^2
				\Big]
				-
				\mathbb E\Big[
				\|\omega_N^n\|_{\mathbb{H}^{-1}(\T^2)}^2
				\Big]
				\Big)
				+
				\Dt\,\frac{c_K}{2}
				\mathbb E\Big[
				\|\omega_N^{n+1}\|_{ \mathbb{L}^2(\T^2)}^2
				\Big]
				\\
				&\qquad
				\le
				\frac{\Dt}{2}
				\mathbb E\Big[
				\mathfrak D_{N, K}(\omega_N^n)
				\Big].
			\end{aligned}
		\end{equation}
		We subtract
		$
		\Dt\,\frac{c_K}{2}
		\mathbb E\Big[
		\|\omega_N^n\|_{ \mathbb{L}^2(\T^2)}^2
		\Big]
		$
		from both sides. Then
		\begin{equation}\label{eq:energy_before_maurelli_corrected}
			\begin{aligned}
				&\frac12
				\Big(
				\mathbb E\Big[
				\|\omega_N^{n+1}\|_{\mathbb{H}^{-1}(\T^2)}^2
				\Big]
				-
				\mathbb E\Big[
				\|\omega_N^n\|_{\mathbb{H}^{-1}(\T^2)}^2
				\Big]
				\Big)
				+
				\Dt\,\frac{c_K}{2}
				\Big(
				\mathbb E\Big[
				\|\omega_N^{n+1}\|_{ \mathbb{L}^2(\T^2)}^2
				\Big]
				-
				\mathbb E\Big[
				\|\omega_N^n\|_{ \mathbb{L}^2(\T^2)}^2
				\Big]
				\Big)
				\\
				&\qquad
				\le
				\Dt\,
				\mathbb E\Big[
				\frac12\mathfrak D_{N, K}(\omega_N^n)
				-
				\frac{c_K}{2}
				\|\omega_N^n\|_{ \mathbb{L}^2(\T^2)}^2
				\Big].
			\end{aligned}
		\end{equation}
		
		\medskip
		\noindent\textbf{Step 3.}
		By coercivity estimate~\eqref{eq:finite_dimensional_coercivity}, we obtain
		\[
		\frac12\mathfrak D_{N, K}(\omega_N^n)
		-
		\frac{c_K}{2}
		\|\omega_N^n\|_{ \mathbb{L}^2(\T^2)}^2
		\le
		-c_0
		\|\omega_N^n\|_{\mathbb{H}^{-\alpha}(\T^2)}^2
		+
		C_0
		\|\omega_N^n\|_{\mathbb{H}^{-1}(\T^2)}^2 .
		\]
		Hence \eqref{eq:energy_before_maurelli_corrected} yields
		\begin{equation}\label{eq:recursion_keep_corrected}
			\begin{aligned}
				&\mathbb E\Big[
				\|\omega_N^{n+1}\|_{\mathbb{H}^{-1}(\T^2)}^2
				\Big]
				+
				\Dt\,c_K
				\mathbb E\Big[
				\|\omega_N^{n+1}\|_{ \mathbb{L}^2(\T^2)}^2
				\Big]
				+
				2c_0\Dt\,
				\mathbb E\Big[
				\|\omega_N^n\|_{\mathbb{H}^{-\alpha}(\T^2)}^2
				\Big]
				\\
				&\qquad
				\le
				(1+2C_0\Dt)
				\mathbb E\Big[
				\|\omega_N^n\|_{\mathbb{H}^{-1}(\T^2)}^2
				\Big]
				+
				\Dt\,c_K
				\mathbb E\Big[
				\|\omega_N^n\|_{ \mathbb{L}^2(\T^2)}^2
				\Big].
			\end{aligned}
		\end{equation}
		
		\medskip
		\noindent\textbf{Step 4: discrete Gronwall.}
		From \eqref{eq:recursion_keep_corrected} we obtain
		\begin{align*}
			&\mathbb E\Big[
			\|\omega_N^{n+1}\|_{\mathbb{H}^{-1}(\T^2)}^2
			\Big]
			+
			\Dt\,c_K
			\mathbb E\Big[
			\|\omega_N^{n+1}\|_{ \mathbb{L}^2(\T^2)}^2
			\Big]\\
			&\le
			(1+C_0\Dt)
			\Bigg(
			\mathbb E\Big[
			\|\omega_N^{n}\|_{\mathbb{H}^{-1}(\T^2)}^2
			\Big]
			+
			\Dt\,c_K
			\mathbb E\Big[
			\|\omega_N^{n}\|_{ \mathbb{L}^2(\T^2)}^2
			\Big]
			\Bigg).
		\end{align*}
		Therefore, by iteration,
		\[
		\mathbb E\Big[
		\|\omega_N^{m}\|_{\mathbb{H}^{-1}(\T^2)}^2
		\Big]
		+
		\Dt\,c_K
		\mathbb E\Big[
		\|\omega_N^{m}\|_{ \mathbb{L}^2(\T^2)}^2
		\Big]
		\le
		(1+C_0\Dt)^m
		\Bigg(
		\|\omega_N^0\|_{\mathbb{H}^{-1}(\T^2)}^2
		+
		\Dt\,c_K\,\|\omega_N^0\|_{ \mathbb{L}^2(\T^2)}^2
		\Bigg)
		\]
		for every \(0\le m\le N_T\). Since \((1+C_0\Dt)^m\le e^{C_0T}\), it follows that
		\[
		\max_{0\le m\le N_T}
		\Bigg(
		\mathbb E\Big[
		\|\omega_N^{m}\|_{\mathbb{H}^{-1}(\T^2)}^2
		\Big]
		+
		\Dt\,c_K
		\mathbb E\Big[
		\|\omega_N^{m}\|_{ \mathbb{L}^2(\T^2)}^2
		\Big]
		\Bigg)
		\le
		C
		\Big(
		\|\omega_N^0\|_{\mathbb{H}^{-1}(\T^2)}^2
		+
		\Dt\,c_K\,\|\omega_N^0\|_{ \mathbb{L}^2(\T^2)}^2
		\Big).
		\]
		Next, summing \eqref{eq:recursion_keep_corrected} from \(n=0\) to \(N_T-1\), we obtain
		\[
		c_0
		\sum_{n=0}^{N_T-1}
		\Dt\,
		\mathbb E\Big[
		\|\omega_N^n\|_{\mathbb{H}^{-\alpha}(\T^2)}^2
		\Big]
		\le
		\|\omega_N^0\|_{\mathbb{H}^{-1}(\T^2)}^2
		+
		\Dt\,c_K\,\|\omega_N^0\|_{ \mathbb{L}^2(\T^2)}^2
		+
		C_0
		\sum_{n=0}^{N_T-1}
		\Dt\,
		\mathbb E\Big[
		\|\omega_N^{n}\|_{\mathbb{H}^{-1}(\T^2)}^2
		\Big].
		\]
		By using the bound already proved for the \(\mathbb{H}^{-1}(\T^2)\)-term, we conclude that
		\[
		c_0
		\sum_{n=0}^{N_T-1}
		\Dt\,
		\mathbb E\Big[
		\|\omega_N^n\|_{\mathbb{H}^{-\alpha}(\T^2)}^2
		\Big]
		\le
		C
		\Big(
		\|\omega_N^0\|_{\mathbb{H}^{-1}(\T^2)}^2
		+
		\Dt\,c_K\,\|\omega_N^0\|_{ \mathbb{L}^2(\T^2)}^2
		\Big).
		\]
		This is exactly \eqref{eq:uniform_energy_keep}.
	\end{proof}

	The following theorem establishes pathwise uniform bounds for the discrete approximations and its proof uses Theorem~\ref{thm:energy_keep_full}. These bounds are a key ingredient in the later stochastic compactness argument, since they provide the control needed to prove tightness of the laws of the approximation components. In particular, they will be used in the proof of Lemma~\ref{lem:stop_high_prob_conv}.

	\begin{theorem}[Uniform strong \(\mathbb{H}^{-1}(\T^2)\) estimate for \eqref{eq:scheme_keep}]
		\label{thm:strong_energy_keep} Let \(N,K\in\mathbb N\) satisfy $K\ge \gamma\sqrt2\,N$, with $\gamma>1+\sqrt{2}$.
		There exists a
		constant \(C>0\), depending only on \((\alpha,T)\), such that
		\begin{equation}\label{eq:strong_energy_keep}
			\begin{aligned}
				&\E\Big[\sup_{0\le n\le N_T}\|\omega_N^n\|_{\mathbb{H}^{-1}(\T^2)}^2\Big]
				+
				\sum_{n=0}^{N_T-1}\Dt\,\E\Big[\|\omega_N^n\|_{\mathbb{H}^{-\alpha}(\T^2)}^2\Big]\\&
				\le
				C\bigg(
				\|\omega_N^0\|_{\mathbb{H}^{-1}(\T^2)}^2
				+
				\Dt\,c_K\,\E\Big[\|\omega_N^0\|_{\mathbb{L}^2(\T^2)}^2\Big]
				\bigg).
			\end{aligned}
		\end{equation}
	\end{theorem}
	
	\begin{proof}
		For each \(n=0,\dots,N_T\), let
		\[
		\psi_N^n:=(-\Delta)^{-1}\omega_N^n,
		\]
		so that
		\[
		\|\omega_N^n\|_{\mathbb{H}^{-1}(\T^2)}^2
		=
		\langle \omega_N^n,\psi_N^n\rangle .
		\]
		We also  recall the stochastic increment from \eqref{today32}.

		\medskip
		\noindent\textbf{Step 1: the basic pathwise energy inequality.}
		We follow Step~1 of the proof for Theorem~\ref{thm:energy_keep_full} to obtain
		\begin{equation}\label{eq:strong_energy_clean_1}
			\begin{aligned}
				&\frac12\Big(
				\|\omega_N^{n+1}\|_{\mathbb{H}^{-1}(\T^2)}^2-\|\omega_N^n\|_{\mathbb{H}^{-1}(\T^2)}^2
				\Big)
				+\frac12\|\omega_N^{n+1}-\omega_N^n\|_{\mathbb{H}^{-1}(\T^2)}^2
				+\Dt\,\frac{c_K}{2}\|\omega_N^{n+1}\|_{\mathbb{L}^2(\T^2)}^2
				\\
				&\qquad
				=
				-\langle \Xi_n,\psi_N^{n}\rangle
				-\langle \Xi_n,\psi_N^{n+1}-\psi_N^{n}\rangle.
			\end{aligned}
		\end{equation}
		We define
		\begin{align}\label{today34}
			\delta\mathcal M_{n+1}:=-\langle \Xi_n,\psi_N^n\rangle .
		\end{align}
		Since \(\omega_N^n\), hence \(\psi_N^n\) is 
		\(\mathcal F_n\)-measurable; while \(\Delta_{n+1} {\bf W}_k\) is independent from $\mathcal{F}_n$, we
		have
		\[
		\E\big[\delta\mathcal M_{n+1}\mid \mathcal F_n\big]=0.
		\]
		Thus the partial sums
		\[
		\mathcal M_m:=\sum_{r=0}^{m-1}\delta\mathcal M_{r+1},
		\qquad
		m=0,\dots,N_T,
		\]
		form a discrete $\{\mathcal{F}_{m}\}_{m=0}^{N_T}$-martingale.
		
		\noindent
		For the remainder term, we proceed as in \eqref{eq:young_half_corrected} to get the pathwise
		inequality
		\begin{equation}\label{eq:strong_energy_clean_2}
			\begin{aligned}
				&\frac12\Big(
				\|\omega_N^{n+1}\|_{\mathbb{H}^{-1}(\T^2)}^2-\|\omega_N^n\|_{\mathbb{H}^{-1}(\T^2)}^2
				\Big)
				+\Dt\,\frac{c_K}{2}\|\omega_N^{n+1}\|_{\mathbb{L}^2(\T^2)}^2
				\le
				\delta\mathcal M_{n+1}
				+
				\frac12\|\Xi_n\|_{\mathbb{H}^{-1}(\T^2)}^2 .
			\end{aligned}
		\end{equation}
		
		\noindent
		By summing \eqref{eq:strong_energy_clean_2} from \(n=0\) to \(m-1\), we find
		\[
		\begin{aligned}
			&\frac12\|\omega_N^{m}\|_{\mathbb{H}^{-1}(\T^2)}^2
			+\Dt\,\frac{c_K}{2}\sum_{n=0}^{m-1}\|\omega_N^{n+1}\|_{\mathbb{L}^2(\T^2)}^2
			\le
			\frac12\|\omega_N^{0}\|_{\mathbb{H}^{-1}(\T^2)}^2
			+\mathcal M_m
			+\frac12\sum_{n=0}^{m-1}\|\Xi_n\|_{\mathbb{H}^{-1}(\T^2)}^2 .
		\end{aligned}
		\]
		By taking the supremum over \(m=0,\dots,N_T\), we obtain
		\begin{equation}\label{eq:strong_energy_clean_3}
			\begin{aligned}
				&\frac12\sup_{0\le m\le N_T}\|\omega_N^{m}\|_{\mathbb{H}^{-1}(\T^2)}^2
				+\Dt\,\frac{c_K}{2}\sum_{n=0}^{N_T-1}\|\omega_N^{n+1}\|_{\mathbb{L}^2(\T^2)}^2
				\\
				&\qquad
				\le
				\frac12\|\omega_N^{0}\|_{\mathbb{H}^{-1}(\T^2)}^2
				+\sup_{0\le m\le N_T}|\mathcal M_m|
				+\sum_{n=0}^{N_T-1}\|\Xi_n\|_{\mathbb{H}^{-1}(\T^2)}^2 .
			\end{aligned}
		\end{equation}
		
		\medskip
		\noindent\textbf{Step 2: estimate the quadratic remainder.}
		By using~\eqref{today33} and \eqref{eq:finite_dimensional_coercivity}, we obtain
		\[
		\E\big[\|\Xi_n\|_{\mathbb{H}^{-1}(\T^2)}^2\big]
		\le
		\Dt\Big(
		{c_K}\E\big[\|\omega_N^n\|_{\mathbb{L}^2(\T^2)}^2\big]
		-c_0\E\big[\|\omega_N^n\|_{\mathbb{H}^{-\alpha}(\T^2)}^2\big]
		+C_0\E\big[\|\omega_N^n\|_{\mathbb{H}^{-1}(\T^2)}^2\big]
		\Big).
		\]
		By summing in \(n\), and using
		\[
		\Dt c_K\sum_{n=0}^{N_T-1}\E\big[\|\omega_N^n\|_{\mathbb{L}^2(\T^2)}^2\big]
		\le
		\Dt c_K\,\E\big[\|\omega_N^0\|_{\mathbb{L}^2(\T^2)}^2\big]
		+
		\Dt c_K\sum_{n=0}^{N_T-1}\E\big[\|\omega_N^{n+1}\|_{\mathbb{L}^2(\T^2)}^2\big],
		\]
		we get
		\begin{equation}\label{eq:strong_energy_clean_4}
			\begin{aligned}
				\frac12\sum_{n=0}^{N_T-1}\E\big[\|\Xi_n\|_{\mathbb{H}^{-1}(\T^2)}^2\big]
				&\le
				\frac{\Dt c_K}{2}\E\big[\|\omega_N^0\|_{\mathbb{L}^2(\T^2)}^2\big]
				+\frac{\Dt c_K}{2}\sum_{n=0}^{N_T-1}\E\big[\|\omega_N^{n+1}\|_{\mathbb{L}^2(\T^2)}^2\big]
				\\
				&\quad
				-c_0\sum_{n=0}^{N_T-1}\Dt\,\E\big[\|\omega_N^n\|_{\mathbb{H}^{-\alpha}(\T^2)}^2\big]
				+C_0\sum_{n=0}^{N_T-1}\Dt\,\E\big[\|\omega_N^n\|_{\mathbb{H}^{-1}(\T^2)}^2\big] .
			\end{aligned}
		\end{equation}
		
		\noindent
		By taking expectations in \eqref{eq:strong_energy_clean_3} and inserting
		\eqref{eq:strong_energy_clean_4}, the \(\mathbb{L}^2(\T^2)\)-term on the right is absorbed by the
		same term on the left. This gives
		\begin{equation}\label{eq:strong_energy_clean_5}
			\begin{aligned}
				&\frac12\,\E\Big[\sup_{0\le m\le N_T}\|\omega_N^{m}\|_{\mathbb{H}^{-1}(\T^2)}^2\Big]
				+c_0\sum_{n=0}^{N_T-1}\Dt\,\E\big[\|\omega_N^n\|_{\mathbb{H}^{-\alpha}(\T^2)}^2\big]
				\\
				&\qquad
				\le
				\frac12\|\omega_N^0\|_{\mathbb{H}^{-1}(\T^2)}^2
				+\frac{\Dt c_K}{2}\E\big[\|\omega_N^0\|_{\mathbb{L}^2(\T^2)}^2\big]
				+\E\Big[\sup_{0\le m\le N_T}|\mathcal M_m|\Big]
				\\
				&\qquad\quad
				+
				C_0\sum_{n=0}^{N_T-1}\Dt\,\E\big[\|\omega_N^n\|_{\mathbb{H}^{-1}(\T^2)}^2\big] .
			\end{aligned}
		\end{equation}
		
		\medskip
		\noindent\textbf{Step 3: estimate the martingale term.}
		Since \((\mathcal M_m)_{m=0}^{N_T}\) is a discrete $\{\mathcal{F}_{m}\}_{m=0}^{N_T}$-martingale, Doob's \(\mathbb{L}^2(\T^2)\)-inequality gives
		\[
		\E\Big[\sup_{0\le m\le N_T}|\mathcal M_m|\Big]
		\le
		\Big(\E\Big[\sup_{0\le m\le N_T}|\mathcal M_m|^2\Big]\Big)^{1/2}
		\le
		2\Big(\E\big[|\mathcal M_{N_T}|^2\big]\Big)^{1/2}.
		\]
		Because martingale increments are orthogonal in $\mathbb{L}^2(\Omega)$,
		\[
		\E\big[|\mathcal M_{N_T}|^2\big]
		=
		\sum_{n=0}^{N_T-1}\E\big[|\delta\mathcal M_{n+1}|^2\big]
		=
		\sum_{n=0}^{N_T-1}\E\Big[\E\big[|\delta\mathcal M_{n+1}|^2\mid\mathcal F_n\big]\Big].
		\]
		Now because of \eqref{today34}, we obtain
		\[
		\E\big[|\delta\mathcal M_{n+1}|^2\big]
		=
		\Dt\sum_{k\in\mathcal I^K}
		\mathbb{E}\bigg[\Big|
		\big\langle \PN[(\boldsymbol{\sigma}_k\cdot \nabla)\omega_N^n],\psi_N^n\big\rangle
		\Big|^2\bigg].
		\]
		By Lemma~\ref{lem:discrete_martingale_coeff_keep}, we have
		\[
		\E\big[|\delta\mathcal M_{n+1}|^2\big]
		\le
		C\Dt\,\mathbb{E}\bigg[\|\omega_N^n\|_{\mathbb{H}^{-1}(\T^2)}^2\,
		\|\omega_N^n\|_{\mathbb{H}^{-\alpha}(\T^2)}^2\bigg].
		\]
		Therefore
		\[
		\begin{aligned}
			\E\Big[\sup_{0\le m\le N_T}|\mathcal M_m|\Big]
			&\le
			C\Bigg(
			\E\Big[
			\sum_{n=0}^{N_T-1}
			\Dt\,\|\omega_N^n\|_{\mathbb{H}^{-1}(\T^2)}^2\,
			\|\omega_N^n\|_{\mathbb{H}^{-\alpha}(\T^2)}^2
			\Big]
			\Bigg)^{1/2}
			\\
			&\le
			C\Bigg(
			\E\Big[
			\sup_{0\le n\le N_T}\|\omega_N^n\|_{\mathbb{H}^{-1}(\T^2)}^2
			\sum_{n=0}^{N_T-1}\Dt\,\|\omega_N^n\|_{\mathbb{H}^{-\alpha}(\T^2)}^2
			\Big]
			\Bigg)^{1/2}.
		\end{aligned}
		\]
		By applying Cauchy--Schwarz and then Young's inequality, we obtain, for every
		\(\varepsilon>0\),
		\[
		\E\Big[\sup_{0\le m\le N_T}|\mathcal M_m|\Big]
		\le
		\varepsilon\,
		\E\Big[\sup_{0\le n\le N_T}\|\omega_N^n\|_{\mathbb{H}^{-1}(\T^2)}^2\Big]
		+
		C_\varepsilon
		\sum_{n=0}^{N_T-1}\Dt\,\E\big[\|\omega_N^n\|_{\mathbb{H}^{-\alpha}(\T^2)}^2\big].
		\]
		By choosing \(\varepsilon=\frac14\) and inserting this bound into
		\eqref{eq:strong_energy_clean_5}, we arrive at
		\begin{equation}\label{eq:strong_energy_clean_6}
			\begin{aligned}
				&\frac14\,\E\Big[\sup_{0\le m\le N_T}\|\omega_N^{m}\|_{\mathbb{H}^{-1}(\T^2)}^2\Big]
				+c_0\sum_{n=0}^{N_T-1}\Dt\,\E\big[\|\omega_N^n\|_{\mathbb{H}^{-\alpha}(\T^2)}^2\big]
				\\
				&\qquad
				\le
				\frac12\|\omega_N^0\|_{\mathbb{H}^{-1}(\T^2)}^2
				+\frac{\Dt c_K}{2}\E\big[\|\omega_N^0\|_{\mathbb{L}^2(\T^2)}^2\big]
				\\&\qquad+
				C\sum_{n=0}^{N_T-1}\Dt\,\E\big[\|\omega_N^n\|_{\mathbb{H}^{-\alpha}(\T^2)}^2\big]
				+
				C_0\sum_{n=0}^{N_T-1}\Dt\,\E\big[\|\omega_N^n\|_{\mathbb{H}^{-1}(\T^2)}^2\big].
			\end{aligned}
		\end{equation}
		
		\medskip
		\noindent\textbf{Step 4: use the estimate from Theorem~\ref{thm:energy_keep_full}.} By using Theorem~\ref{thm:energy_keep_full} to estimate~\eqref{eq:strong_energy_clean_6}, we obtain
		\begin{align*}
			&\E\Big[\sup_{0\le n\le N_T}\|\omega_N^n\|_{\mathbb{H}^{-1}(\T^2)}^2\Big]
			+
			\sum_{n=0}^{N_T-1}\Dt\,\E\big[\|\omega_N^n\|_{\mathbb{H}^{-\alpha}(\T^2)}^2\big]\\
			&\le
			C\Big(
			\|\omega_N^0\|_{\mathbb{H}^{-1}(\T^2)}^2
			+
			\Dt c_K\,\E\big[\|\omega_N^0\|_{\mathbb{L}^2(\T^2)}^2\big]
			\Big),
		\end{align*}
		which is exactly \eqref{eq:strong_energy_keep}.
	\end{proof}
	\begin{remark}[Need of the implicit--explicit structure]
		\label{rem:necessity_implicit_euler}
		The particular implicit--explicit form of the fully discrete scheme
		\eqref{eq:scheme_keep} is important for the stability argument. We emphasize
		three points.
		
		\noindent
		\textup{(i)} The scheme can be tested with the unknown quantity
		\((-\Delta)^{-1}\omega_N^{n+1}\). This gives the exact discrete identity
		\eqref{1111}. The last term
		\[
		\frac12\|\omega_N^{n+1}-\omega_N^n\|_{\mathbb H^{-1}}^2
		\]
		in \eqref{1111} is crucial. In the stochastic part of the estimate, after adding
		and subtracting the old-time test function in \eqref{2222}, Young's inequality
		in \eqref{eq:young_half_corrected} produces exactly a term with the factor
		\[
		\frac12\|\omega_N^{n+1}-\omega_N^n\|_{\mathbb H^{-1}}^2,
		\]
		which can be absorbed by the corresponding positive term in \eqref{1111}. This
		absorption is one of the key advantages of testing at the implicit time level.
		
		\noindent
		\textup{(ii)} The mixed-time convection is compatible with this test. Hence the
		nonlinear transport term does not contribute to the \(\mathbb H^{-1}\)-energy
		balance; see also Remark~\ref{rem:mixed_time_conv}.
		
		\noindent
		\textup{(iii)} The implicit treatment of the It\^o correction gives the favorable
		coercive contribution \eqref{3333}. Thus the energy estimate
		\eqref{eq:uniform_energy_keep} contains the positive time-integrated
		\(\mathbb L^2\)-term generated by the implicit Laplacian. This term is important
		for obtaining the uniform bound on the right-continuous interpolant
		\(\overline\omega_j\) in Lemma~\ref{lem:right_endpoint_bound_conv}. Moreover, it
		combines with the quadratic variation of the stochastic transport term through
		the finite-dimensional coercivity estimate
		\eqref{eq:finite_dimensional_coercivity}; see
		\eqref{eq:energy_before_maurelli_corrected}.
		
		Together, these three features provide the stability mechanism behind
		Theorem~\ref{thm:strong_energy_keep}. They are tied to the specific
		implicit--explicit structure of \eqref{eq:scheme_keep} and are not available in
		the same form for a fully explicit discretization, or for a different time
		discretization which does not allow testing at the new time level while
		preserving the mixed-time convection cancellation and the implicit Laplacian
		coercivity.
	\end{remark}
	
	\section{Convergence of the fully discrete scheme to a weak martingale solution}
	\label{sec:convergence_martingale}
	
	In this section we prove convergence, along subsequences, of the fully discrete
	scheme \eqref{eq:scheme_keep} to a weak martingale solution of the limiting
	Euler equations~\eqref{eq:intro_strat} on \(\T^2\).
	
	\noindent
	In this section, we assume
	\[
	0<\alpha<\frac12 .
	\]
	Let as $j\to\infty,$
	\[
	\Dt_j\downarrow0,\qquad N_j\to\infty,\qquad K_j\to\infty.
	\]
	For simplicity we assume \(T/\Dt_j\in\N\) and set
	\[
	N_T^j:=\frac{T}{\Dt_j}.
	\]
	We also assume
	\begin{equation}\label{eq:resolution_condition_conv}
		K_j\ge \gamma \sqrt{2} N_j\qquad\text{for some }\, \gamma>1+\sqrt{2}.
	\end{equation}
	We use the shorthand
	\[
	\mathcal{P}_j:=\mathcal P_{N_j},
	\qquad
	\omega_j^0:=\mathcal{P}_j\omega_0.
	\]
	For general initial data \(\omega_0\in \mathbb{H}^{-1}(\T^2)\), we assume the standard
	well-preparedness condition
	\begin{equation}\label{eq:well_prepared_initial_data_conv}
		\sup_{j\ge1}
		\Big(
		\|\omega_j^0\|_{\mathbb{H}^{-1}(\T^2)}^2
		+
		\Dt_j\,c_{K_j}\|\omega_j^0\|_{\mathbb{L}^2}^2
		\Big)<\infty.
	\end{equation}
	This is automatic if \(\omega_0\in \mathbb{L}^2(\T^2)\).  If \(\omega_0\in \mathbb{H}^{-1}(\T^2)\), it is enough,
	for instance, to assume \(\sup_{j\in\mathbb{N}} \Dt_j N_j^2<\infty\) as a CFL-type condition to address the roughness of initial data $\omega_0$ that we allow in this work; see also the right-hand side of \eqref{eq:uniform_energy_keep}.
	
	\noindent
	\textbf{The limiting equation and the weak martingale formulation:} Recall that
	\[
	q(\boldsymbol{\xi})=\langle\boldsymbol{\xi}\rangle^{-(2+2\alpha)},
	\qquad
	c_\infty:=\sum_{m\in\Z^2\setminus\{{\bf 0}\}}q({\bf m})<\infty.
	\]
	The limiting It\^o equation on \(\T^2\) is
	\begin{equation}\label{eq:limit_Ito_equation_conv}
		\dd\omega
		+
		({\bf u}\cdot \nabla)\omega\,\dd t
		+
		\sum_{k\in\mathcal I}
		(\boldsymbol{\sigma}_k\cdot \nabla)\omega\,\dd{\bf W}_k(t)
		=
		\frac{c_\infty}{2}\Delta\omega\,\dd t,
		\qquad
		{\bf u}=\nabla^\perp(-\Delta)^{-1}\omega.
	\end{equation}
	Before proving convergence, we specify the martingale formulation of
	\eqref{eq:intro_strat}. The limit obtained from the fully discrete
	scheme \eqref{eq:scheme_keep} will be shown to satisfy this formulation.
	\begin{definition}[Weak martingale solution]
		\label{def:weak_martingale_solution_torus_conv}
		Let \(\omega_0\in \mathbb{H}^{-1}(\T^2)\) be mean-zero.  A weak martingale solution of
		\eqref{eq:limit_Ito_equation_conv} is a system
		\[
		\Big(
		\widetilde\Omega,\widetilde{\mathcal F},
		(\widetilde{\mathcal F}_t)_{t\in[0,T]},
		\widetilde{\mathbb P},
		(\widetilde{\bf W}_k)_{k\in\mathcal I},
		\widetilde\omega
		\Big)
		\]
		such that:
		
		\begin{enumerate}[label=\textup{(\roman*)}]
			\item
			\((\widetilde\Omega,\widetilde{\mathcal F},
			(\widetilde{\mathcal F}_t)_{t\in[0,T]},
			\widetilde{\mathbb P})\) is a filtered probability space satisfying the usual
			conditions;
			
			\item
			\((\widetilde{\bf W}_k)_{k\in\mathcal I}\) is a family of independent real
			\((\widetilde{\mathcal F}_t)\)-Brownian motions;
			
			\item
			\(\widetilde\omega\) is \((\widetilde{\mathcal F}_t)\)-adapted, mean-zero, and
			\[
			\widetilde\omega
			\in
			\mathbb{L}^2\big(\widetilde\Omega;\mathbb{L}^\infty([0,T];\mathbb{H}^{-1}(\T^2))\big)
			\cap
			\mathbb{L}^2\big(\widetilde\Omega;\mathbb{L}^2\bigl([0,T];\mathbb{H}^{-\alpha}(\T^2)\bigr)\big),
			\]
			with a continuous version in \(C\bigl([0,T];\mathbb{H}^{-5}(\T^2)\bigr)\);
			
			\item
			for every \(\varphi\in C^\infty(\T^2)\) and every \(t\in[0,T]\),
			\(\widetilde{\mathbb P}\)-a.s.,
			\begin{equation}\label{eq:weak_martingale_formulation_conv}
				\begin{aligned}
					\langle \widetilde\omega(t),\varphi\rangle
					&=
					\langle \omega_0,\varphi\rangle
					+
					\int_0^t
					\langle \widetilde\omega(s),\widetilde{\bf u}(s)\cdot \nabla\varphi\rangle\,\dd s
					+
					\frac{c_\infty}{2}
					\int_0^t
					\langle \widetilde\omega(s),\Delta\varphi\rangle\,\dd s
					\\
					&\quad
					+
					\sum_{k\in\mathcal I}
					\int_0^t
					\langle \widetilde\omega(s)\boldsymbol{\sigma}_k,\nabla\varphi\rangle\,
					\dd \widetilde{\bf W}_k(s),
					\qquad
					\widetilde{\bf u}=\nabla^\perp(-\Delta)^{-1}\widetilde\omega.
				\end{aligned}
			\end{equation}
		\end{enumerate}
	\end{definition}
	We now state the main convergence result of this section. It asserts that,
	after passing to a subsequence and changing the probability space, the fully discrete approximations converge to a weak martingale
	solution of the stochastic Euler equations. 
	\begin{theorem}[Convergence of the fully discrete scheme]
		\label{thm:convergence_fully_discrete_martingale}
		Let \(0<\alpha<1/2\), and let \(\omega_0\in\mathbb H^{-1}(\T^2)\) be mean‐zero.
		Assume that the parameters \((N_j,K_j,\Delta t_j)_{j\in\mathbb{N}}\) satisfy as $j\to \infty$
		\[
		N_j\to\infty,\qquad K_j\to\infty,\qquad \Delta t_j\to0,
		\qquad
		K_j\ge \sqrt{2}\,\gamma N_j,\quad \gamma>1+\sqrt{2},
		\]
		together with the CFL‐type condition $ \sup_{j\in\mathbb N}\Delta t_j N_j^2 < \infty .$ Then, there exist a filtered probability space
		$\bigl(\widetilde\Omega,\widetilde{\mathcal F}, (\widetilde{\mathcal F}_t)_{t\in[0,T]},\, (\widetilde{\bf W}_k)_{k\in\mathcal I},\widetilde{\mathbb P}\bigr),$ a subsequence (not relabeled) of fully-discrete
		approximations $\bigl(\big\{\widetilde{{\omega}}_{N_j}^n,
		\widetilde{\mathbf u}_{N_j}^n\}_{n=0}^{N_{T}^{j}})_{j\in \mathbb{N}}$ that satisfy the scheme
		\eqref{eq:scheme_keep} with parameters \((N_j,K_j,\Delta t_j)\) on a filtered probability space
		$\bigl(\widetilde\Omega,\widetilde{\mathcal F}, (\widetilde{\mathcal F}_t)_{t\in[0,T]},\, (\widetilde{\bf W}_k)_{k\in\mathcal I},\widetilde{\mathbb P}\bigr),$  and a \((\widetilde{\mathcal F}_t)\)-predictable process
		\((\widetilde\omega, \widetilde{\mathbf u})\in \mathbb{L}^2\bigl(\Omega;  \mathbb L^2\bigl([0,T];\mathbb H^{-\alpha}(\T^2)\times \mathbb{L}^2(\T^2;\mathbb R^2)\bigr)\bigr)\) such that
		\[
		\widetilde{\underline{\omega}}_j:=\sum_{n=0}^{N_T^j-1}\mathbf{1}_{[t_n,t_{n+1})}\widetilde{{\omega}}_{N_j}^n
		\longrightarrow
		\widetilde\omega
		\quad
		\widetilde{\mathbb P}-\text{a.s. in }
		\mathbb L^2\bigl([0,T];\mathbb H^{-5}(\T^2)\bigr),
		\]
		\[
		\widetilde{\underline{\omega}}_j
		\rightharpoonup
		\widetilde\omega
		\quad
		\widetilde{\mathbb P}-\text{a.s. in }
		\mathbb L^2\bigl([0,T];\mathbb H^{-\alpha}(\T^2)\bigr),
		\]
		and, 
		\[
		\widetilde{\mathbf u}_j:=\sum_{n=0}^{N_T^j-1}\mathbf{1}_{[t_n,t_{n+1})}\widetilde{{\bf u}}_{N_j}^n
		\longrightarrow
		\widetilde{\mathbf u}
		\quad
		\widetilde{\mathbb P}-\text{a.s. in }
		\mathbb L^2\bigl([0,T];\mathbb L^2(\T^2;\mathbb R^2)\bigr).
		\]
		Moreover, the system
		\[
		\Big(
		\widetilde\Omega,\widetilde{\mathcal F},
		(\widetilde{\mathcal F}_t)_{t\in[0,T]},
		\widetilde{\mathbb P},
		(\widetilde{\bf W}_k)_{k\in\mathcal I},
		\widetilde\omega
		\Big)
		\]
		is a weak martingale solution of the stochastic Euler equations in the sense of
		Definition~\ref{def:weak_martingale_solution_torus_conv}.
	\end{theorem}
	\begin{proof}
		The proof is organized over the remaining subsections of this section. We first
		establish the required tightness and compactness properties of the fully
		discrete approximations, then pass to the limit on a new probability space, and
		finally identify the limit as a weak martingale solution in the sense of
		Definition~\ref{def:weak_martingale_solution_torus_conv}.


		\subsection{Discrete interpolants}
		\label{subsec:interpolants_conv}
		
		For each \(j\in\mathbb{N}\), let \((\omega_{N_j}^n)_{n=0}^{N_T^j}\) be the solution of
		\eqref{eq:scheme_keep}. We define the left- and right-endpoint piecewise
		constant interpolants by
		\[
		\underline\omega_j(t):=\omega_{N_j}^n,
		\qquad
		\overline\omega_j(t):=\omega_{N_j}^{n+1},
		\qquad
		t\in[t_j^n,t_j^{n+1}),\quad n=0,\dots,N_T^j-1,
		\]
		with the convention
		\[
		\underline\omega_j(T):=\omega_{N_j}^{N_T^j-1},
		\qquad
		\overline\omega_j(T):=\omega_{N_j}^{N_T^j}.
		\]
		We also set 
		\[
		\overline{\mathbf u}_j(t)
		:=
		\nabla^\perp(-\Delta)^{-1}\overline\omega_j(t)\qquad\forall\,t\in[0,T].
		\]
		We define the piecewise constant drift by
		\begin{equation}\label{eq:drift_j_def_conv}
			F_j(t)
			:=
			-\mathcal P_{j}\big[
			(\overline{\mathbf u}_j(t)\cdot\nabla)\underline\omega_j(t)
			\big]
			+
			\frac{c_{K_j}}{2}\Delta\overline\omega_j(t)\qquad\forall\,t\in[0,T],
		\end{equation}
		and the continuous martingale by
		\begin{equation}\label{eq:martingale_j_def_conv}
			\mathcal M_j(t)
			:=
			-\sum_{k\in\mathcal I^{K_j}}
			\int_0^t
			\mathcal P_{N_j}\big[
			(\boldsymbol{\sigma}_k\cdot\nabla)\underline\omega_j(s)
			\big]\,
			\dd {\bf W}_k^j(s)\qquad\forall\,t\in[0,T].
		\end{equation}
		The continuous equation-based interpolant is then given by
		\begin{equation}\label{eq:omega_sharp_def_conv}
			\omega_j^\sharp(t)
			:=
			\omega_{N_j}^0+\int_0^t F_j(s)\,\dd s+\mathcal M_j(t)\qquad\forall\,t\in[0,T].
		\end{equation}
		By construction,
		\[
		\omega_j^\sharp(t_j^n)=\omega_{N_j}^n,
		\qquad n=0,\dots,N_T^j.
		\]
		For every \(\varphi\in C^\infty(\T^2)\), define the scalar martingale
		\[
		\mathcal M_j^\varphi(t):=\langle \mathcal M_j(t),\varphi\rangle\qquad\forall\,t\in[0,T].
		\]
		By using orthogonality of \(\mathcal P_{N_j}\), the divergence-free property
		of \(\boldsymbol{\sigma}_k\), and integration by parts, we obtain
		\begin{equation}\label{eq:Mjphi_integrand_conv}
			\mathcal M_j^\varphi(t)
			=
			\sum_{k\in\mathcal I^{K_j}}
			\int_0^t
			\left\langle
			\underline\omega_j(s)\boldsymbol{\sigma}_k,
			\nabla(\mathcal P_{j}\varphi)
			\right\rangle
			\dd{\bf W}_k^j(s)\qquad\forall\,t\in[0,T].
		\end{equation}
		
		\noindent
		We record the following consequence of the uniform energy estimate in
		Theorem~\ref{thm:strong_energy_keep}. It provides the
		\(\mathbb L^2(0,T;\mathbb H^{-\alpha})\)-bounds for the left- and
		right-continuous interpolants needed in the compactness argument.
		\begin{lemma}[Right-endpoint \(\mathbb{H}^{-\alpha}\) bound]
			\label{lem:right_endpoint_bound_conv} There exists a constant $C>0$ such that
			\begin{equation}\label{eq:right_endpoint_bound_conv}
				\sup_{j\in\mathbb{N}}\E\left[
				\sum_{n=0}^{N_T^j-1}\Dt_j\|\omega_{N_j}^{n+1}\|_{\mathbb{H}^{-\alpha}(\T^2)}^2
				\right]\le C.
			\end{equation}
			In particular,
			\begin{equation}\label{eq:interpolant_Halpha_bounds_conv}
				\sup_{j\in\mathbb{N}}\E\Big[
				\|\underline\omega_j\|_{\mathbb{L}^2\bigl([0,T];\mathbb{H}^{-\alpha}(\T^2)\bigr)}^2
				+
				\|\overline\omega_j\|_{\mathbb{L}^2\bigl([0,T];\mathbb{H}^{-\alpha}(\T^2)\bigr)}^2
				\Big]<\infty.
			\end{equation}
		\end{lemma}
		
		\begin{proof}
			Fix $j\in\mathbb{N}$. The estimate for the left-endpoint interpolant follows directly from
			Theorem~\ref{thm:strong_energy_keep}:
			\begin{align}\label{new}
				\E\left[
				\sum_{n=0}^{N_T^j-1}\Dt_j
				\|\omega_{N_j}^{n}\|_{\mathbb{H}^{-\alpha}(\T^2)}^2
				\right]
				\le C.
			\end{align}
			We now prove the corresponding right-endpoint estimate. By shifting the summation
			index,
			\[
			\sum_{n=0}^{N_T^j-1}\Dt_j
			\|\omega_{N_j}^{n+1}\|_{\mathbb{H}^{-\alpha}(\T^2)}^2
			=
			\sum_{m=1}^{N_T^j}\Dt_j
			\|\omega_{N_j}^{m}\|_{\mathbb{H}^{-\alpha}(\T^2)}^2.
			\]
			Hence
			\[
			\sum_{m=1}^{N_T^j}\Dt_j
			\|\omega_{N_j}^{m}\|_{\mathbb{H}^{-\alpha}(\T^2)}^2
			\le
			\sum_{m=0}^{N_T^j-1}\Dt_j
			\|\omega_{N_j}^{m}\|_{\mathbb{H}^{-\alpha}(\T^2)}^2
			+
			\Dt_j
			\|\omega_{N_j}^{N_T^j}\|_{\mathbb{H}^{-\alpha}(\T^2)}^2.
			\]
			The first term on the right-hand side is already bounded in expectation by
			\eqref{new}. It remains to control the final endpoint.
			
			\noindent
			By interpolation between \(\mathbb{H}^{-1}(\T^2)\) and \(\mathbb{L}^2(\T^2)\),
			\[
			\|f\|_{\mathbb{H}^{-\alpha}(\T^2)}^2
			\le
			C
			\|f\|_{\mathbb{H}^{-1}(\T^2)}^{2\alpha}
			\|f\|_{ \mathbb{L}^2(\T^2)}^{2(1-\alpha)}.
			\]
			Multiplying by \(\Dt_j\), we get
			\[
			\Dt_j\|f\|_{\mathbb{H}^{-\alpha}(\T^2)}^2
			\le
			C
			\Dt_j^\alpha c_{K_j}^{-(1-\alpha)}
			\big(\|f\|_{\mathbb{H}^{-1}(\T^2)}^2\big)^\alpha
			\big(\Dt_j c_{K_j}\|f\|_{ \mathbb{L}^2(\T^2)}^2\big)^{1-\alpha}.
			\]
			Since \(K_j\ge1\), we have
			\[
			c_{K_j}\ge c_1>0,
			\]
			with \(c_1\) independent of \(j\). Also \(\Dt_j\le1\). Therefore
			\[
			\Dt_j^\alpha c_{K_j}^{-(1-\alpha)}\le C.
			\]
			Young's inequality then yields
			\[
			\Dt_j\|f\|_{\mathbb{H}^{-\alpha}(\T^2)}^2
			\le
			C\left(
			\|f\|_{\mathbb{H}^{-1}(\T^2)}^2
			+
			\Dt_j c_{K_j}\|f\|_{ \mathbb{L}^2(\T^2)}^2
			\right).
			\]
			By applying this with \(f=\omega_{N_j}^{N_T^j}\), taking expectations, and using the
			uniform bound for the second term in the estimate of \eqref{eq:uniform_energy_keep}, we obtain
			\[
			\E\left[
			\Dt_j
			\|\omega_{N_j}^{N_T^j}\|_{\mathbb{H}^{-\alpha}(\T^2)}^2
			\right]
			\le C.
			\]
			Combining this endpoint bound with the left-endpoint estimate gives
			\[
			\E\left[
			\sum_{n=0}^{N_T^j-1}\Dt_j
			\|\omega_{N_j}^{n+1}\|_{\mathbb{H}^{-\alpha}(\T^2)}^2
			\right]
			\le C.
			\]
			This proves \eqref{eq:right_endpoint_bound_conv}.
			
			\noindent
			Finally, by the definitions of the piecewise constant interpolants,
			\[
			\|\underline\omega_j\|_{\mathbb{L}^2\bigl([0,T];\mathbb{H}^{-\alpha}(\T^2)\bigr)}^2
			=
			\sum_{n=0}^{N_T^j-1}\Dt_j
			\|\omega_{N_j}^{n}\|_{\mathbb{H}^{-\alpha}(\T^2)}^2,
			\]
			while
			\[
			\|\overline\omega_j\|_{\mathbb{L}^2\bigl([0,T];\mathbb{H}^{-\alpha}(\T^2)\bigr)}^2
			=
			\sum_{n=0}^{N_T^j-1}\Dt_j
			\|\omega_{N_j}^{n+1}\|_{\mathbb{H}^{-\alpha}(\T^2)}^2.
			\]
			The two estimates above therefore imply
			\eqref{eq:interpolant_Halpha_bounds_conv}.
		\end{proof}
		

		\subsection{Localization and time regularity in \(\mathbb{H}^{-4}(\T^2)\)}
		\label{subsec:time_regularity_conv}
		In this subsection, by following the strategy of \cite[Sec.~5]{CoghiMaurelli2026}, we
		introduce a discrete stopping time. This localization is essential for proving
		tightness of the approximation components in the appropriate path spaces.
		
		\noindent
		For \(n=0,\dots,N_T^j\), we define
		\[
		\mathcal E_j^n
		:=
		\sup_{0\le m\le n}\|\omega_{N_j}^m\|_{\mathbb{H}^{-1}(\T^2)}^2
		+
		\sum_{m=0}^{n}\Dt_j\|\omega_{N_j}^{m}\|_{\mathbb{H}^{-\alpha}(\T^2)}^2.
		\]
		For \(R>0\), we define the stopping time
		\begin{equation}\label{eq:stopping_time_conv}
			\tau_{j,R}
			:=
			\inf\{t_j^{n}:\ 0\le n\le N_T^j,\ \mathcal E_j^n>R\}.
		\end{equation}
		\noindent
		We define the stopped drift and stopped martingale by 
		\begin{equation}\label{eq:stopped_drift_conv}
			F_{j,R}(t):=\mathbf 1_{\{t< \tau_{j,R}\}}\,F_j(t)\qquad\forall\,t\in[0,T],
		\end{equation}
		and
		\begin{equation}\label{eq:stopped_martingale_conv}
			\mathcal M_{j,R}(t)
			:=
			-\sum_{k\in\mathcal I^{K_j}}
			\int_0^t
			\mathbf 1_{\{s<\tau_{j,R}\}}\,
			\mathcal{P}_j[(\boldsymbol{\sigma}_k\cdot \nabla)\underline\omega_j(s)]\,\dd{\bf W}_k(s)\qquad\forall\,t\in[0,T].
		\end{equation}
		Finally, define the stopped continuous interpolant for all $t\in[0,T]$,
		\begin{equation}\label{eq:stopped_interpolant_conv}
			\omega_{j,R}^\sharp(t)
			:=
			\omega_j^0+\int_0^tF_{j,R}(s)\,\dd s+\mathcal M_{j,R}(t).
		\end{equation}
		On the event \(\{\tau_{j,R}=T\}\), we have
		\[
		\omega_{j,R}^\sharp=\omega_j^\sharp
		\qquad\text{on }[0,T].
		\]
		\begin{lemma}[High probability of no stopping]
			\label{lem:stop_high_prob_conv} There exists a constant $C>0$ such that
			for every \(R>0\),
			\[
			\sup_{j\in\mathbb{N}} \mathbb P\bigl[\tau_{j,R}<T\bigr]\le \frac{C}{R}.
			\]
			In particular,
			\[
			\lim_{R\to\infty}\sup_{j\in\mathbb{N}}\mathbb P\bigl[\tau_{j,R}<T\bigr ]=0.
			\]
		\end{lemma}
		
		\begin{proof}
			Since \((\mathcal E_j^n)_{n=0}^{N_T^j}\) is non-decreasing in \(n\), the definition
			of \(\tau_{j,R}\) implies
			\[
			\{\tau_{j,R}<T\}\subset \{\mathcal E_j^{N_T^j}>R\}.
			\]
			Hence, by Markov's inequality,
			\[
			\mathbb P\bigl[\tau_{j,R}<T\bigr]
			\le
			\frac{\mathbb E[\mathcal E_j^{N_T^j}]}{R}.
			\]
			By Theorem~\ref{thm:strong_energy_keep} and Lemma~\ref{lem:right_endpoint_bound_conv},
			we obtain
			\[
			\mathbb E[\mathcal E_j^{N_T^j}]\le C
			\]
			which holds uniformly in \(j\). Therefore
			\[
			\sup_{j\in\mathbb{N}}\mathbb P\bigl[\tau_{j,R}<T\bigr]\le \frac{C}{R}.
			\]
			The last claim follows by letting \(R\to\infty\).
		\end{proof}

		\begin{lemma}[Localized drift bound in expectation]
			\label{lem:localized_drift_bound_conv}
			Let
			\[
			q_*:=2(1-\alpha)>1.
			\]
			For every \(R>0\), there exists \(C_R>0\) such that
			\[
			\sup_{j\in\mathbb{N}}
			\mathbb E\Big[
			\|F_{j,R}\|_{\mathbb{L}^{q_*}\left([0,T];\mathbb{H}^{-4}(\T^2)\right)}^{q_*}
			\Big]
			\le C_R.
			\]
		\end{lemma}
		
		\begin{proof}
			We estimate separately the nonlinear part and the Laplacian part of the stopped
			drift \(F_{j,R}\).
			
			\medskip
			\noindent\textbf{Step 1: the nonlinear part.}
			Fix \(\phi\in \mathbb H^4(\T^2)\) with
			\(\|\phi\|_{\mathbb H^4}=1\). Since \(\mathcal P_j\) is self-adjoint and
			\(\operatorname{div}\overline{\mathbf u}_j=0\), we have
			\[
			\begin{aligned}
				\left\langle
				\mathcal P_j\bigl[(\overline{\mathbf u}_j\cdot\nabla)\underline\omega_j\bigr],
				\phi
				\right\rangle
				&=
				\left\langle
				(\overline{\mathbf u}_j\cdot\nabla)\underline\omega_j,
				\mathcal P_j\phi
				\right\rangle  \\
				&=
				-\left\langle
				\underline\omega_j,
				\overline{\mathbf u}_j\cdot\nabla(\mathcal P_j\phi)
				\right\rangle .
			\end{aligned}
			\]
			Therefore, by duality between
			\(\mathbb H^{-(1-\alpha)}(\T^2)\) and
			\(\mathbb H^{1-\alpha}(\T^2)\),
			\[
			\left|
			\left\langle
			\mathcal P_j\bigl[(\overline{\mathbf u}_j\cdot\nabla)\underline\omega_j\bigr],
			\phi
			\right\rangle
			\right|
			\le
			\|\underline\omega_j\|_{\mathbb H^{-(1-\alpha)}}
			\,
			\|\overline{\mathbf u}_j\cdot\nabla(\mathcal P_j\phi)\|_{\mathbb H^{1-\alpha}} .
			\]
			Since \(\mathcal P_j\) is uniformly bounded on \(\mathbb H^4(\T^2)\), we have
			\[
			\|\mathcal P_j\phi\|_{\mathbb H^4}
			\le
			\|\phi\|_{\mathbb H^4}
			=
			1 .
			\]
			Moreover, in two dimensions,
			\[
			\mathbb H^4(\T^2)\hookrightarrow W^{2,\infty}(\T^2),
			\]
			and hence
			\[
			\|\nabla(\mathcal P_j\phi)\|_{W^{1,\infty}}
			\le C\|\mathcal P_j\phi\|_{\mathbb H^4}
			\le C .
			\]
			Since multiplication by a \(W^{1,\infty}\)-function is bounded on
			\(\mathbb H^{1-\alpha}(\T^2)\), uniformly in \(j\), we obtain
			\[
			\|\overline{\mathbf u}_j\cdot\nabla(\mathcal P_j\phi)\|_{\mathbb H^{1-\alpha}}
			\le
			C\|\overline{\mathbf u}_j\|_{\mathbb H^{1-\alpha}} .
			\]
			Combining the preceding estimates gives
			\[
			\left|
			\left\langle
			\mathcal P_j\bigl[(\overline{\mathbf u}_j\cdot\nabla)\underline\omega_j\bigr],
			\phi
			\right\rangle
			\right|
			\le
			C
			\|\underline\omega_j\|_{\mathbb H^{-(1-\alpha)}}
			\|\overline{\mathbf u}_j\|_{\mathbb H^{1-\alpha}} .
			\]
			Taking the supremum over all \(\phi\in\mathbb H^4(\T^2)\) with
			\(\|\phi\|_{\mathbb H^4}=1\), we conclude that
			\[
			\left\|
			\mathcal P_j\bigl[(\overline{\mathbf u}_j\cdot\nabla)\underline\omega_j\bigr]
			\right\|_{\mathbb H^{-4}}
			\le
			C
			\|\underline\omega_j\|_{\mathbb H^{-(1-\alpha)}}
			\|\overline{\mathbf u}_j\|_{\mathbb H^{1-\alpha}} .
			\]
			By the Biot--Savart estimate,
			\[
			\|\overline {\bf u}_j\|_{\mathbb{H}^{1-\alpha}}
			\le
			C\|\overline\omega_j\|_{\mathbb{H}^{-\alpha}(\T^2)}.
			\]
			Moreover, setting
			\[
			\theta:=\frac{\alpha}{1-\alpha}\in(0,1),
			\]
			we have
			\[
			-(1-\alpha)=(1-\theta)(-1)+\theta(-\alpha),
			\]
			and therefore
			\[
			\|\underline\omega_j\|_{\mathbb{H}^{-(1-\alpha)}}
			\le
			C
			\|\underline\omega_j\|_{\mathbb{H}^{-1}(\T^2)}^{1-\theta}
			\|\underline\omega_j\|_{\mathbb{H}^{-\alpha}(\T^2)}^\theta.
			\]
			Hence
			\[
			\|\mathcal{P}_j[(\overline {\bf u}_j\cdot \nabla)\underline\omega_j]\|_{\mathbb{H}^{-4}(\T^2)}
			\le
			C
			\|\underline\omega_j\|_{\mathbb{H}^{-1}(\T^2)}^{1-\theta}
			\|\underline\omega_j\|_{\mathbb{H}^{-\alpha}(\T^2)}^\theta
			\|\overline\omega_j\|_{\mathbb{H}^{-\alpha}(\T^2)}.
			\]
			Now on the event \(\{t<\tau_{j,R}\}\), the stopping-time definition gives
			\[ 
			\|\underline\omega_j(t)\|_{\mathbb{H}^{-1}(\T^2)}^2\le R.
			\]
			Therefore with $C_R=C\cdot R$,
			\begin{align}\label{new31}
				\mathbf 1_{\{t<\tau_{j,R}\}}
				\|\mathcal{P}_j[(\overline {\bf u}_j\cdot \nabla)\underline\omega_j](t)\|_{\mathbb{H}^{-4}(\T^2)}
				\le
				C_R
				\mathbf 1_{\{t<\tau_{j,R}\}}
				\|\underline\omega_j(t)\|_{\mathbb{H}^{-\alpha}(\T^2)}^\theta
				\|\overline\omega_j(t)\|_{\mathbb{H}^{-\alpha}(\T^2)}.
			\end{align}
			Raise this to the power \(q_*\). Since
			\[
			q_*=\frac{2}{1+\theta}=2(1-\alpha),
			\qquad
			\frac{\theta q_*}{2}+\frac{q_*}{2}=1,
			\]
			H\"older's inequality in time in \eqref{new31} yields
			\[
			\begin{aligned}
				&\int_0^T
				\mathbf 1_{\{t<\tau_{j,R}\}}
				\|\mathcal{P}_j[(\overline {\bf u}_j\cdot \nabla)\underline\omega_j](t)\|_{\mathbb{H}^{-4}(\T^2)}^{q_*}\,\dd t
				\\
				&\qquad\le
				C_R
				\left(
				\int_0^T
				\mathbf 1_{\{t<\tau_{j,R}\}}
				\|\underline\omega_j(t)\|_{\mathbb{H}^{-\alpha}(\T^2)}^2\,\dd t
				\right)^{\theta q_*/2}
				\left(
				\int_0^T
				\mathbf 1_{\{t<\tau_{j,R}\}}
				\|\overline\omega_j(t)\|_{\mathbb{H}^{-\alpha}(\T^2)}^2\,\dd t
				\right)^{q_*/2}.
			\end{aligned}
			\]
			By using~\eqref{eq:stopping_time_conv},
			\[
			\int_0^T
			\mathbf 1_{\{t<\tau_{j,R}\}}
			\|\underline\omega_j(t)\|_{\mathbb{H}^{-\alpha}(\T^2)}^2\,\dd t
			\le R
			\qquad\text{a.s.}
			\]
			Hence, now with $C_R=C\cdot R^2$
			\[
			\begin{aligned}
				&\int_0^T
				\mathbf 1_{\{t<\tau_{j,R}\}}
				\|\mathcal{P}_j[(\overline {\bf u}_j\cdot \nabla)\underline\omega_j](t)\|_{\mathbb{H}^{-4}(\T^2)}^{q_*}\,\dd t
				\\
				&\qquad\le
				C_R
				\left(
				\int_0^T
				\|\overline\omega_j(t)\|_{\mathbb{H}^{-\alpha}(\T^2)}^2\,\dd t
				\right)^{q_*/2}.
			\end{aligned}
			\]
			Taking expectations and using \(q_*/2\le 1\), we obtain by Jensen's inequality
			\[
			\begin{aligned}
				&\mathbb E\Big[
				\|\mathbf 1_{\{t<\tau_{j,R}\}}
				\mathcal{P}_j[(\overline {\bf u}_j\cdot \nabla)\underline\omega_j]\|_{\mathbb{L}^{q_*}\left([0,T];\mathbb{H}^{-4}(\T^2)\right)}^{q_*}
				\Big]
				\\
				&\qquad\le
				C_R
				\mathbb E\left[
				\left(
				\int_0^T
				\|\overline\omega_j(t)\|_{\mathbb{H}^{-\alpha}(\T^2)}^2\,\dd t
				\right)^{q_*/2}
				\right]
				\\
				&\qquad\le
				C_R
				\left(
				\mathbb E
				\int_0^T
				\|\overline\omega_j(t)\|_{\mathbb{H}^{-\alpha}(\T^2)}^2\,\dd t
				\right)^{q_*/2}.
			\end{aligned}
			\]
			By using~\eqref{eq:interpolant_Halpha_bounds_conv}, the right-hand side is bounded
			uniformly in \(j\). Thus
			\[
			\sup_{j\in\mathbb{N}}
			\mathbb E\Big[
			\|\mathbf 1_{\{t<\tau_{j,R}\}}
			\mathcal{P}_j[(\overline {\bf u}_j\cdot \nabla)\underline\omega_j]\|_{\mathbb{L}^{q_*}\left([0,T];\mathbb{H}^{-4}(\T^2)\right)}^{q_*}
			\Big]
			\le C_R.
			\]
			
			\medskip
			\noindent\textbf{Step 2: the Laplacian part.}
			Recall that the Laplacian term in \(F_{j,R}\) is
			\(\mathbf 1_{\{t<\tau_{j,R}\}}\frac{c_{K_j}}{2}\Delta\overline\omega_j\).
			Since \(c_{K_j}\le c_\infty\), it is enough to estimate
			\(\mathbf 1_{\{t<\tau_{j,R}\}}\Delta\overline\omega_j\). We have
			\[
			\|\Delta\overline\omega_j\|_{\mathbb{H}^{-4}(\T^2)}
			=
			\|\overline\omega_j\|_{\mathbb{H}^{-2}}
			\le
			C\|\overline\omega_j\|_{\mathbb{H}^{-\alpha}(\T^2)},
			\]
			because \(-2<-\alpha\) for \(0<\alpha<1\). Therefore
			\[
			\|\mathbf 1_{\{t<\tau_{j,R}\}}\Delta\overline\omega_j\|_{\mathbb{L}^{q_*}\left([0,T];\mathbb{H}^{-4}(\T^2)\right)}^{q_*}
			\le
			C
			\int_0^T
			\|\overline\omega_j(t)\|_{\mathbb{H}^{-\alpha}(\T^2)}^{q_*}\,\dd t.
			\]
			Since \(q_*<2\),H\"older's inequality implies
			\[
			\int_0^T
			\|\overline\omega_j(t)\|_{\mathbb{H}^{-\alpha}(\T^2)}^{q_*}\,\dd t
			\le
			T^{1-q_*/2}
			\left(
			\int_0^T
			\|\overline\omega_j(t)\|_{\mathbb{H}^{-\alpha}(\T^2)}^2\,\dd t
			\right)^{q_*/2}.
			\]
			Taking expectations and using again the uniform energy estimate~\eqref{eq:interpolant_Halpha_bounds_conv}, we obtain
			\[
			\sup_{j\in\mathbb{N}}
			\mathbb E\Big[
			\|\mathbf 1_{\{t<\tau_{j,R}\}}\Delta\overline\omega_j\|_{\mathbb{L}^{q_*}\left([0,T];\mathbb{H}^{-4}(\T^2)\right)}^{q_*}
			\Big]
			\le C_R.
			\]
			
			\medskip
			\noindent\textbf{Step 3: conclusion.}
			By combining the nonlinear and Laplacian estimates, we conclude that
			\[
			\sup_{j\in\mathbb{N}}
			\mathbb E\Big[
			\|F_{j,R}\|_{\mathbb{L}^{q_*}\left([0,T];\mathbb{H}^{-4}(\T^2)\right)}^{q_*}
			\Big]
			\le C_R.
			\]
			This proves the lemma.
		\end{proof}
		
		\begin{lemma}[LocalizedH\"older continuity in time]
			\label{lem:localized_holder_conv}
			Let \(R>0\) and let \(p>2\). Then for every
			\[
			0<\gamma<
			\min\left\{
			1-\frac1{q_*},\,
			\frac12-\frac1p
			\right\},
			\qquad q_*:=2(1-\alpha),
			\]
			there exists a constant \(C_{R,p,\gamma}>0\) such that
			\[
			\sup_{j\in\mathbb{N}}
			\mathbb E\Big[
			\|\omega_{j,R}^\sharp\|_{C^\gamma\left([0,T];\mathbb{H}^{-4}(\T^2)\right)}^{\,q_*}
			\Big]
			\le
			C_{R,p,\gamma}.
			\]
		\end{lemma}
		
		\begin{proof}
			Recall \eqref{eq:omega_sharp_def_conv}.	We estimate separately the drift part and the martingale part of \eqref{eq:omega_sharp_def_conv}.
			
			\medskip
			\noindent\textbf{Step 1: the drift part.}
			Let
			\[
			D_{j,R}(t):=\int_0^tF_{j,R}({r})\,\dd r.
			\]
			For \(0\le s<t\le T\),H\"older's inequality gives
			\[
			\|D_{j,R}(t)-D_{j,R}(s)\|_{\mathbb{H}^{-4}(\T^2)}
			=
			\left\|\int_s^tF_{j,R}({r})\,\dd r\right\|_{\mathbb{H}^{-4}(\T^2)}
			\le
			|t-s|^{1-1/q_*}\|F_{j,R}\|_{\mathbb{L}^{q_*}\left([0,T];\mathbb{H}^{-4}(\T^2)\right)}.
			\]
			Hence
			\[
			[D_{j,R}]_{C^{1-1/q_*}([0,T];\mathbb{H}^{-4}(\T^2))}
			\le
			\|F_{j,R}\|_{\mathbb{L}^{q_*}\left([0,T];\mathbb{H}^{-4}(\T^2)\right)},
			\]
			where $[\cdot]_{C^{1-1/q_*}([0,T];\mathbb{H}^{-4}(\T^2))}$ denotes the usual semi-norm on $C^{1-1/q_*}([0,T];\mathbb{H}^{-4}(\T^2))$.
			Moreover,
			\[
			\|D_{j,R}(t)\|_{\mathbb{H}^{-4}(\T^2)}
			\le
			\int_0^t\|F_{j,R}({r})\|_{\mathbb{H}^{-4}(\T^2)}\,\dd r
			\le
			T^{1-1/q_*}\|F_{j,R}\|_{\mathbb{L}^{q_*}\left([0,T];\mathbb{H}^{-4}(\T^2)\right)},
			\]
			so
			\[
			\|D_{j,R}\|_{C\bigl([0,T];\mathbb{H}^{-4}(\T^2)\bigr)}
			\le
			T^{1-1/q_*}\|F_{j,R}\|_{\mathbb{L}^{q_*}\left([0,T];\mathbb{H}^{-4}(\T^2)\right)}.
			\]
			Therefore
			\[
			\|D_{j,R}\|_{C^{1-1/q_*}\bigl([0,T];\mathbb{H}^{-4}(\T^2)\bigr)}
			\le
			C_T\|F_{j,R}\|_{\mathbb{L}^{q_*}\left([0,T];\mathbb{H}^{-4}(\T^2)\right)}.
			\]
			By Lemma~\ref{lem:localized_drift_bound_conv}, we therefore conclude that
			
			\[
			\mathbb E\Big[
			\|D_{j,R}\|_{C^{1-1/q_*}\bigl([0,T];\mathbb{H}^{-4}(\T^2)\bigr)}^{q_*}
			\Big]
			\le C_{R,T}.
			\]
			Finally, if
			\[
			0<\gamma<1-\frac1{q_*},
			\]
			the continuous embedding
			\[
			C^{1-1/q_*}\bigl([0,T];\mathbb{H}^{-4}(\T^2)\bigr)
			\hookrightarrow
			C^\gamma\bigl([0,T];\mathbb{H}^{-4}(\T^2)\bigr)
			\]
			then implies
			\[
			\mathbb E\Big[
			\|D_{j,R}\|_{C^\gamma\left([0,T];\mathbb{H}^{-4}(\T^2)\right)}^{q_*}
			\Big]
			\le
			C_{R,T,\gamma}.
			\]
			\medskip
			\noindent\textbf{Step 2: the martingale part.}
			Let \(0\le s<t\le T\). By the Burkholder--Davis--Gundy inequality,
			\[
			\E\big[\|\mathcal M_{j,R}(t)-\mathcal M_{j,R}(s)\|_{\mathbb{H}^{-4}(\T^2)}^p\big]
			\le
			C_p
			\E\Bigg[
			\int_s^t
			\sum_{k\in\mathcal I^{K_j}}
			\mathbf 1_{\{r<\tau_{j,R}\}}
			\|\mathcal{P}_j[(\boldsymbol{\sigma}_k\cdot \nabla)\underline\omega_j({r})]\|_{\mathbb{H}^{-4}(\T^2)}^2\,\dd r
			\Bigg]^{p/2}.
			\]
			By using Lemma~\ref{lem:noise_estimate_conv}, we obtain
			\[
			\E\big[\|\mathcal M_{j,R}(t)-\mathcal M_{j,R}(s)\|_{\mathbb{H}^{-4}(\T^2)}^p\big]
			\le
			C_p
			\E\Bigg[
			\int_s^t
			\mathbf 1_{\{r<\tau_{j,R}\}}
			\|\underline\omega_j({r})\|_{\mathbb{H}^{-1}(\T^2)}^2\,\dd r
			\Bigg]^{p/2}.
			\]
			On the set \(\{r<\tau_{j,R}\}\), the stopping-time definition~\eqref{eq:stopping_time_conv} gives
			\[
			\|\underline\omega_j({\bf r})\|_{\mathbb{H}^{-1}(\T^2)}^2\le R.
			\]
			Hence
			\[
			\E\big[\|\mathcal M_{j,R}(t)-\mathcal M_{j,R}(s)\|_{\mathbb{H}^{-4}(\T^2)}^p\big]
			\le
			C_{R,p}|t-s|^{p/2}.
			\]
			Since \(\mathcal M_{j,R}({\bf 0})=0\), Kolmogorov's continuity criterion in the Hilbert
			space \(\mathbb{H}^{-4}(\T^2)\) yields, for every
			\[
			0<\gamma<\frac12-\frac1p,
			\]
			a constant \(C_{R,p,\gamma}\) such that
			\begin{align}\label{new40}
				\E\big[\|\mathcal M_{j,R}\|_{C^\gamma\left([0,T];\mathbb{H}^{-4}(\T^2)\right)}^p\big]\le C_{R,p,\gamma}.
			\end{align}
			
			\medskip
			\medskip
			\noindent\textbf{Step 3: conclusion.}
			Since \(\omega_j^0=\mathcal{P}_j\omega_0\) is deterministic and uniformly bounded in
			\(\mathbb{H}^{-4}(\T^2)\), and
			\[
			\omega_{j,R}^\sharp=\omega_j^0+D_{j,R}+\mathcal M_{j,R},
			\]
			we estimate the \(C^\gamma\bigl([0,T];\mathbb{H}^{-4}(\T^2)\bigr)\)-norm in the exponent \(q_*\). We obtain
			\[
			\|\omega_{j,R}^\sharp\|_{C^\gamma\left([0,T];\mathbb{H}^{-4}(\T^2)\right)}^{q_*}
			\le
			C\Big(
			\|\omega_j^0\|_{\mathbb{H}^{-4}(\T^2)}^{q_*}
			+
			\|D_{j,R}\|_{C^\gamma\left([0,T];\mathbb{H}^{-4}(\T^2)\right)}^{q_*}
			+
			\|\mathcal M_{j,R}\|_{C^\gamma\left([0,T];\mathbb{H}^{-4}(\T^2)\right)}^{q_*}
			\Big).
			\]
			The first term is uniformly bounded in \(j\). The second term is bounded by Step 1, and so is the last term in Step 2, since for \(q_*<p\),H\"older's inequality and \eqref{new40} imply
			\[
			\mathbb E\Big[
			\|\mathcal M_{j,R}\|_{C^\gamma\left([0,T];\mathbb{H}^{-4}(\T^2)\right)}^{q_*}
			\Big]
			\le
			C
			\left(
			\mathbb E\Big[
			\|\mathcal M_{j,R}\|_{C^\gamma\left([0,T];\mathbb{H}^{-4}(\T^2)\right)}^{p}
			\Big]
			\right)^{q_*/p}
			\le
			C_{R,p,\gamma}.
			\]
			Therefore
			\[
			\mathbb E\Big[
			\|\omega_{j,R}^\sharp\|_{C^\gamma\left([0,T];\mathbb{H}^{-4}(\T^2)\right)}^{q_*}
			\Big]
			\le
			C_{R,p,\gamma}
			\]
			for every
			\[
			0<\gamma<
			\min\left\{
			1-\frac1{q_*},\,
			\frac12-\frac1p
			\right\}.
			\]
			This proves the claim.
		\end{proof}
		
		\begin{lemma}[The continuous interpolant is close to the left interpolant]
			\label{lem:sharp_under_close_conv}
			It holds as $j\to \infty$,
			\[
			\omega_j^\sharp-\underline\omega_j\to0
			\qquad\text{in probability in }\mathbb{L}^2\bigl([0,T];\mathbb{H}^{-4}(\T^2)\bigr).
			\]
		\end{lemma}
		
		\begin{proof}
			Fix \(R>0\).  On the event \(\{\tau_{j,R}=T\}\), the stopped and unstopped
			processes coincide on \([0,T]\). Hence, for every \(t\in[t_j^n,t_j^{n+1}]\),
			\begin{align}\label{new42}
				\omega_j^\sharp(t)-\underline\omega_j(t)
				=
				\int_{t_j^n}^{t}F_{j,R}(s)\,\dd s
				+
				\mathcal M_{j,R}(t)-\mathcal M_{j,R}(t_j^n).
			\end{align}
			Therefore
			\[
			\mathbf 1_{\{\tau_{j,R}=T\}}
			\|\omega_j^\sharp-\underline\omega_j\|_{\mathbb{L}^2\left([0,T];\mathbb{H}^{-4}(\T^2)\right)}
			\le
			A_{j,R}+B_{j,R},
			\]
			where
			\[
			A_{j,R}
			:=
			\Big\|
			t\mapsto
			\int_{t_j^{\lfloor t/\Dt_j\rfloor}}^tF_{j,R}(s)\,\dd s
			\Big\|_{\mathbb{L}^2\left([0,T];\mathbb{H}^{-4}(\T^2)\right)},
			\]
			and
			\[
			B_{j,R}
			:=
			\Big\|
			t\mapsto
			\mathcal M_{j,R}(t)-\mathcal M_{j,R}(t_j^{\lfloor t/\Dt_j\rfloor})
			\Big\|_{\mathbb{L}^2\left([0,T];\mathbb{H}^{-4}(\T^2)\right)}.
			\]
			
			\medskip
			\noindent\textbf{Step 1: estimate of the drift term.}
			By H\"older's inequality,
			\[
			A_{j,R}^2
			\le
			C\Delta t_j^{\,2(1-1/q_*)}
			\|F_{j,R}\|_{\mathbb{L}^{q_*}\left([0,T];\mathbb{H}^{-4}(\T^2)\right)}^2.
			\]
			Equivalently,
			\[
			A_{j,R}^{q_*}
			\le
			C\Delta t_j^{\,q_*-1}
			\|F_{j,R}\|_{\mathbb{L}^{q_*}\left([0,T];\mathbb{H}^{-4}(\T^2)\right)}^{q_*}.
			\]
			Taking expectations and using Lemma~\ref{lem:localized_drift_bound_conv}, we obtain
			\[
			\mathbb E\big[A_{j,R}^{q_*}\big]
			\le
			C\Delta t_j^{\,q_*-1}
			\mathbb E\Big[
			\|F_{j,R}\|_{\mathbb{L}^{q_*}\left([0,T];\mathbb{H}^{-4}(\T^2)\right)}^{q_*}
			\Big]
			\le
			C_R\,\Delta t_j^{\,q_*-1}.
			\]
			Since \(q_*>1\), it follows that, as \(j\to\infty\),
			\[
			\mathbb E\big[A_{j,R}^{q_*}\big]\to0.
			\]
			Hence
			\[
			A_{j,R}\to0
			\qquad\text{in }\mathbb{L}^{q_*}(\Omega),
			\]
			and in particular
			\[
			A_{j,R}\to0
			\qquad\text{in probability}.
			\]
			\medskip
			\noindent\textbf{Step 2: estimate of the martingale term.}
			By It\^o's isometry,
			\[
			\begin{aligned}
				\E\big[B_{j,R}^2\big]
				&=
				\sum_{n=0}^{N_T^j-1}
				\int_{t_j^n}^{t_j^{n+1}}
				\E\Big[
				\big\|
				\mathcal M_{j,R}(t)-\mathcal M_{j,R}(t_j^n)
				\big\|_{\mathbb{H}^{-4}(\T^2)}^2
				\Big]\,\dd t
				\\
				&=
				\sum_{n=0}^{N_T^j-1}
				\int_{t_j^n}^{t_j^{n+1}}
				\int_{t_j^n}^{t}
				\E\Big[
				\sum_{k\in\mathcal I^{K_j}}
				\mathbf 1_{\{s<\tau_{j,R}\}}
				\|\mathcal{P}_j[(\boldsymbol{\sigma}_k\cdot \nabla)\underline\omega_j(s)]\|_{\mathbb{H}^{-4}(\T^2)}^2
				\Big]\,\dd s\,\dd t.
			\end{aligned}
			\]
			By using Lemma~\ref{lem:noise_estimate_conv}, we obtain
			\[
			\begin{aligned}
				\E\big[B_{j,R}^2\big]
				&\le
				C
				\sum_{n=0}^{N_T^j-1}
				\int_{t_j^n}^{t_j^{n+1}}
				\int_{t_j^n}^{t}
				\E\Big[
				\mathbf 1_{\{s<\tau_{j,R}\}}
				\|\underline\omega_j(s)\|_{\mathbb{H}^{-1}(\T^2)}^2
				\Big]\,\dd s\,\dd t.
			\end{aligned}
			\]
			On the event \(\{s<\tau_{j,R}\}\), the stopping-time definition~\eqref{eq:stopping_time_conv} gives
			\[
			\|\underline\omega_j(s)\|_{\mathbb{H}^{-1}(\T^2)}^2\le R.
			\]
			Therefore
			\[
			\begin{aligned}
				\E\big[B_{j,R}^2\big]
				&\le
				CR
				\sum_{n=0}^{N_T^j-1}
				\int_{t_j^n}^{t_j^{n+1}}
				\int_{t_j^n}^{t}\dd s\,\dd t
				\\
				&=
				CR\frac{T\Dt_j}{2}.
			\end{aligned}
			\]
			Thus as $j\to \infty$,
			\[
			B_{j,R}\to0
			\qquad\text{in }\mathbb{L}^2(\Omega),
			\]
			hence in probability.
			
			\medskip
			\noindent\textbf{Step 3: conclusion.}
			Let \(\varepsilon>0\). Then by \eqref{new42}, we obtain
			\[
			\mathbb P\Big[
			\|\omega_j^\sharp-\underline\omega_j\|_{\mathbb{L}^2\left([0,T];\mathbb{H}^{-4}(\T^2)\right)}>\varepsilon
			\Big]
			\le
			\mathbb P\big[\tau_{j,R}<T\big]
			+
			\mathbb P\big[A_{j,R}+B_{j,R}>\varepsilon\big].
			\]
			For fixed \(R\), the second term tends to \(0\) as \(j\to\infty\). Therefore
			\[
			\limsup_{j\to\infty}
			\mathbb P\Big[
			\|\omega_j^\sharp-\underline\omega_j\|_{\mathbb{L}^2\left([0,T];\mathbb{H}^{-4}(\T^2)\right)}>\varepsilon
			\Big]
			\le
			\limsup_{j\to\infty}\mathbb P\big[\tau_{j,R}<T\big].
			\]
			By Lemma~\ref{lem:stop_high_prob_conv},
			\[
			\sup_{j\in\mathbb{N}}\mathbb P\big[\tau_{j,R}<T\big]\le \frac{C}{R},
			\]
			hence
			\[
			\limsup_{j\to\infty}
			\mathbb P\Big[
			\|\omega_j^\sharp-\underline\omega_j\|_{\mathbb{L}^2\left([0,T];\mathbb{H}^{-4}(\T^2)\right)}>\varepsilon
			\Big]
			\le
			\frac{C}{R}.
			\]
			Letting \(R\to\infty\) proves the claim.
		\end{proof}
		\begin{lemma}[The left and right interpolants are close]
			\label{lem:bar_under_close_conv} It holds as $j\to \infty$,
			\[
			\overline\omega_j-\underline\omega_j\to0
			\qquad\text{in probability in }\mathbb{L}^2\bigl([0,T];\mathbb{H}^{-4}(\T^2)\bigr).
			\]
		\end{lemma}
		
		\begin{proof}
			By Lemma~\ref{lem:sharp_under_close_conv}, we already know that as $j\to \infty$,
			\[
			\omega_j^\sharp-\underline\omega_j\to0
			\qquad\text{in probability in }\mathbb{L}^2\bigl([0,T];\mathbb{H}^{-4}(\T^2)\bigr).
			\]
			Thus it is enough to prove that as $j\to \infty$,
			\[
			\omega_j^\sharp-\overline\omega_j\to0
			\qquad\text{in probability in }\mathbb{L}^2\bigl([0,T];\mathbb{H}^{-4}(\T^2)\bigr).
			\]
			
			\medskip
			\noindent\textbf{Step 1: representation on the event of no stopping.}
			Fix \(R>0\). On the event \(\{\tau_{j,R}=T\}\), the stopped and unstopped objects
			coincide on \([0,T]\). Let \(t\in[t_j^n,t_j^{n+1}]\). Since
			\[
			\omega_j^\sharp(t)
			=
			\omega_{N_j}^{n}
			+
			\int_{t_j^n}^t F_{j,R}(s)\,\dd s
			+
			\mathcal M_{j,R}(t)-\mathcal M_{j,R}(t_j^n),
			\]
			and
			\[
			\overline\omega_j(t)=\omega_{N_j}^{n+1}
			=
			\omega_j^\sharp(t_j^{n+1}),
			\]
			then we obtain
			\[
			\omega_j^\sharp(t)-\overline\omega_j(t)
			=
			-\int_t^{t_j^{n+1}} F_{j,R}(s)\,\dd s
			-\bigl(\mathcal M_{j,R}(t_j^{n+1})-\mathcal M_{j,R}(t)\bigr).
			\]
			We now consider
			\[
			\mathbf 1_{\{\tau_{j,R}=T\}}
			\|\omega_j^\sharp-\overline\omega_j\|_{\mathbb{L}^2\left([0,T];\mathbb{H}^{-4}(\T^2)\right)}
			\le
			A_{j,R}^{+}+B_{j,R}^{+},
			\]
			where
			\[
			A_{j,R}^{+}
			:=
			\Big\|
			t\mapsto
			\int_t^{t_j^{\lfloor t/\Dt_j\rfloor+1}}F_{j,R}(s)\,\dd s
			\Big\|_{\mathbb{L}^2\left([0,T];\mathbb{H}^{-4}(\T^2)\right)},
			\]
			and
			\[
			B_{j,R}^{+}
			:=
			\Big\|
			t\mapsto
			\mathcal M_{j,R}(t_j^{\lfloor t/\Dt_j\rfloor+1})-\mathcal M_{j,R}(t)
			\Big\|_{\mathbb{L}^2\left([0,T];\mathbb{H}^{-4}(\T^2)\right)}.
			\]
			
			\medskip
			\noindent\textbf{Step 2: estimate of the drift term.}
			For \(t\in[t_j^n,t_j^{n+1}]\), H\"older's inequality gives
			\[
			\Big\|
			\int_t^{t_j^{n+1}}F_{j,R}(s)\,\dd s
			\Big\|_{\mathbb{H}^{-4}(\T^2)}
			\le
			\Delta t_j^{\,1-1/q_*}
			\|F_{j,R}\|_{\mathbb{L}^{q_*}(t_j^n,t_j^{n+1};\mathbb{H}^{-4}(\T^2))}.
			\]
			Integrating in time over \([0,T]\), we obtain
			\[
			(A_{j,R}^{+})^2
			\le
			C\Delta t_j^{\,2(1-1/q_*)}
			\|F_{j,R}\|_{\mathbb{L}^{q_*}\left([0,T];\mathbb{H}^{-4}(\T^2)\right)}^2.
			\]
			Equivalently,
			\[
			(A_{j,R}^{+})^{q_*}
			\le
			C\Delta t_j^{\,q_*-1}
			\|F_{j,R}\|_{\mathbb{L}^{q_*}\left([0,T];\mathbb{H}^{-4}(\T^2)\right)}^{q_*}.
			\]
			Taking expectations and using Lemma~\ref{lem:localized_drift_bound_conv}, we get
			\[
			\mathbb E\big[(A_{j,R}^{+})^{q_*}\big]
			\le
			C\Delta t_j^{\,q_*-1}
			\mathbb E\Big[
			\|F_{j,R}\|_{\mathbb{L}^{q_*}\left([0,T];\mathbb{H}^{-4}(\T^2)\right)}^{q_*}
			\Big]
			\le
			C_R\,\Delta t_j^{\,q_*-1}.
			\]
			Since \(q_*>1\), it follows that
			\[
			\mathbb E\big[(A_{j,R}^{+})^{q_*}\big]\to 0
			\qquad\text{as }j\to\infty.
			\]
			Hence
			\[
			A_{j,R}^{+}\to0
			\qquad\text{in }\mathbb{L}^{q_*}(\Omega),
			\]
			and therefore also
			\[
			A_{j,R}^{+}\to0
			\qquad\text{in probability}.
			\]
			
			\medskip
			\noindent\textbf{Step 3: estimate of the martingale term.}
			By It\^o's isometry,
			\[
			\begin{aligned}
				\E\big[(B_{j,R}^{+})^2\big]
				&=
				\sum_{n=0}^{N_T^j-1}
				\int_{t_j^n}^{t_j^{n+1}}
				\E\Big[
				\big\|
				\mathcal M_{j,R}(t_j^{n+1})-\mathcal M_{j,R}(t)
				\big\|_{\mathbb{H}^{-4}(\T^2)}^2
				\Big]\,\dd t
				\\
				&=
				\sum_{n=0}^{N_T^j-1}
				\int_{t_j^n}^{t_j^{n+1}}
				\int_t^{t_j^{n+1}}
				\E\Big[
				\sum_{k\in\mathcal I^{K_j}}
				\mathbf 1_{\{s<\tau_{j,R}\}}
				\|\mathcal{P}_j[(\boldsymbol{\sigma}_k\!\cdot \nabla)\underline\omega_j(s)]\|_{\mathbb{H}^{-4}(\T^2)}^2
				\Big]\,\dd s\,\dd t.
			\end{aligned}
			\]
			By using Lemma~\ref{lem:noise_estimate_conv}, we obtain
			\[
			\begin{aligned}
				\E\big[(B_{j,R}^{+})^2\big]
				&\le
				C
				\sum_{n=0}^{N_T^j-1}
				\int_{t_j^n}^{t_j^{n+1}}
				\int_t^{t_j^{n+1}}
				\E\Big[
				\mathbf 1_{\{s<\tau_{j,R}\}}
				\|\underline\omega_j(s)\|_{\mathbb{H}^{-1}(\T^2)}^2
				\Big]\,\dd s\,\dd t.
			\end{aligned}
			\]
			On the event \(\{s<\tau_{j,R}\}\), the definition of the stopping time gives
			\[
			\|\underline\omega_j(s)\|_{\mathbb{H}^{-1}(\T^2)}^2\le R.
			\]
			Therefore
			\[
			\E\big[(B_{j,R}^{+})^2\big]
			\le
			CR
			\sum_{n=0}^{N_T^j-1}
			\int_{t_j^n}^{t_j^{n+1}}
			\int_t^{t_j^{n+1}} \dd s\,\dd t
			\le
			CR\,T\,\Dt_j.
			\]
			Hence as $j\to \infty$,
			\[
			B_{j,R}^{+}\to0
			\qquad\text{in }\mathbb{L}^2(\Omega),
			\]
			and therefore also in probability.
			
			\medskip
			\noindent\textbf{Step 4: remove the stopping.}
			Let \(\varepsilon>0\). From Step 1 we have
			\[
			\mathbf 1_{\{\tau_{j,R}=T\}}
			\|\omega_j^\sharp-\overline\omega_j\|_{\mathbb{L}^2\left([0,T];\mathbb{H}^{-4}(\T^2)\right)}
			\le
			A_{j,R}^{+}+B_{j,R}^{+},
			\]
			and so
			\[
			\mathbb P\Big[
			\|\omega_j^\sharp-\overline\omega_j\|_{\mathbb{L}^2\left([0,T];\mathbb{H}^{-4}(\T^2)\right)}
			>\varepsilon
			\Big]
			\le
			\mathbb P\bigl[\tau_{j,R}<T\bigr]
			+
			\mathbb P\bigl[ A_{j,R}^{+}+B_{j,R}^{+}>\varepsilon\bigr].
			\]
			For fixed \(R\), Step 2 and Step 3 imply as $j\to \infty$,
			\[
			A_{j,R}^{+}+B_{j,R}^{+}\to0
			\qquad\text{in probability},
			\]
			hence
			\[
			\limsup_{j\to\infty}
			\mathbb P\Big[
			\|\omega_j^\sharp-\overline\omega_j\|_{\mathbb{L}^2\left([0,T];\mathbb{H}^{-4}(\T^2)\right)}
			>\varepsilon
			\Big]
			\le
			\sup_{j\in\mathbb{N}}\mathbb P\bigl[\tau_{j,R}<T\bigr].
			\]
			By letting \(R\to\infty\) and using Lemma~\ref{lem:stop_high_prob_conv},, we conclude that as $j\to \infty$,
			\[
			\omega_j^\sharp-\overline\omega_j\to0
			\qquad\text{in probability in }\mathbb{L}^2\bigl([0,T];\mathbb{H}^{-4}(\T^2)\bigr).
			\]
			
			\medskip
			\noindent\textbf{Step 5: conclude.}
			Finally, by the triangle inequality,
			\[
			\|\overline\omega_j-\underline\omega_j\|_{\mathbb{L}^2\left([0,T];\mathbb{H}^{-4}(\T^2)\right)}
			\le
			\|\overline\omega_j-\omega_j^\sharp\|_{\mathbb{L}^2\left([0,T];\mathbb{H}^{-4}(\T^2)\right)}
			+
			\|\omega_j^\sharp-\underline\omega_j\|_{\mathbb{L}^2\left([0,T];\mathbb{H}^{-4}(\T^2)\right)}.
			\]
			Both terms on the right converge to \(0\) in probability. Therefore as $j\to \infty$,
			\[
			\overline\omega_j-\underline\omega_j\to0
			\qquad\text{in probability in }\mathbb{L}^2\bigl([0,T];\mathbb{H}^{-4}(\T^2)\bigr),
			\]
			as claimed.
		\end{proof}
		\subsection{Canonical path space and tightness}
		\label{subsec:tightness_conv}
		
		We now introduce the path space used for the compactness argument. For the continuous interpolant we set
		\[
		\mathcal X_\omega
		:=
		C\bigl([0,T];\mathbb{H}^{-5}(\T^2)\bigr).
		\]
		For the left and right interpolants we use
		\[
		\mathcal Y_\omega
		:=
		\mathbb{L}^2\bigl([0,T];\mathbb{H}^{-5}(\T^2)\bigr)
		\cap
		\mathbb{L}^2\bigl([0,T];\mathbb{H}^{-\alpha}(\T^2)\bigr)_{\mathrm{weak}}.
		\]
		We also define
		\[
		\mathcal X_{\bf W}:=C\bigl([0,T];\mathbb R^2\bigr)^{\mathcal I}
		\]
		with the product topology. 
		The full path space is
		\begin{equation}\label{eq:full_path_space_conv}
			\mathfrak X
			:=
			\mathcal X_\omega
			\times
			\mathcal Y_\omega
			\times
			\mathcal Y_\omega
			\times
			\mathcal X_{\bf W}.
		\end{equation}
		Its five coordinates correspond respectively to
		\[
		\omega_j^\sharp,\qquad
		\underline\omega_j,\qquad
		\overline\omega_j,\qquad
		({\bf W}_k)_{k\in\mathcal I}.
		\]
		
		\begin{proposition}[Tightness]
			\label{prop:tightness_main_conv}
			The laws of
			\[
			\Big(
			\omega_j^\sharp,\underline\omega_j,\overline\omega_j,
			({\bf W}_k)_{k\in\mathcal I}
			\Big)
			\]
			are tight on \(\mathfrak X\).
		\end{proposition}
		
		\begin{proof}
			We prove tightness coordinate by coordinate.
			
			\medskip
			\noindent\textbf{Step 1: the continuous interpolant \(\omega_j^\sharp\).}
			We first prove tightness of \((\omega_j^\sharp)_{j\in\mathbb{N}}\) in
			\(C\bigl([0,T];\mathbb{H}^{-5}(\T^2)\bigr)\). Fix \(\varepsilon>0\). By
			Lemma~\ref{lem:localized_holder_conv}, for every \(R>0\) and every \(p>2\), there
			exist \(\gamma>0\) and a constant \(C_{R,p}\), independent of \(j\), such that
			\[
			\mathbb E\Big[
			\|\omega_{j,R}^\sharp\|_{C^\gamma\left([0,T];\mathbb{H}^{-4}(\T^2)\right)}^p
			\Big]
			\le C_{R,p}.
			\]
			Hence, by Markov's inequality, for every \(\eta>0\) one can choose \(M=M(\eta,R)\)
			so large that
			\begin{align}\label{new50}
				\sup_{j\in\mathbb{N}}
				\mathbb P\Big[
				\|\omega_{j,R}^\sharp\|_{C^\gamma\left([0,T];\mathbb{H}^{-4}(\T^2)\right)}>M
				\Big]
				\le \eta.
			\end{align}
			Now the embedding
			\[
			C^\gamma\bigl([0,T];\mathbb{H}^{-4}(\T^2)\bigr)\hookrightarrow C\bigl([0,T];\mathbb{H}^{-5}(\T^2)\bigr)
			\]
			is compact, because the embedding \(\mathbb{H}^{-4}(\T^2)\hookrightarrow
			\mathbb{H}^{-5}(\T^2)\) is  compact, and \(\gamma>0\). Therefore the set
			\[
			K_{R,M}
			:=
			\Big\{
			v\in C^\gamma\bigl([0,T];\mathbb{H}^{-4}(\T^2)\bigr) : \|v\|_{C^\gamma\left([0,T];\mathbb{H}^{-4}(\T^2)\right)}\le M
			\Big\}
			\]
			is relatively compact in \(C\bigl([0,T];\mathbb{H}^{-5}(\T^2)\bigr)\). Since
			\(\omega_j^\sharp=\omega_{j,R}^\sharp\) on the event \(\{\tau_{j,R}=T\}\), we obtain
			\[
			\mathbb P\bigl[\omega_j^\sharp\notin K_{R,M}\bigr]
			\le
			\mathbb P\bigl[\tau_{j,R}<T\bigr]
			+
			\mathbb P\Big[
			\|\omega_{j,R}^\sharp\|_{C^\gamma\left([0,T];\mathbb{H}^{-4}(\T^2)\right)}>M
			\Big].
			\]
			By Lemma~\ref{lem:stop_high_prob_conv}, the first term is uniformly small for \(R\)
			large, and then the second term is uniformly small for \(M\) large by \eqref{new50}. This proves
			tightness of \((\omega_j^\sharp)_{j\in\mathbb{N}}\) in \(C\bigl([0,T];\mathbb{H}^{-5}(\T^2)\bigr)\). Therefore, we conclude that the laws of \((\omega_j^\sharp)_{j\in\mathbb{N}}\)
			are tight in \(\mathcal X_\omega\).
			
			\medskip
			\noindent\textbf{Step 2: the left interpolant \(\underline\omega_j\).}
			The same uniform estimate \eqref{eq:interpolant_Halpha_bounds_conv} implies
			tightness of \((\underline\omega_j)_{j\in\mathbb{N}}\) in
			\[
			\mathbb{L}^2\bigl([0,T];\mathbb{H}^{-\alpha}(\T^2)\bigr)_{\mathrm{weak}}.
			\]
			To obtain tightness in the strong space \(\mathbb{L}^2\bigl([0,T];\mathbb{H}^{-5}(\T^2)\bigr)\), we use
			its continuous interpolant $\omega_j^\sharp$. Since the embedding
			\[
			C\bigl([0,T];\mathbb{H}^{-5}(\T^2)\bigr)\hookrightarrow \mathbb{L}^2\bigl([0,T];\mathbb{H}^{-5}(\T^2)\bigr)
			\]
			is continuous, Step~1 implies that \((\omega_j^\sharp)_{j\in\mathbb{N}}\) is tight in
			\(\mathbb{L}^2\bigl([0,T];\mathbb{H}^{-5}(\T^2)\bigr)\).
			By Lemma~\ref{lem:sharp_under_close_conv}, as $j\to \infty$,
			\[
			\omega_j^\sharp-\underline\omega_j\to0
			\qquad\text{in probability in }\mathbb{L}^2\bigl([0,T];\mathbb{H}^{-4}(\T^2)\bigr).
			\]
			Since the embedding
			\[
			\mathbb{H}^{-4}(\T^2)\hookrightarrow \mathbb{H}^{-5}(\T^2)
			\]
			is continuous, we also have as $j\to \infty$,
			\[
			\omega_j^\sharp-\underline\omega_j\to0
			\qquad\text{in probability in }\mathbb{L}^2\bigl([0,T];\mathbb{H}^{-5}(\T^2)\bigr).
			\]
			Because \(\mathbb{L}^2\bigl([0,T];\mathbb{H}^{-5}(\T^2)\bigr)\) is Polish, tightness is stable under
			perturbations converging to zero in probability. Therefore the laws of
			\((\underline\omega_j)_{j\in\mathbb{N}}\) are tight in \(\mathbb{L}^2\bigl([0,T];\mathbb{H}^{-5}(\T^2)\bigr)\), and
			hence in \(\mathcal Y_\omega\).
			
			\medskip
			\noindent\textbf{Step 3: the right interpolant \(\overline\omega_j\).}
			Again \eqref{eq:interpolant_Halpha_bounds_conv} implies tightness of \(\overline\omega_j\) in
			\[
			\mathbb{L}^2\bigl([0,T];\mathbb{H}^{-\alpha}(\T^2)\bigr)_{\mathrm{weak}}.
			\]
			Moreover, Lemma~\ref{lem:bar_under_close_conv} gives as $j\to \infty$,
			\[
			\overline\omega_j-\underline\omega_j\to0
			\qquad\text{in probability in }\mathbb{L}^2\bigl([0,T];\mathbb{H}^{-4}(\T^2)\bigr),
			\]
			and therefore also in \(\mathbb{L}^2\bigl([0,T];\mathbb{H}^{-5}(\T^2)\bigr)\). Since
			\((\underline\omega_j)_{j\in\mathbb{N}}\) is already tight in \(\mathbb{L}^2\bigl([0,T];\mathbb{H}^{-5}(\T^2)\bigr)\),
			the same stability argument shows that \((\overline\omega_j)_{j\in\mathbb{N}}\) is tight in
			\(\mathbb{L}^2\bigl([0,T];\mathbb{H}^{-5}(\T^2)\bigr)\). Thus the laws of \((\overline\omega_j)_{j\in\mathbb{N}}\) are
			tight in \(\mathcal Y_\omega\).
			
			\medskip
			\noindent\textbf{Step 4: the Brownian coordinate.}
			The law of \(({\bf W}_k)_{k\in\mathcal I}\) on \(\mathcal X_{\bf W}\) is the product Wiener
			measure, hence tight.
			
			\medskip
			Combining the four coordinates proves tightness on \(\mathfrak X\).
		\end{proof}

		\subsection{Skorokhod--Jakubowski representation}
		\label{subsec:skorokhod_conv}
		
		We now pass to a new probability space on which the approximating variables
		converge almost surely but now is on new probability space.
		
		\begin{proposition}[Skorokhod--Jakubowski representation]
			\label{prop:skorokhod_jakubowski_conv}
			After extraction of a subsequence (not relabeled), there exist a probability space
			\[
			(\widetilde\Omega,\widetilde{\mathcal F},\widetilde{\mathbb P})
			\]
			and random variables
			\[
			\widetilde X_j
			:=
			\Big(
			\widetilde\omega_j^\sharp,\widetilde{\underline\omega}_j,
			\widetilde{\overline\omega}_j,\bigl( \widetilde{\mathbf W}^j_k \bigr)_{k\in\mathcal{I}}
			\Big),
			\qquad
			\widetilde X
			:=
			\Big(
			\widetilde\omega^\sharp,\widetilde{\underline\omega},
			\widetilde{\overline\omega},\bigl(\widetilde{\mathbf W}_k\bigl)_{k\in\mathcal{I}}
			\Big),
			\]
			with values in \(\mathfrak X\), such that:
			
			\begin{enumerate}[label=\textup{(\roman*)}]
				\item for every \(j\in \mathbb{N}\), the law of \(\widetilde X_j\) coincides with the law of
				\[
				\Big(
				\omega_j^\sharp,\underline\omega_j,\overline\omega_j,
				({\bf W}_k^j)_{k\in\mathcal I}
				\Big);
				\]
				
				\item as $j\to \infty$,
				\[
				\widetilde\omega_j^\sharp\to\widetilde\omega^\sharp
				\quad\widetilde{\mathbb{P}}-\text{a.s. in }C\bigl([0,T];\mathbb{H}^{-5}(\T^2)\bigr),
				\]
				\[\widetilde{\underline\omega}_j\rightharpoonup\widetilde{\underline\omega},
				\qquad
				\widetilde{\overline\omega}_j\rightharpoonup\widetilde{\overline\omega}
				\quad\widetilde{\mathbb{P}}-\text{a.s. in }\mathbb{L}^2\bigl([0,T];\mathbb{H}^{-\alpha}(\T^2)\bigr),
				\]
				and also
				\[(\widetilde{\bf W}_k^j)_{k\in\mathcal I}\to (\widetilde{\bf W}_k)_{k\in\mathcal I}\quad\widetilde{\mathbb{P}}-\text{a.s. in } C\bigl([0,T];\mathbb R^2\bigr)^{\mathcal I};\]
				\item  there exists a process \(\widetilde\omega\) such that
				\[
				\widetilde\omega^\sharp=\widetilde{\underline\omega}
				=\widetilde{\overline\omega}
				=:\widetilde\omega
				\]
				\(\widetilde{\mathbb P}\)-a.s. as elements of
				\(\mathbb{L}^2\bigl([0,T];\mathbb{H}^{-5}(\T^2)\bigr)\), and \(\widetilde\omega\)
				has a continuous \(\mathbb{H}^{-5}\)-valued version, still denoted by
				\(\widetilde\omega\);
				
				\item if \((\widetilde{\mathcal F}_t)_{t\in[0,T]}\) denotes the usual augmentation of
				the filtration generated by
				\[
				\bigg\{\bigl(\widetilde{\omega}(s),\widetilde{\bf W}_k({s})\bigr):
				0\le s\le t,\ k\in\mathcal I\bigg\},
				\]
				then \(\widetilde\omega\) is \((\widetilde{\mathcal F}_t)\)-adapted and
				\((\widetilde{\bf W}_k)_{k\in\mathcal I}\) is a family of independent
				\((\widetilde{\mathcal F}_t)\)-Brownian motions.
			\end{enumerate}
		\end{proposition}
		
		\begin{proof}
			By Proposition~\ref{prop:tightness_main_conv}, the laws of
			\[
			\bigl(
			\omega_j^\sharp,\underline{\omega}_j,\overline{\omega}_j,({\bf W}_k)_{k\in\mathcal I}
			\bigr)
			\]
			are tight on \(\mathfrak X\). Since \(\mathfrak X\) is a topological space with a
			countable family of continuous functions that separate points, we may apply Skorokhod--Jakubowski representation theorem 
			\cite{Jakubowski1997}. Hence there exist a probability space
			\[
			(\widetilde\Omega,\widetilde{\mathcal F},\widetilde{\mathbb P})
			\]
			and \(\mathfrak X\)-valued random variables \(\widetilde X_j\) and \(\widetilde X\)
			satisfying \textup{(i)}--\textup{(ii)}.

			We next identify the three limiting vorticity coordinates. Since the law of
			\(\widetilde\omega_j^\sharp-\widetilde{\underline\omega}_j\) coincides with the law
			of \(\omega_j^\sharp-\underline\omega_j\), Lemma~\ref{lem:sharp_under_close_conv}
			and the continuous embedding
			\[
			\mathbb{H}^{-4}(\T^2)\hookrightarrow \mathbb{H}^{-5}(\T^2)
			\]
			imply as $j\to \infty$,
			\[
			\widetilde\omega_j^\sharp-\widetilde{\underline\omega}_j\to0
			\qquad\text{in probability in }\mathbb{L}^2\bigl([0,T];\mathbb{H}^{-5}(\T^2)\bigr).
			\]
			On the other hand, both sequences converge almost surely in
			\(\mathbb{L}^2\bigl([0,T];\mathbb{H}^{-5}(\T^2)\bigr)\), hence
			\[
			\widetilde\omega_j^\sharp-\widetilde{\underline\omega}_j
			\to
			\widetilde\omega^\sharp-\widetilde{\underline\omega}
			\qquad\widetilde{\mathbb{P}}-\text{a.s. in }\mathbb{L}^2\bigl([0,T];\mathbb{H}^{-5}(\T^2)\bigr).
			\]
			Therefore
			\[
			\widetilde\omega^\sharp=\widetilde{\underline\omega}
			\qquad
			\widetilde{\mathbb P}-\text{a.s. in }\mathbb{L}^2\bigl([0,T];\mathbb{H}^{-5}(\T^2)\bigr).
			\]
			Similarly, by transport of the law and Lemma~\ref{lem:bar_under_close_conv}, as $j\to \infty$,
			\[
			\widetilde{\overline\omega}_j-\widetilde{\underline\omega}_j\to0
			\qquad\text{in probability in }\mathbb{L}^2\bigl([0,T];\mathbb{H}^{-4}(\T^2)\bigr),
			\]
			hence also in \(\mathbb{L}^2\bigl([0,T];\mathbb{H}^{-5}(\T^2)\bigr)\). Since as $j\to \infty$,
			\[
			\widetilde{\overline\omega}_j-\widetilde{\underline\omega}_j
			\to
			\widetilde{\overline\omega}-\widetilde{\underline\omega}
			\qquad\widetilde{\mathbb{P}}-\text{a.s. in }\mathbb{L}^2\bigl([0,T];\mathbb{H}^{-5}(\T^2)\bigr),
			\]
			we conclude that
			\[
			\widetilde{\overline\omega}=\widetilde{\underline\omega}
			\qquad
			\widetilde{\mathbb P}-\text{a.s. in }\mathbb{L}^2\bigl([0,T];\mathbb{H}^{-5}(\T^2)\bigr).
			\]
			We denote the common limit by
			\[
			\widetilde\omega:=\widetilde\omega^\sharp
			=\widetilde{\underline\omega}
			=\widetilde{\overline\omega}.
			\]
			Since \(\widetilde\omega^\sharp\in C\bigl([0,T];\mathbb{H}^{-5}(\T^2)\bigr)\), this
			gives a continuous \(\mathbb{H}^{-5}(\T^2)\)-valued version of \(\widetilde\omega\). This
			proves \textup{(iii)}.
			
			It remains to prove that \((\widetilde{\bf W}_k)_{k\in\mathcal I}\) is a family of
			independent Brownian motions with respect to {the usual augmentation filtration \((\widetilde{\mathcal F}_t)_{t\in[0,T]}\) of the filtration generated by \(\big(\widetilde{\omega}, (\widetilde{\bf W}_k)_{k\in\mathcal I}\big)\)}. Fix \(0\le s<t\le T\), \(k,\ell\in\mathcal I\), and let \(\Psi\) be a
			bounded continuous functional of the path variables up to time \(s\). Since the law
			of \(\widetilde X_j\) coincides with the law of the original tuple, we have
			\[
			\widetilde{\mathbb E}\Big[
			(\widetilde{\bf W}_k^j(t)-\widetilde{\bf W}_k^j(s))\,\Psi(\widetilde X_j|_{[0,s]})
			\Big]=0,
			\]
			\[
			\widetilde{\mathbb E}\Big[
			\big((\widetilde{\bf W}_k^j(t)-\widetilde{\bf W}_k^j(s))^2-(t-s)\big)\,
			\Psi(\widetilde X_j|_{[0,s]})
			\Big]=0,
			\]
			and, if \(k\neq \ell\),
			\[
			\widetilde{\mathbb E}\Big[
			(\widetilde{\bf W}_k^j(t)-\widetilde{\bf W}_k^j(s))
			(\widetilde {\bf W}_\ell^j(t)-\widetilde {\bf W}_\ell^j(s))\,
			\Psi(\widetilde X_j|_{[0,s]})
			\Big]=0.
			\]
			Since \(\widetilde X_j\to\widetilde X\) almost surely in \(\mathfrak X\), dominated
			convergence yields the same identities with \(\widetilde X_j\) replaced by
			\(\widetilde X\). By a monotone-class argument, these identities extend to all
			bounded \(\widetilde{\mathcal F}_s\)-measurable random variables. Therefore each
			\(\widetilde{\bf W}_k\) is a continuous \((\widetilde{\mathcal F}_t)\)-martingale with
			quadratic variation \(t\), and the cross-variations vanish for \(k\neq \ell\). By
			L\'evy's characterization theorem~\cite[Theorem~3.16]{KaratzasShreve1991}, \((\widetilde{\bf W}_k)_{k\in\mathcal I}\) is a
			family of independent \((\widetilde{\mathcal F}_t)\)-Brownian motions. This proves
			\textup{(iv)}.
		\end{proof}
		
		\subsection{Passage to the limit}
		\label{subsec:driftimit_conv}
		
		Fix \(\varphi\in C^\infty(\T^2)\). Since the law of
		\(\bigl(\widetilde\omega_j^\sharp,\widetilde{\underline\omega}_j,
		\widetilde{\overline\omega}_j,(\widetilde{\mathbf W}_k^j)_{k\in\mathcal{I}}\bigl)\) coincides with the law of
		\(\bigl(\omega_j^\sharp,\underline\omega_j,\overline\omega_j,(\mathbf W_k^j)_{k\in\mathcal{I}}\bigr)\), the discrete
		weak formulation reads \(\widetilde{\mathbb P}\)-a.s.
		\begin{equation}\label{eq:discrete_weak_form_conv}
			\begin{aligned}
				\langle \widetilde\omega_j^\sharp(t),\varphi\rangle
				&=
				\langle \mathcal{P}_j\omega_0,\varphi\rangle
				-
				\int_0^t
				\big\langle
				(\widetilde{\overline {\bf u}}_j(s)\cdot \nabla)\widetilde{\underline\omega}_j(s),
				\mathcal{P}_j\varphi
				\big\rangle\,\dd s
				\\
				&\quad
				+
				\frac{c_{K_j}}{2}
				\int_0^t
				\langle \widetilde{\overline\omega}_j(s),\Delta(\mathcal{P}_j\varphi)\rangle\,\dd s
				+
				\sum_{k\in\mathcal I}
				\int_0^t g_{j,k}^{\varphi}(s)\,\dd\widetilde{\bf W}^j_k(s),
			\end{aligned}
		\end{equation}
		where
		\[
		\widetilde{\overline {\bf u}}_j
		:=
		\nabla^\perp(-\Delta)^{-1}\widetilde{\overline\omega}_j,
		\]
		and
		\[
		g_{j,k}^{\varphi}(s)
		:=
		\mathbf 1_{\{k\in\mathcal I^{K_j}\}}
		\langle \widetilde{\underline\omega}_j(s)\boldsymbol{\sigma}_k,\nabla(\mathcal{P}_j\varphi)\rangle .
		\]
		We now pass to the limit term by term.
		
		\noindent
		\textbf{Initial condition:} Since as $j\to \infty$, \(\mathcal{P}_j\omega_0\to\omega_0\) in \(\mathbb{H}^{-1}(\T^2)\), then as $j\to \infty$,
		\[
		\langle \mathcal{P}_j\omega_0,\varphi\rangle\to\langle \omega_0,\varphi\rangle .
		\]
		
		\noindent
		\textbf{Strong convergence of the interpolants:} From Proposition~\ref{prop:skorokhod_jakubowski_conv} and the previous closeness
		lemmas, we already know that as $j\to \infty$,
		\[
		\widetilde{\underline\omega}_j\to\widetilde\omega,
		\qquad
		\widetilde{\overline\omega}_j\to\widetilde\omega
		\qquad\widetilde{\mathbb{P}}-\text{a.s. in }\mathbb{L}^2\bigl([0,T];\mathbb{H}^{-5}(\T^2)\bigr).
		\]
		Moreover, both sequences are a.s. bounded in
		\(\mathbb{L}^2\bigl([0,T];\mathbb{H}^{-\alpha}(\T^2)\bigr)\). Therefore, for every
		\(\gamma\) satisfying
		\[
		\alpha<\gamma_1<5,
		\]
		interpolation yields as $j\to \infty$,
		\begin{equation}\label{eq:strong_Hminusgamma_conv}
			\widetilde{\underline\omega}_j\to\widetilde\omega,
			\qquad
			\widetilde{\overline\omega}_j\to\widetilde\omega
			\qquad\text{a.s. in }\mathbb{L}^2([0,T];\mathbb{H}^{-\gamma_1}(\T^2)).
		\end{equation}
		
		\noindent
		\textbf{Nonlinear term:} Fix \(\beta\) such that
		\begin{equation}\label{eq:beta_choice_conv}
			\alpha<\beta<\frac12.
		\end{equation}
		By \eqref{eq:strong_Hminusgamma_conv}, taking \(\gamma_1=\beta\), we have as $j\to \infty$,
		\[
		\widetilde{\underline\omega}_j\to\widetilde\omega,
		\qquad
		\widetilde{\overline\omega}_j\to\widetilde\omega
		\qquad\widetilde{\mathbb{P}}-\text{a.s. in }\mathbb{L}^2([0,T];\mathbb{H}^{-\beta}(\T^2)).
		\]
		Define
		\[
		\widetilde{\overline {\bf u}}_j
		:=
		\nabla^\perp(-\Delta)^{-1}\widetilde{\overline\omega}_j,
		\qquad
		\widetilde{\bf u}
		:=
		\nabla^\perp(-\Delta)^{-1}\widetilde\omega.
		\]
		Since the Biot--Savart operator is continuous from \(\mathbb{H}^{-\beta}\) to
		\(\mathbb{H}^{1-\beta}\), it follows that as $j\to \infty$,
		\begin{align}\label{strongconvergene}
			\widetilde{\overline {\bf u}}_j\to\widetilde{\bf u}
			\qquad\text{a.s. in }\mathbb{L}^2([0,T];\mathbb{H}^{1-\beta}(\T^2)).
		\end{align}
		Because \(\beta<1/2\), we have
		\[
		\mathbb{H}^{1-\beta}(\T^2)\hookrightarrow \mathbb{H}^\beta(\T^2).
		\]
		Since \(\mathcal{P}_j\varphi\to\varphi\) in \(C^\infty(\T^2)\), we obtain as $j\to \infty$,
		\[
		\widetilde{\overline {\bf u}}_j\cdot \nabla(\mathcal{P}_j\varphi)
		\to
		\widetilde{\bf u}\cdot \nabla\varphi
		\qquad\widetilde{\mathbb{P}}-\text{a.s. in }\mathbb{L}^2([0,T];\mathbb{H}^\beta(\T^2)).
		\]
		By using \(\operatorname{div}\widetilde{\overline {\bf u}}_j=0\), we write
		\[
		-\big\langle
		(\widetilde{\overline {\bf u}}_j\cdot \nabla)\widetilde{\underline\omega}_j,
		\mathcal{P}_j\varphi
		\big\rangle
		=
		\big\langle
		\widetilde{\underline\omega}_j,\,
		\widetilde{\overline {\bf u}}_j\cdot \nabla(\mathcal{P}_j\varphi)
		\big\rangle .
		\]
		By the duality \(\mathbb{H}^{-\beta}\times\mathbb{H}^\beta\), we conclude that as $j\to \infty$,
		\[
		\big\langle
		\widetilde{\underline\omega}_j,\,
		\widetilde{\overline {\bf u}}_j\cdot \nabla(\mathcal{P}_j\varphi)
		\big\rangle
		\to
		\big\langle
		\widetilde\omega,\widetilde{\bf u}\cdot \nabla\varphi
		\big\rangle
		\qquad\widetilde{\mathbb{P}}-\text{a.s. in }\mathbb{L}^1(0,T).
		\]
		Therefore, for every \(t\in[0,T]\), as $j\to \infty$,
		\begin{equation}\label{eq:nonlinearimit_conv}
			-\int_0^t
			\big\langle
			(\widetilde{\overline {\bf u}}_j(s)\cdot \nabla)\widetilde{\underline\omega}_j(s),
			\mathcal{P}_j\varphi
			\big\rangle\,\dd s
			\to
			\int_0^t
			\big\langle
			\widetilde\omega(s),\widetilde{\bf u}(s)\cdot \nabla\varphi
			\big\rangle\,\dd s
			\qquad\widetilde{\mathbb{P}}-\text{a.s.}
		\end{equation}
		
		\noindent
		\textbf{It\^o correction:}
		Since as $j\to \infty$, \(c_{K_j}\to c_\infty\), \(\Delta(\mathcal{P}_j\varphi)\to\Delta\varphi\) in
		\(\mathbb{H}^1(\T^2)\), and
		\(\widetilde{\overline\omega}_j\to\widetilde\omega\) strongly in
		\(\mathbb{L}^2([0,T];\mathbb{H}^{-1}(\T^2))\) by \eqref{eq:strong_Hminusgamma_conv}, we obtain as $j\to \infty$,
		\begin{equation}\label{eq:itoimit_conv}
			\frac{c_{K_j}}{2}
			\int_0^t
			\langle \widetilde{\overline\omega}_j(s),\Delta(\mathcal{P}_j\varphi)\rangle\,\dd s
			\to
			\frac{c_\infty}{2}
			\int_0^t
			\langle \widetilde\omega(s),\Delta\varphi\rangle\,\dd s .
		\end{equation}
		
		\noindent
		\textbf{Stochastic term:}
		For \(k\in\mathcal I\), define
		\[
		g_{j,k}^{\varphi}(s)
		:=
		\mathbf 1_{\{k\in\mathcal I^{K_j}\}}
		\langle \widetilde{\underline\omega}_j(s)\boldsymbol{\sigma}_k,\nabla(\mathcal{P}_j\varphi)\rangle,
		\qquad
		g_k^\varphi(s)
		:=
		\langle \widetilde\omega(s)\boldsymbol{\sigma}_k,\nabla\varphi\rangle.
		\]
		We
		may write $\widetilde{\mathbb{P}}$-a.s., for all $t\in[0,T]$,
		\[
		I_j^\varphi(t)
		:=
		\sum_{k\in\mathcal I}
		\int_0^t g_{j,k}^{\varphi}(s)\,\dd\widetilde{\bf W}^j_k(s),
		\qquad
		I^\varphi(t)
		:=
		\sum_{k\in\mathcal I}
		\int_0^t g_k^\varphi(s)\,\dd\widetilde{\bf W}_k(s).
		\]
		
		\begin{lemma}[Convergence of the stochastic coefficients]
			\label{lem:coeff_convergence_conv}
			For every \(\varphi\in C^\infty(\T^2)\), as $j\to \infty$,
			\begin{equation}\label{eq:coeff_convergence_conv}
				\widetilde{\mathbb E}\Bigg[
				\int_0^T
				\sum_{k\in\mathcal I}
				|g_{j,k}^{\varphi}(s)-g_k^\varphi(s)|^2\,\dd s
				\Bigg]
				\longrightarrow0 .
			\end{equation}
		\end{lemma}
		
		\begin{proof}
			We write
			\[
			g_{j,k}^{\varphi}-g_k^\varphi
			=
			A_{j,k}+B_{j,k}+C_{j,k},
			\]
			where
			\[
			A_{j,k}
			:=
			\mathbf 1_{\{k\in\mathcal I^{K_j}\}}
			\langle
			(\widetilde{\underline\omega}_j-\widetilde\omega)\boldsymbol{\sigma}_k,\nabla(\mathcal{P}_j\varphi)
			\rangle,
			\]
			\[
			B_{j,k}
			:=
			\mathbf 1_{\{k\in\mathcal I^{K_j}\}}
			\langle
			\widetilde\omega\,\boldsymbol{\sigma}_k,\nabla(\mathcal{P}_j\varphi-\varphi)
			\rangle,
			\]
			\[
			C_{j,k}
			:=
			-\mathbf 1_{\{k\notin\mathcal I^{K_j}\}}
			\langle
			\widetilde\omega\,\boldsymbol{\sigma}_k,\nabla\varphi
			\rangle.
			\]
			
			\medskip
			\noindent\textbf{Term \(A_{j,k}:\)}
			By \eqref{eq:noise_scalar_estimate_conv}, applied with
			\(f=\widetilde{\underline\omega}_j(s)-\widetilde\omega(s)\) and the test function
			\(\mathcal{P}_j\varphi\), we have
			\[
			\sum_{k\in\mathcal I}|A_{j,k}(s)|^2
			\le
			\|\mathcal{P}_j\varphi\|_{\mathbb{H}^4(\T^2)}^2
			\|\widetilde{\underline\omega}_j(s)-\widetilde\omega(s)\|_{\mathbb{H}^{-1}(\T^2)}^2.
			\]
			Therefore
			\[
			\widetilde{\mathbb E}\Bigg[
			\int_0^T\sum_{k\in\mathcal I}|A_{j,k}(s)|^2\,\dd s
			\Bigg]
			\le
			\|\varphi\|_{\mathbb{H}^4(\T^2)}^2\,
			\widetilde{\mathbb E}\Big[
			\|\widetilde{\underline\omega}_j-\widetilde\omega\|_{\mathbb{L}^2([0,T];\mathbb{H}^{-1}(\T^2))}^2
			\Big].
			\]
			By \eqref{eq:strong_Hminusgamma_conv}, the right-hand side tends to \(0\) as $j\to \infty$.
			
			\medskip
			\noindent\textbf{Term \(B_{j,k}:\)}
			Again by \eqref{eq:noise_scalar_estimate_conv}, now with
			\(f=\widetilde\omega(s)\) and test function \(\mathcal{P}_j\varphi-\varphi\),
			\[
			\sum_{k\in\mathcal I}|B_{j,k}(s)|^2
			\le
			\|{\mathcal{P}_j\varphi-\varphi}\|_{\mathbb{H}^4(\mathbb{T}^2)}^2
			\|\widetilde\omega(s)\|_{\mathbb{H}^{-1}(\T^2)}^2.
			\]
			Since $\|{\mathcal{P}_j\varphi-\varphi}\|_{\mathbb{H}^4(\mathbb{T}^2)}^2\to 0$ as $j\to \infty$,
			therefore as $j\to \infty$,
			\[
			\widetilde{\mathbb E}\Bigg[
			\int_0^T\sum_{k\in\mathcal I}|B_{j,k}(s)|^2\,\dd s
			\Bigg]
			\longrightarrow0.
			\]
			
			\medskip
			\noindent\textbf{Term \(C_{j,k}:\)} 
			For a.e. \(s\in[0,T]\), \eqref{eq:noise_scalar_estimate_conv} gives \(\widetilde{\mathbb P}\)-a.s.,
			\[
			\sum_{k\in\mathcal I}
			|\langle \widetilde\omega(s)\boldsymbol{\sigma}_k,\nabla\varphi\rangle|^2
			\le
			\|\varphi\|^2_{\mathbb{H}^4(\T^2)}\|\widetilde\omega(s)\|_{\mathbb{H}^{-1}(\T^2)}^2.
			\]
			Hence
			\[
			0\le
			\int_0^T\sum_{k\in\mathcal I}|C_{j,k}(s)|^2\,\dd s
			\le
			\|\varphi\|^2_{\mathbb{H}^4(\T^2)}\|\widetilde\omega\|_{\mathbb{L}^2([0,T];\mathbb{H}^{-1}(\T^2))}^2,
			\]
			and the right-hand side is integrable. Since \(K_j\to\infty\), for a.e. \(s\in [0,T]\) the
			tail over \(k\notin\mathcal I^{K_j}\) converges to \(0\). Therefore,
			by dominated convergence, as $j\to \infty$,
			\[
			\widetilde{\mathbb E}\Bigg[
			\int_0^T\sum_{k\in\mathcal I}|C_{j,k}(s)|^2\,\dd s
			\Bigg]
			\longrightarrow0.
			\]
			Combining the three estimates proves \eqref{eq:coeff_convergence_conv}.
		\end{proof}
		
		\noindent
		
		{The equality of laws in \textup{(i)} of Proposition~\ref{prop:skorokhod_jakubowski_conv} transfers the fully discrete scheme \eqref{eq:scheme_keep} to the new probability space; see \eqref{eq:discrete_weak_form_conv}. Hence, for
			each \(j\in\mathbb N\), the process \(\widetilde{\underline\omega}_j\) satisfies
			the fully discrete scheme \eqref{eq:scheme_keep} with respect to the Brownian
			family \((\widetilde {\bf W}_k^j)_{k\in\mathcal I}\) and is adapted to the usual filtration
			\[
			\widetilde{\mathcal F}_t^j
			:=
			\sigma\bigl\{
			\widetilde {\bf W}_k^j(s):0\le s\le t,\ k\in\mathcal I
			\bigr\}
			\qquad\forall\, t\in[0,T].
			\]
			Moreover, the limiting process \(\widetilde\omega\) is adapted to the limiting
			filtration \((\widetilde{\mathcal F}_t)_{t\in[0,T]}\). Therefore, by using
			\cite[Lemma~2.1]{DebusscheGlattHoltzTemam2011} in the form stated in
			\cite[Lemma~4.3]{BagnaraMaurelliXu2025}, together with
			Lemma~\ref{lem:coeff_convergence_conv}, we obtain, as \(j\to\infty\),
			\begin{equation}\label{eq:stochasticimit_conv}
				I_j^\varphi\longrightarrow I^\varphi
				\qquad
				\text{in probability in } C([0,T];\mathbb R).
		\end{equation}}
		By combining \eqref{eq:discrete_weak_form_conv}, \eqref{eq:nonlinearimit_conv},
		\eqref{eq:itoimit_conv}, and \eqref{eq:stochasticimit_conv}, we conclude that
		for every \(\varphi\in C^\infty(\T^2)\), for every \(t\in[0,T]\),
		\(\widetilde{\mathbb P}\)-a.s.,
		\begin{equation}\label{eq:weak_formimit_conv}
			\begin{aligned}
				\langle \widetilde\omega(t),\varphi\rangle
				&=
				\langle \omega_0,\varphi\rangle
				+
				\int_0^t
				\langle \widetilde\omega(s),\widetilde{\bf u}(s)\cdot \nabla\varphi\rangle\,\dd s
				+
				\frac{c_\infty}{2}
				\int_0^t
				\langle \widetilde\omega(s),\Delta\varphi\rangle\,\dd s
				\\
				&\quad
				+
				\sum_{k\in\mathcal I}
				\int_0^t
				\langle \widetilde\omega(s)\boldsymbol{\sigma}_k,\nabla\varphi\rangle\,
				\dd\widetilde{\bf W}_k(s) .
			\end{aligned}
		\end{equation}
		\subsection{Conclusion}

		\label{subsec:conclusion_conv}
		
		Equation \eqref{eq:weak_formimit_conv} for $\omega_0\in \mathbb{H}^{-1}(\mathbb{T}^2)$ shows that the limit process
		\(\widetilde\omega\) satisfies the weak formulation of the limiting equation on the
		new probability space. Since \(\widetilde\omega\in C\bigl([0,T];\mathbb{H}^{-5}(\T^2)\bigr)\cap
		\mathbb{L}^2\bigl([0,T];\mathbb{H}^{-\alpha}(\T^2)\bigr)\), and since
		\((\widetilde{\bf W}_k)_{k\in\mathcal I}\) is a family of independent
		\((\widetilde{\mathcal F}_t)\)-Brownian motions, the system
		\[
		\bigr(\widetilde\Omega,\widetilde{\mathcal F},(\widetilde{\mathcal F}_t)_{t\in[0,T]},
		\widetilde{\mathbb P},\widetilde\omega,(\widetilde{\mathbf W}_k)_{k\in\mathcal{I}}\bigr)
		\]
		is a weak martingale solution of the limiting equation. This completes the proof of Theorem~\ref{thm:convergence_fully_discrete_martingale}.
	\end{proof}
	\begin{remark}[On the compactness strategy]
		\label{remark6.12}
		The convergence argument above uses the continuous interpolant
		\(\omega_j^\sharp\) as the main compactness variable, rather than the left and
		right piecewise--constant interpolants. This choice is natural for the present
		scheme: \(\omega_j^\sharp\) agrees with the fully discrete approximation at the
		time nodes, is continuous in time by construction, and therefore already has the
		temporal regularity required in the limiting weak martingale formulation; see
		Definition~\ref{def:weak_martingale_solution_torus_conv}. The tightness result
		in Proposition~\ref{prop:tightness_main_conv} is obtained on the path space
		\(\mathfrak X\) of \eqref{eq:full_path_space_conv}, using the localized
		time-regularity estimates from Lemma~\ref{lem:localized_holder_conv} together
		with the stopping-time localization \eqref{eq:stopping_time_conv}.
		
		The discontinuous interpolants \(\underline{\omega}_j\) and
		\(\overline{\omega}_j\) are nevertheless essential for writing the discrete weak
		form. Their role in the compactness argument is auxiliary: by
		Lemmas~\ref{lem:sharp_under_close_conv}--\ref{lem:bar_under_close_conv}, they
		are asymptotically equivalent to \(\omega_j^\sharp\) in
		\(\mathbb{L}^2([0,T];\mathbb{H}^{-4}(\T^2))\) in probability. Hence, after applying
		Proposition~\ref{prop:skorokhod_jakubowski_conv}, the three limiting
		interpolants coincide and define a single adapted process
		\[
		\widetilde{\omega}^{\sharp}
		=
		\widetilde{\underline{\omega}}
		=
		\widetilde{\overline{\omega}}
		=:
		\widetilde{\omega}.
		\]
		This allows us to identify the limiting martingale problem in the natural
		continuous path space, while still passing to the limit through the discrete weak
		form.
		
		Thus the present proof differs from the discontinuous-path compactness approach
		of \cite{OndrejatProhlWalkington2023}. Instead of relying directly on global
		bounds in a Skorokhod-type space, we combine the continuous interpolant with
		localized stopping-time estimates and the negative-Sobolev coercivity structure
		of the Kraichnan noise. This is the point at which the specific structure of the
		fully discrete scheme \eqref{eq:scheme_keep} enters the convergence proof.
	\end{remark}

	\appendix
	\section{Technical results}
	\subsection{Noise estimate}
	\label{subsec:noise_estimate_conv}
	We state the following result as a useful ingredient in the proofs of
	Lemmas~\ref{lem:localized_holder_conv}--\ref{lem:bar_under_close_conv} and
	Lemma~\ref{lem:coeff_convergence_conv}. The proof follows the same main idea as
	that of \cite[Lemma~5.3]{CoghiMaurelli2026}.
	\begin{lemma}[Noise estimate in \(\mathbb{H}^{-4}(\T^2)\)]
		\label{lem:noise_estimate_conv}
		There exists \(C=C(\alpha)\), independent of \(j\), such that for every
		mean-zero \(f\in \mathbb{H}^{-1}(\T^2)\),
		\begin{equation}\label{eq:noise_estimate_conv}
			\sum_{k\in\mathcal I^{K_j}}
			\|\mathcal{P}_j[(\boldsymbol{\sigma}_k\cdot \nabla)f]\|_{\mathbb{H}^{-4}(\T^2)}^2
			\le C\|f\|_{\mathbb{H}^{-1}(\T^2)}^2.
		\end{equation}
		Consequently, for every fixed \(\varphi\in C^4(\T^2)\),
		\begin{equation}\label{eq:noise_scalar_estimate_conv}
			\sum_{k\in\mathcal I^{K_j}}
			\big|
			\langle f\boldsymbol{\sigma}_k,\nabla\varphi\rangle
			\big|^2
			\le \|\varphi\|_{\mathbb{H}^4}^2 \|f\|_{\mathbb{H}^{-1}(\T^2)}^2.
		\end{equation}
	\end{lemma}
	
	\begin{proof}
		Since \(\mathcal{P}_j\) is the Fourier projection onto a finite set of modes, it is an
		orthogonal projection on every Sobolev space \(\mathbb{H}^s(\T^2)\). In particular,
		\[
		\|\mathcal{P}_j g\|_{\mathbb{H}^{-4}(\T^2)}\le \|g\|_{\mathbb{H}^{-4}(\T^2)}
		\qquad\text{for every }g\in \mathbb{H}^{-4}(\T^2).
		\]
		Therefore it is enough to prove \eqref{eq:noise_estimate_conv} without \(\mathcal{P}_j\).
		
		\medskip
		\noindent\textbf{Step 1: Fourier estimate for one mode.}
		Recall \eqref{today07}. We write the noise basis as
		\[
		\{\boldsymbol{\sigma}_{{\bf m},c},\boldsymbol{\sigma}_{{\bf m},s}\}_{m\in\Z^2\setminus\{{\bf 0}\}},
		\]
		where the cosine and sine modes correspond to the same wave-vector \(\bf m\); see \eqref{eq:noise_ modes}. For a
		fixed \({\bf m}\neq{\bf 0}\), a direct Fourier computation as done in the proof of Theorem~\ref{prop:finite_dimensional_coercivity} gives, for every \({\bf r}\in\Z^2\),
		\[
		\sum_{\ell\in\{c,s\}}
		\left|
		\widehat{(\boldsymbol{\sigma}_{m,\ell}\cdot \nabla f)}({\bf r})
		\right|^2
		\le\,C\,
		q({\bf m})
		\Big(
		|{\bf e}_{\bf m}\cdot {{\bf r}-{\bf m}}|^2\,|\widehat f({\bf r}-{\bf m})|^2
		+
		|{\bf e}_{\bf m}\cdot {{\bf r}+{\bf m}}|^2\,|\widehat f({\bf r}+{\bf m})|^2
		\Big).
		\]
		Here we used that summing the
		cosine and sine contributions cancels the cross terms. Hence
		\[
		\sum_{k\in\mathcal I^{K_j}}
		\|(\boldsymbol{\sigma}_k\cdot \nabla f)\|_{\mathbb{H}^{-4}(\T^2)}^2
		\le\,C\,
		\Sigma_1+\Sigma_2,
		\]
		where
		\[
		\Sigma_1
		:=
		\sum_{{\bf r}\in\Z^2}
		\sum_{{\bf m}\neq{\bf 0}}
		\langle {\bf r}\rangle^{-8}
		q({\bf m})
		|{\bf e}_{\bf m}\cdot {{\bf r}-{\bf m}}|^2
		|\widehat f({\bf r}-{\bf m})|^2,
		\]
		and
		\[
		\Sigma_2
		:=
		\sum_{{\bf r}\in\Z^2}
		\sum_{{\bf m}\neq{\bf 0}}
		\langle {\bf r}\rangle^{-8}
		q({\bf m})
		|{\bf e}_{\bf m}\cdot {{\bf r}+{\bf m}}|^2
		|\widehat f({\bf r}+{\bf m})|^2.
		\]
		The two terms are identical after the change of variable \({\bf m}\mapsto -{\bf m}\), so it is
		enough to estimate \(\Sigma_1\). We set
		\[
		{\bf a}:={\bf r}-{\bf m}.
		\]
		Then
		\[
		\Sigma_1
		=
		\sum_{{\bf a}\in\Z^2}
		|\widehat f({\bf a})|^2
		\sum_{{\bf m}\neq{\bf 0}}
		q({\bf m})\,
		|{\bf e}_{\bf m}\cdot {\bf a}|^2\,
		\langle{{\bf a}+{\bf m}}\rangle^{-8}.
		\]
		Thus it remains to prove that
		\begin{equation}\label{eq:lattice_estimate_conv_clean}
			\sum_{{\bf m}\neq{\bf 0}}
			q({\bf m})\,
			|{\bf e}_{\bf m}\cdot {\bf a}|^2\,
			\langle{{\bf a}+{\bf m}}\rangle^{-8}
			\le\,C\,
			\langle {\bf a}\rangle^{-2}
			\qquad\text{for every }{\bf a}\in\Z^2.
		\end{equation}
		
		\medskip
		\noindent\textbf{Step 2: proof of the lattice estimate.}
		Fix \({\bf a}\in\Z^2\). We split the sum into
		\[
		A_{\bf a}:=\{{\bf m}\neq{\bf 0}:\ |{{\bf a}+{\bf m}}|\le |{\bf a}|/2\},
		\qquad
		B_{\bf a}:=\{{\bf m}\neq{\bf 0}:\ |{{\bf a}+{\bf m}}|> |{\bf a}|/2\}.
		\]
		
		\smallskip
		\noindent\textbf{Contribution of \(A_{\bf a}\).}
		If \({\bf m}\in A_{\bf a}\), then
		\[
		|{\bf m}|
		=
		|{({\bf a}+{\bf m})-{\bf a}}|
		\ge
		|{\bf a}|-|{{\bf a}+{\bf m}}|
		\ge
		\frac{|{\bf a}|}{2},
		\]
		and also
		\[
		|{\bf m}|
		\le
		|{\bf a}|+|{{\bf a}+{\bf m}}|
		\le
		\frac{3}{2}|{\bf a}|.
		\]
		Hence
		\[
		\langle {\bf m}\rangle\sim \langle {\bf a}\rangle,
		\qquad
		q({\bf m})\le\,C\, \langle {\bf a}\rangle^{-2-2\alpha}.
		\]
		Moreover, since \({\bf e}_{\bf m}\cdot {\bf m}=0\),
		\[
		{\bf e}_{\bf m}\cdot {\bf a}
		=
		{\bf e}_{\bf m}\cdot ({\bf a}+{\bf m})
		=
		{\bf e}_{\bf m}\cdot {{\bf a}+{\bf m}},
		\]
		and therefore
		\[
		|{\bf e}_{\bf m}\cdot {\bf a}|
		\le
		|{{\bf a}+{\bf m}}|.
		\]
		Thus
		\[
		\sum_{m\in A_{\bf a}}
		q({\bf m})\,
		|{\bf e}_{\bf m}\cdot {\bf a}|^2\,
		\langle{{\bf a}+{\bf m}}\rangle^{-8}
		\le\,C\,
		\langle {\bf a}\rangle^{-2-2\alpha}
		\sum_{m\in A_{\bf a}}
		|{{\bf a}+{\bf m}}|^2\langle{{\bf a}+{\bf m}}\rangle^{-8}.
		\]
		Now make the change of variable \({\bf n}={\bf a}+{\bf m}\). Since
		\[
		\sum_{n\in\Z^2}
		|{\bf n}|^2\langle {\bf n}\rangle^{-8}<\infty,
		\]
		we obtain
		\[
		\sum_{m\in A_{\bf a}}
		q({\bf m})\,
		|{\bf e}_{\bf m}\cdot {\bf a}|^2\,
		\langle{{\bf a}+{\bf m}}\rangle^{-8}
		\le\,C\,
		\langle {\bf a}\rangle^{-2-2\alpha}
		\le\,C\,
		\langle {\bf a}\rangle^{-2}.
		\]
		
		\smallskip
		\noindent\textbf{Contribution of \(B_{\bf a}\).}
		If \({\bf m}\in B_{\bf a}\), then \(|{{\bf a}+{\bf m}}|>|{\bf a}|/2\), so
		\[
		\langle{{\bf a}+{\bf m}}\rangle^{-8}
		\le\,C\,
		\langle {\bf a}\rangle^{-8}.
		\]
		Also,
		\[
		|{\bf e}_{\bf m}\cdot {\bf a}|\le |{\bf a}|.
		\]
		Therefore
		\[
		\sum_{m\in B_{\bf a}}
		q({\bf m})\,
		|{\bf e}_{\bf m}\cdot {\bf a}|^2\,
		\langle{{\bf a}+{\bf m}}\rangle^{-8}
		\le\,C\,
		|{\bf a}|^2\langle {\bf a}\rangle^{-8}
		\sum_{{\bf m}\neq{\bf 0}} q({\bf m}).
		\]
		Since \(\alpha>0\), the series \(\sum_{{\bf m}\neq{\bf 0}}q({\bf m})\) converges, and hence
		\[
		\sum_{m\in B_{\bf a}}
		q({\bf m})\,
		|{\bf e}_{\bf m}\cdot {\bf a}|^2\,
		\langle{{\bf a}+{\bf m}}\rangle^{-8}
		\le\,C\,
		\langle {\bf a}\rangle^{-6}
		\le\,C\,
		\langle {\bf a}\rangle^{-2}.
		\]
		By combining the estimates on \(A_{\bf a}\) and \(B_{\bf a}\), we obtain
		\eqref{eq:lattice_estimate_conv_clean}.
		
		\medskip
		\noindent\textbf{Step 3: conclusion of the \(\mathbb{H}^{-4}(\T^2)\) estimate.}
		Substituting \eqref{eq:lattice_estimate_conv_clean} into the expression for
		\(\Sigma_1\), we get
		\[
		\Sigma_1
		\le\,C\,
		\sum_{{\bf a}\in\Z^2}
		\langle {\bf a}\rangle^{-2}
		|\widehat f({\bf a})|^2
		=
		\|f\|_{\mathbb{H}^{-1}(\T^2)}^2.
		\]
		The same bound holds for \(\Sigma_2\). Therefore
		\begin{align}\label{todaynew}
			\sum_{k\in\mathcal I^{K_j}}
			\|(\boldsymbol{\sigma}_k\cdot \nabla f)\|_{\mathbb{H}^{-4}(\T^2)}^2
			\le\,C\,
			\|f\|_{\mathbb{H}^{-1}(\T^2)}^2.
		\end{align}
		Since \(\|\mathcal{P}_j g\|_{\mathbb{H}^{-4}(\T^2)}\le \|g\|_{\mathbb{H}^{-4}(\T^2)}\), this proves
		\eqref{eq:noise_estimate_conv}.
		
		\medskip
		\noindent\textbf{Step 4: scalar estimate.}
		Let \(\varphi\in C^4(\T^2)\) be fixed. Since \(\mathrm{div}\boldsymbol{\sigma}_k=0\), we have
		\[
		\langle f\boldsymbol{\sigma}_k,\nabla\varphi\rangle
		=
		-\langle (\boldsymbol{\sigma}_k\cdot \nabla)f,\varphi\rangle .
		\]
		Because 
		\[
		\big|
		\langle f\boldsymbol{\sigma}_k,\nabla\varphi\rangle
		\big|
		\le
		\|(\boldsymbol{\sigma}_k\cdot \nabla)f]\|_{\mathbb{H}^{-4}(\T^2)}\,
		\|\varphi\|_{\mathbb{H}^4}.
		\]
		Squaring and summing over \(k\in\mathcal I^{K_j}\), and then using
		\eqref{todaynew}, we obtain
		\[
		\sum_{k\in\mathcal I^{K_j}}
		\big|
		\langle f\boldsymbol{\sigma}_k,\nabla\varphi\rangle
		\big|^2
		\le
		\|\varphi\|_{\mathbb{H}^4}^2
		\sum_{k\in\mathcal I^{K_j}}
		\|(\boldsymbol{\sigma}_k\cdot \nabla)f\|_{\mathbb{H}^{-4}(\T^2)}^2\le\,C\,\|\varphi\|_{\mathbb{H}^4}^2\|f\|_{\mathbb{H}^{-1}(\T^2)}^2.
		\]
		This proves \eqref{eq:noise_scalar_estimate_conv}.
	\end{proof}
	
	\noindent
	We state the following result as a useful ingredient in the proof of
	Theorem~\ref{thm:strong_energy_keep}. The proof follows the same main idea as
	that of \cite[Lemma~5.1]{CoghiMaurelli2026}.
	\begin{lemma}[Discrete martingale coefficient estimate]
		\label{lem:discrete_martingale_coeff_keep}
		Let \(\omega\in\mathbb{V}_N\) be mean-zero, and set
		\[
		\psi:=(-\Delta)^{-1}\omega,
		\qquad
		{\bf u}:=\nabla^\perp\psi=\nabla^\perp(-\Delta)^{-1}\omega.
		\]
		Then there exists a constant \(C>0\), depending only on \(\alpha\), such that
		\begin{equation}\label{eq:discrete_martingale_coeff_keep}
			\sum_{k\in\mathcal I^K}
			\Big|
			\big\langle \PN[(\boldsymbol{\sigma}_k\cdot \nabla)\omega],\psi\big\rangle
			\Big|^2
			\le
			C\,\|\omega\|_{\mathbb{H}^{-1}(\T^2)}^2\,\|\omega\|_{\mathbb{H}^{-\alpha}(\T^2)}^2.
		\end{equation}
	\end{lemma}
	
	\begin{proof}
		Since \(\psi\in\mathcal V_N\) and \(\mathcal{P}_N\) is self-adjoint on \(\mathbb{L}^2(\T^2)\), we may
		drop the projection inside the pairing:
		\[
		\big\langle \PN[(\boldsymbol{\sigma}_k\cdot \nabla)\omega],\psi\big\rangle
		=
		\big\langle (\boldsymbol{\sigma}_k\cdot \nabla)\omega,\psi\big\rangle .
		\]
		Thus it is enough to estimate the right-hand side.
		
		\medskip
		\noindent\textbf{Step 1: rewrite the coefficient in velocity form.}
		Because \(\omega=-\Delta\psi\) and \(\mathrm{div}\boldsymbol{\sigma}_k=0\), we have
		\[
		\begin{aligned}
			\big\langle (\boldsymbol{\sigma}_k\cdot \nabla)\omega,\psi\big\rangle
			&=
			-\big\langle \omega,\boldsymbol{\sigma}_k\cdot \nabla\psi\big\rangle
			\\
			&=
			\big\langle \Delta\psi,\boldsymbol{\sigma}_k\cdot \nabla\psi\big\rangle
			\\
			&=
			-\sum_{i=1}^2\int_{\T^2}
			\partial_i\psi\,\partial_i(\boldsymbol{\sigma}_k\cdot \nabla\psi)\,\dd {\bf x} .
		\end{aligned}
		\]
		Expanding the derivative,
		\[
		\partial_i(\boldsymbol{\sigma}_k\cdot \nabla\psi)
		=
		\sum_{j=1}^2
		(\partial_i\boldsymbol{\sigma}_k^{j})\,\partial_j\psi
		+
		\sum_{j=1}^2
		\boldsymbol{\sigma}_k^{j}\,\partial_{ij}\psi .
		\]
		Hence
		\[
		\begin{aligned}
			\big\langle (\boldsymbol{\sigma}_k\cdot \nabla)\omega,\psi\big\rangle
			&=
			-\sum_{i,j=1}^2
			\int_{\T^2}
			\partial_i\psi\,(\partial_i\boldsymbol{\sigma}_k^{j})\,\partial_j\psi\,\dd {\bf x}
			\\
			&\quad
			-\sum_{i,j=1}^2
			\int_{\T^2}
			\partial_i\psi\,\boldsymbol{\sigma}_k^{j}\,\partial_{ij}\psi\,\dd {\bf x} .
		\end{aligned}
		\]
		The second term vanishes, since
		\[
		\sum_{i,j=1}^2
		\int_{\T^2}
		\partial_i\psi\,\boldsymbol{\sigma}_k^{j}\,\partial_{ij}\psi\,\dd {\bf x}
		=
		\frac12\int_{\T^2}\boldsymbol{\sigma}_k\cdot \nabla(|\nabla\psi|^2)\,\dd {\bf x}
		=0
		\]
		by periodicity and \(\mathrm{div}\boldsymbol{\sigma}_k=0\). Therefore
		\begin{equation}\label{eq:coeff_grad_form_keep_clean}
			\big\langle (\boldsymbol{\sigma}_k\cdot \nabla)\omega,\psi\big\rangle
			=
			-\sum_{i,j=1}^2
			\int_{\T^2}
			\partial_i\psi\,(\partial_i\boldsymbol{\sigma}_k^{j})\,\partial_j\psi\,\dd {\bf x} .
		\end{equation}
		Now write
		\[
		{\bf u}=\nabla^\perp\psi,
		\qquad
		J:=
		\begin{pmatrix}
			0 & 1\\
			-1 & 0
		\end{pmatrix},
		\]
		so that
		\[
		\nabla\psi=-J{\bf u}.
		\]
		Let \(A:=D\boldsymbol{\sigma}_k\). Since \(\mathrm{div}\boldsymbol{\sigma}_k=0\), we have \(\operatorname{tr}A=0\).
		For any \(2\times2\) trace-free matrix \(A\), one checks directly that
		\begin{equation}\label{eq:matrix_identity_keep_clean}
			-J^\top A J=A^\top .
		\end{equation}
		Using \eqref{eq:coeff_grad_form_keep_clean}, \(\nabla\psi=Ju\), and
		\eqref{eq:matrix_identity_keep_clean}, we obtain
		\[
		\begin{aligned}
			\big\langle (\boldsymbol{\sigma}_k\cdot \nabla)\omega,\psi\big\rangle
			&=
			-\int_{\T^2} (J{\bf u})\cdot  A(J{\bf u})\,\dd {\bf x}
			\\
			&=
			\int_{\T^2} {\bf u}\cdot  A^\top {\bf u}\,\dd {\bf x}
			\\
			&=
			\big\langle {\bf u},(D\boldsymbol{\sigma}_k)^\top {\bf u}\big\rangle .
		\end{aligned}
		\]
		Thus
		\begin{equation}\label{eq:coeff_velocity_form_keep_clean}
			\big\langle \PN[(\boldsymbol{\sigma}_k\cdot \nabla)\omega],\psi\big\rangle
			=
			\big\langle {\bf u},(D\boldsymbol{\sigma}_k)^\top u\big\rangle .
		\end{equation}
		
		\medskip
		\noindent\textbf{Step 2: estimate the square sum over \(k\).}
		Define
		\[
		S({\bf x}):={\bf u}({\bf x})\otimes {\bf u}({\bf x}),
		\qquad
		S_{ij}({\bf x})=u^i({\bf x})u^j({\bf x}).
		\]
		By \eqref{eq:coeff_velocity_form_keep_clean},
		\[
		\big\langle \PN[(\boldsymbol{\sigma}_k\cdot \nabla)\omega],\psi\big\rangle
		=
		\big\langle {\bf u},(D\boldsymbol{\sigma}_k)^\top {\bf u}\big\rangle
		=
		\sum_{i,j=1}^2
		\int_{\T^2} S_{ij}({\bf x})\,\partial_i\boldsymbol{\sigma}_k^{j}({\bf x})\,\dd {\bf x} .
		\]
		Hence
		\[
		\sum_{k\in\mathcal I^K}
		\Big|
		\big\langle \PN[(\boldsymbol{\sigma}_k\cdot \nabla)\omega],\psi\big\rangle
		\Big|^2
		\le
		\sum_{k\in\mathcal I}
		\big|\langle {\bf u},(D\boldsymbol{\sigma}_k)^\top {\bf u}\rangle\big|^2 .
		\]
		Expanding the square and using Fubini's theorem, we obtain
		\[
		\begin{aligned}
			\sum_{k\in\mathcal I}
			\big|\langle {\bf u},(D\boldsymbol{\sigma}_k)^\top {\bf u}\rangle\big|^2
			&=
			\sum_{k\in\mathcal I}
			\sum_{i,j,\ell, m=1}^2
			\iint_{\T^2\times\T^2}
			S_{ij}({\bf x})\,S_{\ell m}({\bf y})\,
			\partial_i\boldsymbol{\sigma}_k^{j}({\bf x})\,
			\partial_\ell\boldsymbol{\sigma}_k^{m}({\bf y})\,\dd {\bf x}\,\dd {\bf y}
			\\
			&=
			\sum_{i,j,\ell, m=1}^2
			\iint_{\T^2\times\T^2}
			S_{ij}({\bf x})\,S_{\ell m}({\bf y})\,
			\Gamma_{ij\ell m}({\bf x}-{\bf y})\,\dd {\bf x}\,\dd {\bf  y},
		\end{aligned}
		\]
		where
		\[
		Q_{jm}({\bf x}-{\bf y})
		:=
		\sum_{k\in\mathcal I}
		\boldsymbol{\sigma}_k^{j}({\bf x})\boldsymbol{\sigma}_k^{m}({\bf y}),
		\]
		and
		\[
		\Gamma_{ij\ell m}({\bf x}-{\bf y})
		:=
		\sum_{k\in\mathcal I}
		\partial_i\boldsymbol{\sigma}_k^{j}({\bf x})\,\partial_\ell\boldsymbol{\sigma}_k^{m}({\bf y}).
		\]
		Since \(Q_{jm}\) depends only on \({\bf x}-{\bf y}\), so does \(\Gamma_{ij\ell m}\), and
		\[
		\Gamma_{ij\ell m}({\bf x}-{\bf y})
		=
		\partial_i^x\partial_\ell^y Q_{jm}({\bf x}-{\bf y})
		=
		-\,\partial_{z_i}\partial_{z_\ell}Q_{jm}({\bf z}),
		\qquad {\bf z}={\bf x}-{\bf y}.
		\]
		Therefore
		\[
		\widehat{\Gamma_{ij\ell m}}({\bf n})
		=
		n^in^\ell\,\widehat{Q_{jm}}({\bf n}).
		\]
		Now
		\[
		\widehat{Q_{jm}}({\bf n})
		=
		q({\bf n})
		\left(
		\delta_{jm}-\frac{n^jn^m}{|{\bf n}|^2}
		\right),
		\qquad {\bf n}\neq{\bf 0},
		\]
		and hence
		\[
		\big|\widehat{\Gamma_{ij\ell m}}({\bf n})\big|
		\le
		C\,|{\bf n}|^2\,q({\bf n})
		\le
		C\,\langle {\bf n}\rangle^{-2\alpha},
		\]
		because
		\[
		q({\bf n})=\langle {\bf n}\rangle^{-(2+2\alpha)}
		\qquad\text{and}\qquad
		|{\bf n}|^2\le \langle {\bf n}\rangle^2.
		\]

		\noindent
		For each fixed \(i,j,\ell\), the bilinear form on the right is a convolution
		pairing. By Plancherel's theorem on \(\T^2\), there exists a constant
		\(C>0\), depending only on the Fourier normalization on \(\T^2\), such
		that
		\[
		\iint_{\T^2\times\T^2}
		S_{ij}({\bf x})\,S_{\ell m}({\bf y})\,
		\Gamma_{ij\ell m}({\bf x}-{\bf y})\,\dd {\bf x}\,\dd {\bf  y}
		=
		C
		\sum_{n\in\Z^2}
		\widehat{\Gamma_{ij\ell m}}({\bf n})\,
		\widehat{S_{ij}}({\bf n})\,
		\overline{\widehat{S_{\ell m}}({\bf n})}.
		\]
		Hence, using the multiplier bound
		\[
		\big|\widehat{\Gamma_{ij\ell m}}({\bf n})\big|
		\le
		C\,\langle {\bf n}\rangle^{-2\alpha},
		\]
		we obtain
		\[
		\begin{aligned}
			&\Bigg|
			\iint_{\T^2\times\T^2}
			S_{ij}({\bf x})\,S_{\ell m}({\bf y})\,
			\Gamma_{ij\ell m}({\bf x}-{\bf y})\,\dd {\bf x}\,\dd {\bf  y}
			\Bigg|
			\\
			&\qquad\le
			C
			\sum_{n\in\Z^2}
			\langle {\bf n}\rangle^{-2\alpha}
			\big|\widehat{S_{ij}}({\bf n})\big|
			\big|\widehat{S_{\ell m}}({\bf n})\big|.
		\end{aligned}
		\]
		By applying Young's inequality \(ab\le \frac12 a^2+\frac12 b^2\) gives
		\[
		\begin{aligned}
			&\Bigg|
			\iint_{\T^2\times\T^2}
			S_{ij}({\bf x})\,S_{\ell m}({\bf y})\,
			\Gamma_{ij\ell m}({\bf x}-{\bf y})\,\dd {\bf x}\,\dd {\bf  y}
			\Bigg|
			\\
			&\qquad\le
			C
			\sum_{n\in\Z^2}
			\langle {\bf n}\rangle^{-2\alpha}
			\Big(
			|\widehat{S_{ij}}({\bf n})|^2+|\widehat{S_{\ell m}}({\bf n})|^2
			\Big).
		\end{aligned}
		\]
		Now summing over \(i,j,\ell, m\in\{1,2\}\), and absorbing the finite number of index
		combinations into the constant, we conclude that
		\[
		\sum_{k\in\mathcal I}
		\big|\langle {\bf u},(D\boldsymbol{\sigma}_k)^\top {\bf u}\rangle\big|^2
		\le
		C
		\sum_{i,j=1}^2
		\sum_{n\in\Z^2}
		\langle {\bf n}\rangle^{-2\alpha}
		|\widehat{S_{ij}}({\bf n})|^2.
		\]
		Since \(S={\bf u}\otimes {\bf u}\), and since the \(\mathbb{H}^{-\alpha}\)-norm of a
		matrix-valued field is defined componentwise, the right-hand side is exactly
		comparable to
		\[
		\|{\bf u}\otimes {\bf u}\|_{\mathbb{H}^{-\alpha}(\T^2)}^2.
		\]
		Therefore
		\[
		\sum_{k\in\mathcal I}
		\big|\langle {\bf u},(D\boldsymbol{\sigma}_k)^\top {\bf u}\rangle\big|^2
		\le
		C\,\|{\bf u}\otimes {\bf u}\|_{\mathbb{H}^{-\alpha}(\T^2)}^2.
		\]
		\medskip
		\noindent\textbf{Step 3: estimate \({\bf u}\otimes {\bf u}\) in \(\mathbb{H}^{-\alpha}\).}
		Since \(0<\alpha<1\), Sobolev embedding on \(\T^2\) yields
		\[
		\mathbb{H}^{-\alpha}(\T^2)\hookleftarrow
		\mathbb{L}^{2/(1+\alpha)}(\T^2),
		\qquad
		\mathbb{H}^{1-\alpha}(\T^2)\hookrightarrow
		\mathbb{L}^{2/\alpha}(\T^2).
		\]
		Hence
		\[
		\|{\bf u}\otimes {\bf u}\|_{\mathbb{H}^{-\alpha}(\T^2)}
		\le\,C\,
		\|{\bf u}\otimes {\bf u}\|_{\mathbb{L}^{2/(1+\alpha)}}
		\le
		\|{\bf u}\|_{ \mathbb{L}^2(\T^2)}\,\|{\bf u}\|_{\mathbb{L}^{2/\alpha}}
		\le\,C\,
		\|{\bf u}\|_{ \mathbb{L}^2(\T^2)}\,\|{\bf u}\|_{\mathbb{H}^{1-\alpha}}.
		\]
		Since \({\bf u}=\nabla^\perp(-\Delta)^{-1}\omega\) and \(\omega\) has zero mean, the
		Biot--Savart operator is an isomorphism of order \(-1\), and therefore
		\[
		\|{\bf u}\|_{\mathbb{L}^2(\T^2)}
		\simeq
		\|\omega\|_{\mathbb{H}^{-1}(\T^2)},
		\qquad
		\|{\bf u}\|_{\mathbb{H}^{1-\alpha}(\T^2)}
		\simeq
		\|\omega\|_{\mathbb{H}^{-\alpha}(\T^2)}.
		\]
		Consequently,
		\[
		\|{\bf u}\otimes {\bf u}\|_{\mathbb{H}^{-\alpha}(\T^2)}
		\le\,C\,
		\|\omega\|_{\mathbb{H}^{-1}(\T^2)}
		\|\omega\|_{\mathbb{H}^{-\alpha}(\T^2)}.
		\]
		By combining the estimates from Steps 2 and 3 with
		\eqref{eq:coeff_velocity_form_keep_clean}, we conclude that
		\[
		\sum_{k\in\mathcal I^K}
		\Big|
		\big\langle \PN[(\boldsymbol{\sigma}_k\cdot \nabla)\omega],\psi\big\rangle
		\Big|^2
		\le
		C\,\|\omega\|_{\mathbb{H}^{-1}(\T^2)}^2\,\|\omega\|_{\mathbb{H}^{-\alpha}(\T^2)}^2.
		\]
		This proves \eqref{eq:discrete_martingale_coeff_keep}.
	\end{proof}
	\section*{Data availability statement}
	
	No datasets were generated or analyzed during the current study.


\begin{thebibliography}{99}
		
		\bibitem{BardosTadmor2015}
		C.~Bardos and E.~Tadmor,
		\newblock Stability and spectral convergence of Fourier method for nonlinear problems:
		on the shortcomings of the \(2/3\) de-aliasing method,
		\newblock {\em Numerische Mathematik} \textbf{129} (2015), 749--782.
		
		\bibitem{BagnaraMaurelliXu2025}
		M.~Bagnara, M.~Maurelli and F.~Xu.
		\textit{No blow-up by nonlinear It\^o noise for the Euler equations}.
		\textit{Electronic Journal of Probability} \textbf{30} (2025), 1--29.
		
		\bibitem{BrzezniakFlandoliMaurelli2016}
		Z.~Brze\'zniak, F.~Flandoli and M.~Maurelli,
		\newblock Existence and uniqueness for stochastic 2D Euler flows with bounded vorticity,
		\newblock {\em Archive for Rational Mechanics and Analysis} \textbf{221} (2016), 107--142.
		
		\bibitem{BrzezniakMaurelli2019}
		Z.~Brze\'zniak and M.~Maurelli.
		\textit{Existence for stochastic 2D Euler equations with positive \(H^{-1}\) vorticity}.
		\textit{Stochastics and Partial Differential Equations: Analysis and Computations},
		2026.
		
		\bibitem{CoghiMaurelli2026}
		M.~Coghi and M.~Maurelli,
		\newblock Existence and uniqueness by {K}raichnan noise for 2D {E}uler equations with unbounded vorticity,
		\newblock {\em Nonlinearity} \textbf{39} (2026), no.~5, 055003.
		
		\bibitem{DebusscheGlattHoltzTemam2011}
		A.~Debussche, N.~Glatt-Holtz and R.~Temam.
		\textit{Local martingale and pathwise solutions for an abstract fluids model}.
		\textit{Physica D: Nonlinear Phenomena} \textbf{240} (2011), no.~14--15,
		1123--1144.
		
		\bibitem{Delort1991}
		J.-M.~Delort,
		\newblock Existence de nappes de tourbillon en dimension deux,
		\newblock {\em Journal of the American Mathematical Society} \textbf{4} (1991), no.~3, 553--586.
		
		\bibitem{DiPernaMajda1987}
		R.~J. DiPerna and A.~J. Majda,
		\newblock Concentrations in regularizations for 2D incompressible flow,
		\newblock {\em Communications on Pure and Applied Mathematics} \textbf{40} (1987), no.~3, 301--345.
		
		\bibitem{DiPernaMajda1988}
		R.~J. DiPerna and A.~J. Majda,
		\newblock Reduced Hausdorff dimension and concentration-cancellation for two-dimensional incompressible flow,
		\newblock {\em Journal of the American Mathematical Society} \textbf{1} (1988), no.~1, 59--95.
		
		\bibitem{FlandoliGaleatiLuo2024}
		F.~Flandoli, L.~Galeati and D.~Luo,
		\newblock Quantitative convergence rates for scaling limit of {SPDE}s with transport noise,
		\newblock {\em Journal of Differential Equations} \textbf{394} (2024), 237--277.
		
		\bibitem{FlandoliHuang2023}
		F.~Flandoli and R.~Huang,
		\newblock Noise based on vortex structures in 2D and 3D,
		\newblock {\em Journal of Mathematical Physics} \textbf{64} (2023), no.~5, 053101.
		
		\bibitem{FlandoliLuo2020}
		F.~Flandoli and D.~Luo,
		\newblock Convergence of transport noise to {O}rnstein--{U}hlenbeck for 2D {E}uler equations under the enstrophy measure,
		\newblock {\em The Annals of Probability} \textbf{48} (2020), no.~1, 264--295.
		
		\bibitem{FlandoliPappalettera2021}
		F.~Flandoli and U.~Pappalettera,
		\newblock 2D {E}uler equations with {S}tratonovich transport noise as a large-scale stochastic model reduction,
		\newblock {\em Journal of Nonlinear Science} \textbf{31} (2021), Article~24.
		
		\bibitem{Galeati2020}
		L.~Galeati,
		\newblock On the convergence of stochastic transport equations to a deterministic parabolic one,
		\newblock {\em Stochastics and Partial Differential Equations: Analysis and Computations} \textbf{8} (2020), no.~4, 833--868.
		
		\bibitem{GaleatiGrottoMaurelli2024}
		L.~Galeati, F.~Grotto and M.~Maurelli.
		\textit{Anomalous regularization in Kraichnan's passive scalar model}.
		\textit{Probability Theory and Related Fields}, 2026.
		
		\bibitem{Jakubowski1997}
		A.~Jakubowski,
		\newblock The almost sure {S}korokhod representation for subsequences in nonmetric spaces,
		\newblock {\em Theory of Probability and Its Applications} \textbf{42} (1997), no.~1, 167--174.
		
		\bibitem{JiaoLuo2024}
		S. Jiao and D. Luo,
		\newblock On the pathwise uniqueness of stochastic 2D Euler equations with Kraichnan noise and \(L^p\)-data,
		\newblock J. Math. Fluid Mech. \textbf{27}, 38 (2025).
		
		
		\bibitem{KaratzasShreve1991}
		I.~Karatzas and S.~E. Shreve,
		\newblock {\em Brownian Motion and Stochastic Calculus},
		\newblock 2nd ed., Springer, New York, 1991.
		
		\bibitem{Kraichnan1968}
		R.~H. Kraichnan,
		\newblock Small-scale structure of a scalar field convected by turbulence,
		\newblock {\em Physics of Fluids} \textbf{11} (1968), no.~5, 945--953.
		
		\bibitem{Kraichnan1994}
		R.~H. Kraichnan,
		\newblock Anomalous scaling of a randomly advected passive scalar,
		\newblock {\em Physical Review Letters} \textbf{72} (1994), no.~7, 1016--1019.
		
		\bibitem{LanthalerMishra2015}
		S.~Lanthaler and S.~Mishra,
		\newblock Computation of measure-valued solutions for the incompressible Euler equations,
		\newblock {\em Mathematical Models and Methods in Applied Sciences} \textbf{25} (2015), no.~11, 2043--2088.
		
		\bibitem{LanthalerMishra2019}
		S.~Lanthaler and S.~Mishra,
		\newblock On the convergence of the spectral viscosity method for the two-dimensional incompressible Euler equations with rough initial data,
		\newblock {\em Found. Comput. Math.} \textbf{20}(5), 1309--1362, 2020.
		
		\bibitem{LanthalerMishraParesPulido2021}
		S.~Lanthaler, S.~Mishra and C.~Par\'es-Pulido,
		\newblock Statistical solutions of the incompressible Euler equations,
		\newblock {\em Mathematical Models and Methods in Applied Sciences} \textbf{31} (2021), no.~2, 223--292.
		
		\bibitem{LopesFilhoNussenzveigLopesTadmor2001}
		M.~C. Lopes Filho, H.~J. Nussenzveig Lopes and E.~Tadmor,
		\newblock Approximate solution of the incompressible Euler equations with no concentrations,
		\newblock {\em Annales de l'Institut Henri Poincar\'e C, Analyse Non Lin\'eaire} \textbf{17} (2000), no.~3, 371--412.
		
		\bibitem{Jakubowski1997}
		A.~Jakubowski.
		\textit{The almost sure Skorokhod representation for subsequences in nonmetric spaces}.
		\textit{Theory of Probability and its Applications} \textbf{42} (1997), no.~1,
		167--174.
		
		
		\bibitem{OndrejatProhlWalkington2023}
		M.~Ondrej\'at, A.~Prohl and N.~J. Walkington,
		\newblock Numerical approximation of nonlinear {SPDE}s,
		\newblock {\em Stochastics and Partial Differential Equations: Analysis and Computations} \textbf{11} (2023), 1553--1634.
		
		\bibitem{Schochet1995}
		S.~Schochet,
		\newblock The weak vorticity formulation of the 2D Euler equations and concentration-cancellation,
		\newblock {\em Communications in Partial Differential Equations} \textbf{20} (1995), no.~5--6, 1077--1104.
		
		\bibitem{SteinShakarchi2003}
		E.~M. Stein and R.~Shakarchi,
		\newblock {\em Fourier Analysis: An Introduction},
		\newblock Princeton Lectures in Analysis, Vol.~1, Princeton University Press, Princeton, 2003.
		
		\bibitem{Wiedemann2011}
		E.~Wiedemann,
		\newblock Existence of weak solutions for the incompressible {E}uler equations,
		\newblock {\em Annales de l'Institut Henri Poincar\'e C, Analyse Non Lin\'eaire} \textbf{28} (2011), no.~5, 727--730.
		
		\bibitem{Yudovich1963}
		V.~I. Yudovich,
		\newblock Non-stationary flow of an ideal incompressible fluid,
		\newblock {\em Zhurnal Vychislitel'noi Matematiki i Matematicheskoi Fiziki} \textbf{3} (1963), 1032--1066.
		
	\end{thebibliography}
\end{document}